\documentclass[english,svgnames,table]{article}
\usepackage{preamble}

\begin{document}
\title{Semiadditive Alternating Powers and Twisted Power Operations}
\begin{titlepage}
    \maketitle
    \begin{abstract}    
        We study a class of representations of symmetric groups in higher semiadditive categories. For these representations in $\ModEn$, the transchromatic character of Hopkins--Kuhn--Ravenel and Stapleton is recovered as a sequence of monoidal characters on suitable categorifications, giving an explicit algorithm for its computation, and relating it to the iterated monoidal character in $(\infty,\chrHeight)$-categories. 
        These representations also give rise to notions of alternating powers and power operations in semiadditive categories, extending the classical alternating powers and $\lambda$-operations in $\K$-theory. 
        We provide explicit computations in both the chromatic and higher categorical settings at low heights.

        \begin{figure}[h]
            \centering
            \includegraphics[width=0.36\linewidth]{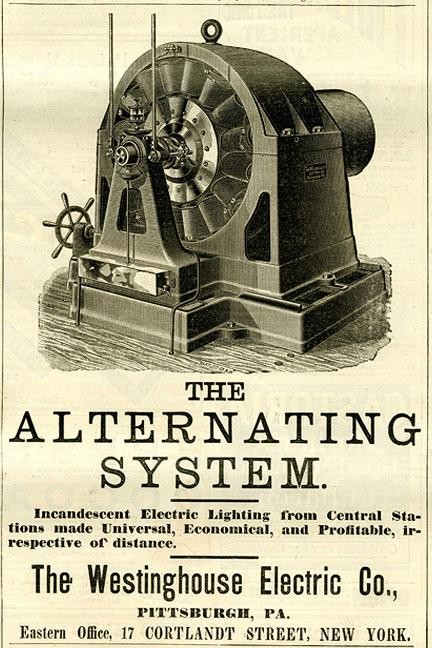}
            \caption*{Westinghouse Electric Company. 1888.\\\quotes{The Alternating Current System.}}
        \end{figure}
    \end{abstract}
\end{titlepage}

\tableofcontents
\newpage

%%%%%%%%%%%%%%%%%%%%%%%%%%%%%%%%%%%%%%%%%%%%%%%%%%%%%%%%%%%%%%%%%%%%%%%%%%%%%%%%
%%%%%%%%%%%%%%%%%%%%%%%%%%%%%%%%%%%%%%%%%%%%%%%%%%%%%%%%%%%%%%%%%%%%%%%%%%%%%%%%
\section{Introduction}

\label{sec:intro}
%%%%%%%%%%%%%%%%%%%%%%%%%%%%%%%%%%%%%%%%%%%%%%%%%%%%%%%%%%%%%%%%%%%%%%%%%%%%%%%%
%%%%%%%%%%%%%%%%%%%%%%%%%%%%%%%%%%%%%%%%%%%%%%%%%%%%%%%%%%%%%%%%%%%%%%%%%%%%%%%%
    Let $\field$ be a field of characteristic 0. Given a $\field$-vector space $V$, its $\degree$-th tensor power induces a $\Sm$-representation
    \begin{equation*}
        \Tm V \coloneqq V\om \in \Vect^{\B\Sm}.
    \end{equation*}
    The classical approach for studying representations is by  decomposing them to isotypic components corresponding to irreducible representations. The simplest irreducible representations are 1-dimensional, i.e.\ characters. 
    For a character $\hchar\colon \Sm \to \field\units$, we denote the isotypic component of $\Tm V$ corresponding to $\hchar$ by $\alt_{\hchar}V$. It can be computed as 
    \begin{equation*}
        \alt_{\hchar} V = (\Tm V \otimes \field[\overline{\hchar}])^{\Sm}
    \end{equation*}
    where $\field[\overline{\hchar}]$ is the 1-dimensional representation corresponding to the inverse character $\overline{\hchar}$. 

    The group $\Sm$ admits exactly 2 characters over $\field$, which identify for $\degree = 0,1$: the trivial and the sign representations
    \begin{equation*}
        \triv \colon \Sm \to e \into \field\units, \qquad
        \sgn \colon \Sm \xto{\sgn} \ZZ/2 \into \field\units
    \end{equation*}
    and their isotypic components correspond to the classical symmetric and alternating power respectively:
    \begin{equation*}
        \alt_{\triv} V = \Sym V, \qquad
        \alt_{\sgn} V = \alt V.
    \end{equation*}

    One can also decategorify this story, using algebraic $\K$-theory, in order to get power operations acting on the ring spectrum $\K(\field)$. That is, we get two maps of spaces
    \begin{equation*}
        \alt_{\triv} = \Sym \colon \Omega^{\infty}\K(\field) \to  \Omega^{\infty}\K(\field), \qquad 
        \alt_{\sgn} = \alt \colon \Omega^{\infty}\K(\field) \to \Omega^{\infty}\K(\field).
    \end{equation*}
    These correspond respectively to the usual $\degree$-th symmetric power map and to the $\lambda$ ring structure $\lambda_{\degree}$.
    
    We extend this story to the higher semiadditive setting, in which the group~$\Sm$ supports far more characters than in the classical linear case.

    %%%%%%%%%%%%%%%%%%%%%%%%%%%%%%%%%%%%%%%%%%%%%%%%%%%%%%%%%%%%%%%%%%%%%%%%%%%%%%%%
    \medskip\noindent
    \textbf{Higher semiadditive characters.}
    Classically, every character of a finite group factors through the roots of unity of the ground field, and it is therefore natural to assume that the field contains all roots of unity.
    The corresponding condition for an $\infty$-semiadditive, symmetric monoidal category~$\cC$\footnote{Throughout this article, we will use the term `category' to mean an `$(\infty,1)$-category.' Similarly we will use the term `$n$-category' to mean an `$(\infty,n)$-category.' We will use the term `space' to mean an '$(\infty,0)$-category' or an `$\infty$-groupoid'.} is \emph{$(\SS_{(p)},\chrHeight)$-orientability} in the sense of~\cite{BCSY-Fourier}.
    When $\cC$ is $(\SS_{(p)},\chrHeight)$-oriented, the unit object carries a map of connective spectra
    \begin{equation*}
        \Inprime \to \ounit_{\cC}\units,
    \end{equation*}
    where $\Inprime = \tau_{\ge 0}\Sigma^{\chrHeight} I_{\QQ_p/\Zp}$ is the connective cover of the $\chrHeight$-shifted $p$-typical Brown--Comenetz dual of the sphere.
    One may regard $\Inprime$ as the height-$\chrHeight$ analogue of the $p$-power roots of unity $\mu_{p^{\infty}}\simeq\QQ_p/\Zp$.

    For example, $\Vect$ is $(\SS_{(p)},0)$-orientable if and only if $\field$ admits all $p$-typical roots of unity.

    A fundamental theorem of Barthel--Carmeli--Schlank--Yanovski states that $\ModEn$ is $(\SS_{(p)}, \chrHeight)$-orientable.
    
    Throughout the paper we restrict attention to those characters that factor through this orientation map; that is, to $\EE_{1}$-maps
    \begin{equation*}
        \hchar \colon \Sm \to \Inprime.
    \end{equation*}
    We note that such characters are in bijection with $\Hom(\pi_{\chrHeight+1}(\redSS[\B\Sm]), \QQ_p/\ZZ_p)$, where $\redSS[-]$ denotes the reduced suspension spectrum functor. In particular, the set of such characters can be arbitrarily large.
    
    For each such character we define a twisted alternating power functor
    \begin{equation*}
        \begin{split}
            \alt_{\hchar} \colon \cC & \to \cC \\
            V & \mapsto (V\om \otimes \ounit_{\cC}[\overline{\hchar}])^{h\Sm}.
        \end{split}
    \end{equation*}

    Similarly, if $H \to \Sm$ is a homomorphism of finite groups and $\hchar \colon H \to \In$ is a character of $H$ we define
    \begin{equation*}
        \alt_{H,\hchar} V \coloneqq (V\om \otimes \ounit_{\cC}[\overline{\hchar}])^{hH}.
    \end{equation*}
    In height 0 this is the classical story, while in height 1 it generalizes the exterior powers of super linear categories as studied in~\cite{Ganter-Kapranov-2014-exterior-categories}\footnote{In the terminology of~\cite{Ganter-Kapranov-2014-exterior-categories}, the naive exterior powers correspond to the character we denote by $\sgn^{(1)}$ (\cref{def:sgn^(1)}) and the exterior powers correspond to the character induced by $\widehat{\eta^2} \in \widehat{\pi}_2(\SS)$ as in~\cref{characters-from-dual-stable-stems}.}.

    %%%%%%%%%%%%%%%%%%%%%%%%%%%%%%%%%%%%%%%%%%%%%%%%%%%%%%%%%%%%%%%%%%%%%%%%%%%%%%%%
    \medskip\noindent
    \textbf{Induced character formula.}
    The dimensions of vector spaces obtained as fixed points of representation, such as symmetric and alternating powers, have a well known formula: If $G$ is a finite group and $V$ is a finite dimensional $G$-representation with character $\chi_V$, then
    \begin{equation*}
        \dim (V^{G}) = \sum_{[g] \in G / \conj} \frac{\chi_{V}(g)}{|\cent_G(g)|}.
    \end{equation*}
    This formula was extended by~\cite{Ponto-Shulman-2014-induced-character, Carmeli-Cnossen-Ramzi-Yanovski-2022-characters} using the induced character formula, giving in the 1-semiadditive context for $V \in \cC\dblA[\BG]$ the formula
    \begin{equation*}
        \dim (V^{h G}) = \int_{\L\BG} \chi_{V}.
    \end{equation*}

    Working in $(\infty,\chrHeight)$-categories, one can iterate this formula to compute the iterated dimension using the iterated character: Assume $V \in \cC$ is fully dualizable, then for $1 \le t \le \chrHeight$
    \begin{equation*}
        \dim^t (V^{h G}) = \int_{\L^t \BG} \chi^t_{V}.
    \end{equation*}
    See~\cref{subsec:higher-cats} for a formal definition of full dualizability, iterated dimensions and iterated characters.

    The category $\ModEn$ has a transchromatic character theory, developed by Hopkins--Kuhn--Ravenel and extended by Stapleton and Lurie~\cite{Hopkins-Kuhn-Ravenel-2000-HKR, Stapleton-2013-HKR, Lurie-2019-Elliptic3}, which is seemingly completely different from, yet analogous to, the iterated character.
    
    The transchromatic character of height $\chrHeight-t$ is a map we denote
    \begin{equation*}
        \chi^{t,\HKR}_{(-)} \colon \En^{A} \to (C^{\chrHeight}_{\chrHeight-t})^{\Lp^t A}
    \end{equation*}
    for any $\pi$-finite space $A$, where $C^{\chrHeight}_{\chrHeight-t}$ is the $\Kn[\chrHeight-t]$-local splitting algebra of the $p$-divisible group on $\LKn[\chrHeight-t]\En$.\footnote{That is, it is the $\Kn[\chrHeight - t]$-localization of the splitting algebra constructed in~\cite{Stapleton-2013-HKR} and~\cite{Lurie-2019-Elliptic3}. Throughout this paper, we always consider the splitting algebra in its monochromatically localized form.} Let $V \in (\ModEn)\dblA[\BG]$. 
    Then, by the induced character formula and the shifted semiadditivity of the transchromatic character (\cite[Remark 7.4.8, Remark~7.4.9]{Lurie-2019-Elliptic3},~\cite[Theorem~A]{Ben-Moshe-2024-shifted-semiadditive}),
    \begin{equation}\label{eq:shifted-semiadditivity}\tag{$\spadesuit$}
        \dim(V^{hG}) = \int^{\En}_{\L\BG} \chi_V = \int^{C^{\chrHeight}_{\chrHeight-t}}_{\Lp^{t-1}\L\BG} \chi^{t-1,\HKR}_{\chi_V}.
    \end{equation}

    The monoidal iterated character and the transchromatic character behave similarly and seem to be connected, though they originate from entirely different constructions: the former through successive decategorification and the latter through the theory of formal groups. One of the main goals of this paper is to show that the transchromatic character can similarly be recovered via iterated decategorification, in the case of twisted alternating powers.

    %%%%%%%%%%%%%%%%%%%%%%%%%%%%%%%%%%%%%%%%%%%%%%%%%%%%%%%%%%%%%%%%%%%%%%%%%%%%%%%%
    \medskip\noindent
    \textbf{Categorification of maps.}
    The chromatic Nullstellensatz~\cite{Yuan-2024-redshift-En, Burklund-Schlank-Yuan-2022-Nullstellensatz} allows one to define a notion of \emph{decategorification} for local systems. Fix a map
    \begin{equation*}
        \LTnp\K((\ModEn)\dbl) \to \Enp(L)
    \end{equation*}
    for some algebraically closed field $L$, and define the decategorification map for a $\pi$-finite space $A$ as the composition
    \begin{equation*}
        \de \colon ((\ModEn)\dblspace)^A \to (\LTnp\K((\ModEn)\dbl))^A \to \Enp(L)^A.
    \end{equation*}
    This provides a means of comparing monoidal characters with transchromatic characters. Specifically, one hopes that the following diagram commutes up to an isomorphism of $\En(K)$:
    \begin{equation*}
        \begin{tikzcd}
            {((\ModEn)\dblspace)^A} && {\Enp(L)^{A}} \\
            {\En^{\L A}} && {C^{\chrHeight+1}_{\chrHeight}(L)^{\Lp A}} \\
            {\En^{\Lp A}} && {\En(K)^{\Lp A},}
            \arrow["\de", from=1-1, to=1-3]
            \arrow["\chi", from=1-1, to=2-1]
            \arrow["{\chi^{1,\HKR}}", from=1-3, to=2-3]
            \arrow[from=2-1, to=3-1]
            \arrow[from=2-3, to=3-3]
            \arrow[from=3-1, to=3-3]
        \end{tikzcd}
    \end{equation*}
    where the map $C^{\chrHeight+1}_{\chrHeight}(L) \to \En(K)$ exists again by the chromatic Nullstellensatz.

    In~\cite[Corollary~5.3.24, Lemma~5.3.25]{Keidar-Ragimov-2025-twisted}, it was shown that the diagram commutes for representations of the form
    \begin{equation*}
        \Tm V \otimes \En{}[\overline{\hchar}] \in (\ModEn)\dblA[\B\Sm],
    \end{equation*}
    where $\hchar \colon \Sm \to \In$ is a character. We extend this result to higher transchromatic characters.
    We start by noticing that the character
    \begin{equation*}
        f \coloneqq \chi_{\Tm V \otimes \En{}[\overline{\hchar}]} \in \En^{\L\B\Sm}
    \end{equation*}
    itself admits a categorification. That is, there exists a local system $V_1 \in (\ModEn[\chrHeight-1])^{\L\B\Sm}$ such that $\de(V_1) \simeq f$, after composing with some map $\En \to \En(L')$. 
    Extending the methods of~\cite{Keidar-Ragimov-2025-twisted}, we show that the above diagram also commutes for the categorification $V_1$, yielding
    \begin{equation*}
        \res{\chi_{V_1}}{\Lp \L \B\Sm} \simeq \chi^{1,\HKR}_{f}.
    \end{equation*}
    This process can be repeated for all heights, that is, the $(t-1)$-transchromatic character admits a categorification, whose character agrees with the $t$-transchromatic character. More precisely:

    \begin{alphThm}[\cref{thm:chi-tilde-t-is-HKR}]\label{alphthm}
        Let $\hchar \colon \Sm \to \In$ be a character and $V \in (\ModEn)\dbl$. 
        Then for every $1 \le t \le \chrHeight$, the map
        \begin{equation*}
            \chi^{t-1,\HKR}_{\chi_{\Tm V \otimes \En{}[\overline{\hchar}]}} \colon \Lp^{t-1}\L \B\Sm \to C^{\chrHeight}_{\chrHeight+1-t} \to \En[\chrHeight+1-t](K_{t-1})
        \end{equation*}
        admits a categorification
        \begin{equation*}
            V_t \in (\ModEn[\chrHeight - t])^{\Lp^{t-1} \L \B\Sm}.
        \end{equation*}
        Moreover, its character
        \begin{equation*}
            \res{\chi_{V_t}}{\Lp^t \L \B\Sm} \colon \Lp^t \L \B\Sm \to \En[\chrHeight - t] \to \En[\chrHeight - t](K_{t})
        \end{equation*}
        agrees with the transchromatic character
        \begin{equation*}
            \chi^{t,\HKR}_{\chi_{\Tm V \otimes \En{}[\overline{\hchar}]}} \colon \Lp^t \L \B\Sm \to C^{\chrHeight}_{\chrHeight - t} \to \En[\chrHeight - t](K_{t}).
        \end{equation*}
    \end{alphThm}

    \begin{remark}
        Set $V_0 \coloneqq \Tm V \otimes \En{}[\overline{\hchar}]$.
        One can also show that for any $1 \le t \le \chrHeight+1$
        \begin{equation*}
            \res{\chi^{t-1, \HKR}_{\chi_{\Tm V \otimes \En{}[\overline{\hchar}]}}}{\Lp^t \B\Sm} \simeq \chi^{t,\HKR}_{\de(\Tm V \otimes \En{}[\overline{\hchar}])} \simeq \res{\chi_{V_{t-1}}}{\Lp^t \B\Sm}.
        \end{equation*}
    \end{remark}

    This yields an explicit formula for the transchromatic character in these cases, which is often otherwise mysterious. Moreover, the result can be used to give an alternative proof of \labelcref{eq:shifted-semiadditivity} using categorical methods.

    A similar approach to computing transchromatic characters via categorification was used in~\cite{BMCSY-cardinality} to evaluate the cardinalities of $\pi$-finite $p$-local spaces in $\En$. Our results are a direct generalization.
    
    The proof of~\cref{alphthm} builds on higher semiadditive orientations, the induced character theorem, the chromatic Nullstellensatz, and the one-step comparison with transchromatic characters established in~\cite{Keidar-Ragimov-2025-twisted}. 
    To iterate this comparison, we must show that the characters obtained at each step themselves admit suitable categorifications. 
    We do this by introducing twisted permutation functions and proving that they admit categorifications whose characters are again twisted permutation functions. 
    The key calculation expresses the character of $\Tm V \in \cC^{\B(G\wr\Sm)}$ in terms of the character of $V \in \cC\dblA[\BG]$~(\cref{lem:character-of-wreath-product}).
    Once this closure property and the one-step comparison are established, the higher comparison follows by induction.

    % We prove~\cref{alphthm} by analyzing iterated characters of $\ounit[\overline{\hchar}]$ and of permutation representations. In the process, we express the character of $\Tm V \in \cC^{\B(G \wr \Sm)}$ in terms of the character of $V$, for any $V \in \cC\dblA[\BG]$ (\cref{lem:character-of-wreath-product}). 
    
    In order to compute the character of $\Tm V$, one must first describe $\L\B(G \wr \Sm)$, i.e.\ classify the conjugacy classes and centralizers in the wreath product $G \wr \Sm$. This classification is well known (see, e.g.,~\cite{Bernhardt-Niemeyer-Rober-Wollenhaupt-2022-wreath}). In~\cref{app:conjugacy-of-wreath}, we provide a short alternative proof of this group-theoretic description, using monoidal dimensions and semiadditive integrals.

    %%%%%%%%%%%%%%%%%%%%%%%%%%%%%%%%%%%%%%%%%%%%%%%%%%%%%%%%%%%%%%%%%%%%%%%%%%%%%%%%
    \medskip\noindent
    \textbf{Twisted power operations.}
    A similar story to that of twisted alternating powers is that of twisted power operations, which provide a decategorification of that construction. Let $\cC$ be a $(\SS_{(p)}, \chrHeight)$-oriented symmetric monoidal category. Given a homomorphism of finite groups $H \to \Sm$ and a map of spaces
    \begin{equation*}
        \hchar \colon \BH \to \Sigma \Inprime[\chrHeight-1] \to \ounit_{\cC}\units,
    \end{equation*}
    we define the \emph{$\hchar$-twisted power operation}
    \begin{equation*}
        \beta^{\degree}_{H, \hchar} \colon \ounit_{\cC} \xto{(-)^{\degree}} \ounit_{\cC}^{\BH} \xto{(-) \cdot \overline{\hchar}} \ounit_{\cC}^{\BH} \xto{\int_{\BH}} \ounit_{\cC}.
    \end{equation*}
    When $\hchar$ is trivial, this reduces to the usual semiadditive power operation (see, e.g.,~\cite[Definition~4.1.1]{CSY-teleambi}).
    
    \begin{example}
        Let $\cC = \Mod_{\KUp[2]}(\Sp^{\wedge}_2)$ be the category of 2-completed $\KUp[2]$-modules. Then the sign character recovers the $\lambda$-ring structure on $\KUp[2]$.
    \end{example}
    
    Working in $\Mod_{\cC}$, which has unit $\cC$, pointed maps
    \begin{equation*}
        \hchar \colon \B\Sm \to \Sigma \Inprime \to \cC\units
    \end{equation*}
    are in bijection with $(\SS_{(p)}, \chrHeight)$-orientation characters in $\cC$. In this setting, the twisted power operations coincide with the alternating powers.
    
    In the chromatic case $\cC = \ModEn$, this construction categorifies as well. Let $P \subseteq \Sm$ be a $p$-subgroup. Then for any $\hchar \colon \BP \to \Sigma \Inprime[\chrHeight-1]$ and $d \in \ZZ$,
    \begin{equation*}
        \beta^{\degree}_{P,\hchar}(d) = \begin{cases}
            \dim(\alt_{P,\hchar} \Enm^{\oplus d}), & d \ge 0, \\
            \dim(\alt_{P,\hchar} \Sigma \Enm^{\oplus (-d)}), & d < 0.
        \end{cases}
    \end{equation*}
    
    \begin{remark}
        Using this and previous results, one can compute the value of twisted power operations on integers. However, we do not know how to evaluate twisted power operations on general elements.
    \end{remark}

    %%%%%%%%%%%%%%%%%%%%%%%%%%%%%%%%%%%%%%%%%%%%%%%%%%%%%%%%%%%%%%%%%%%%%%%%%%%%%%%%
    \medskip\noindent
    \textbf{Computations for the character $\mathbf{sgn^{(1)}}$.}
    A theorem of Schur (\cite{Schur-1911-alternating-groups}) shows that for $\degree \ge 4$ there exists a unique non-trivial map (i.e.\ a character)
    \begin{equation*}
        \sgn^{(1)} \colon \B\Sm \to \Sigma^2 \QQ_2/\ZZ_2 = \Sigma^2 \mu_{\SS_{(2)}}^{(0)} \to \Sigma \mu_{\SS_{(2)}}^{(1)}.
    \end{equation*}
    This is the Schur multiplier class underlying classical spin representation theory. 

    We compute the dimensions of the alternating powers completely for this character in the cases 
    \begin{itemize}
        \item $\cC = \ModEn[1]$ at $p=2$, and
        \item  $\cC = \Mod_{\sVect}$ for $\field$ cyclotomically closed.
    \end{itemize}
    From these, one can deduce the evaluation of the corresponding height 2 twisted power operations on integers.

    We adopt the following notations:
    \begin{notation}
        We define:
        \begin{enumerate}
            \item for $\sigma \in \Sm$ and $k \ge 1$, $\Nk$ to be the number of $k$-cycles in the disjoint cycle decomposition of $\sigma$;
            \item $\mcal{O}(\degree)$ to be the set of conjugacy classes in $\Sm$ whose disjoint cycle decompositions contain no even cycles;
            \item for $n \ge 0$, $\mcal{P}_n(\degree)$ to be the set of $n$-tuples $(\lambda^1, \dots, \lambda^n)$ of (possibly empty) partitions into powers of $2$ with $|\lambda^1| + \cdots + |\lambda^n| = \degree$, which have an odd number of even parts in total;
            \item for $n \ge 0$, $\mcal{Q}_n(\degree) \subseteq \mcal{P}_n(\degree)$ to be the subset of tuples of partitions into distinct powers of $2$.
        \end{enumerate}
        We also write $\Sm^{(2)}$ for the set of $2$-power torsion permutations, and ${a \choose N} \coloneqq a(a-1)\cdots(a-N+1)/N!$ for $a \in \ZZ$ and $N \ge 0$.
    \end{notation}

    \begin{alphThm}[\cref{prop:categorical-height-1},~\cref{prop:dim-of-alternating-in-height-1-ge0},~\cref{prop:dim-of-alternating-in-height-1-<0}]
        Let $\degree \ge 4$. Then
        \begin{enumerate}
            \item In $\cC = \Mod_{\sVect}$, for $d \in \ZZ_{\ge 0}$
            \begin{equation*}
                \dim^2(\alt_{\sgn^{(1)}} (\sVect^{\oplus d})) = \sum_{[\sigma] \in \mcal{O}(\degree)} \bigsqcap_{k=1}^{\degree} {d+\Nk-1 \choose \Nk} + \sum_{\substack{[\sigma] \in \Sm / \conj \\ \sgn(\sigma) = -1}} \bigsqcap_{k=1}^{\degree} {d \choose \Nk}.
            \end{equation*}
            \item For any $V \in (\ModEn[1])\dbl$ with $d \coloneqq \dim(V)$,\footnote{It follows from~\cite[Proposition~10.11]{Mathew-2016-Galois} that $d$ is an integer.}
            \begin{equation*}
                \dim(\alt_{\sgn^{(1)}} V) = {d+\degree-1 \choose \degree} + \sum_{\substack{[\sigma] \in \Sm^{(2)} / \conj \\ \sgn(\sigma) = -1}} \bigsqcap_{k=1}^{\degree} {d \choose \Nk}.
            \end{equation*}
        \end{enumerate}
    \end{alphThm}

        Note that the formula in (2) is the one in (1) with the sums restricted to the classes of $2$-power torsion elements, as $\mcal{O}(\degree) \cap \Sm^{(2)}/\conj = \{[e]\}$. Counting the summands in (2) in terms of tuples of partitions into powers of $2$, we deduce:
    \begin{corollary}[\cref{cor:dim-height-1->=0},~\cref{cor:dim-height-1-<0}]
        Let $V \in (\ModEn[1])\dbl$ with $\dim(V) = d$, and let $\degree \ge 4$ have binary expansion $\degree = \sum_i b_i 2^i$. Then:
        \begin{enumerate}
            \item If $d \ge 0$, then
            \begin{equation*}
                \dim \alt_{\sgn^{(1)}} V = {d+\degree-1 \choose \degree} + |\mcal{Q}_d(\degree)|.
            \end{equation*}
            In particular, if $d = 1$, then $\dim \alt_{\sgn^{(1)}} V = 2$ if $|\{i > 0 \mid b_i = 1\}|$ is odd, and $\dim \alt_{\sgn^{(1)}} V = 1$ otherwise.
            \item If $d < 0$, then
            \begin{equation*}
                \dim \alt_{\sgn^{(1)}} V = (-1)^{\degree}\Big({-d \choose \degree} - |\mcal{P}_{-d}(\degree)|\Big).
            \end{equation*}
            In particular, if $d = -1$, then $\dim \alt_{\sgn^{(1)}} V = \frac{(-1)^{\degree+1}}{2}|\Sm^{(2)}/\conj|$.
        \end{enumerate}
    \end{corollary}

    %%%%%%%%%%%%%%%%%%%%%%%%%%%%%%%%%%%%%%%%%%%%%%%%%%%%%%%%%%%%%%%%%%%%%%%%%%%%%%%%
    \medskip\noindent
    \textbf{Computations in the categorical case.} 
    We also consider characters induced from functionals on the stable stems. A functional    
    \begin{equation*}
        \pichar \in \widehat{\pi}_{\chrHeight+1}(\SS_{(p)}) \coloneqq \Hom_{\Ab}(\pi_{\chrHeight+1}(\SS_{(p)}), \QQ_p/\ZZ_p) \simeq \pi_0 \Inprime[\chrHeight+1]
    \end{equation*}
    induces the map\footnote{As we view $\Inprime$ as an analogue of $\mu_{p^{\infty}}$, and usually consider it with a map to $\ounit_{\cC}\units$, we use multiplicative notation for the $\EE_{\infty}$-operation on its underlying space. In particular, the $\SS$-module structure is $\degree \bullet \pichar = \pichar^{\degree}$.}
    \begin{equation*}
        \pichar^{\degree} \colon \B\Sm \to \Inprime[\chrHeight+1],
    \end{equation*}
    choosing $\pichar^{\degree}$ with the symmetric action.
    This map lands in the connected component $\B \Omega_{\pichar^{\degree}} \Inprime[\chrHeight+1]$ which we identify with $\Sigma \Inprime$. We therefore get a family of character
    \begin{equation*}
        \hchar^{\degree}_{\pichar} \colon \B\Sm \to \Sigma \Inprime.
    \end{equation*}

    These characters were studied in~\cite{Keidar-Ragimov-2025-twisted}, where we showed that, in low heights and assuming additivity, dimensions of symmetric and alternating powers satisfy a generating function identity (\cite[Proposition~5.2.8, Lemma~5.2.11, Lemma~5.2.12]{Keidar-Ragimov-2025-twisted}).
    In this paper, we extend these results to the categorical setting, dropping the additivity assumption at the cost of replacing dimensions with iterated dimensions. 
    The resulting identity generalizes the classical relation between symmetric and exterior powers: the generating series of their dimensions multiply to $1$ when the exterior-power series is evaluated at $-t$.
    
    Here and throughout, statements involving $\cVectn$ assume the existence of the cyclotomic closures described in~\cref{def:cyc-closure}.
    \begin{proposition}[\cref{prop:generating-function}]
        Let $0 \le \chrHeight \le 2$. Let $\pichar \in \widehat{\pi}_{\chrHeight+1}(\SS_{(2)})$ be a generator, and let $V \in (\Mod_{\cVectn})\ndbl[\chrHeight+1]$ be such that the $\TT^i$-action on $\dim^i(V)$ is trivial for $1 \le i \le \chrHeight+1$. Then
        \begin{equation*}
            \left(\sum_{\degree} \dim^{\chrHeight+1}(\Sym V) t^{\degree}\right)
            \left(\sum_{\degree} \dim^{\chrHeight+1}(\alt_{\hchar_{\pichar}^{\degree}} V) (-t)^{\degree}\right) = 1.
        \end{equation*}
    \end{proposition}

    %%%%%%%%%%%%%%%%%%%%%%%%%%%%%%%%%%%%%%%%%%%%%%%%%%%%%%%%%%%%%%%%%%%%%%%%%%%%%%%%
    \subsection{Organization}
    %%%%%%%%%%%%%%%%%%%%%%%%%%%%%%%%%%%%%%%%%%%%%%%%%%%%%%%%%%%%%%%%%%%%%%%%%%%%%%%%

        \cref{sec:prelims} gathers background material. In~\cref{subsec:orientations}, we recall the notion of semiadditive orientations following~\cite{BCSY-Fourier}.
        \cref{subsec:integrals-p-complete} reviews a recent result of Li~\cite{Li-2025-cardinalities}. It gives a formula for computing semiadditive integrals over (free loop spaces of) classifying spaces of finite groups in $p$-complete stable categories, in terms of their restrictions to $p$-subgroups. In~\cref{subsec:higher-cats}, we recall some definitions from $(\infty,n)$-categories and extend the induced character formula to an iterated version.
        
        In~\cref{sec:categorical}, we introduce alternating powers and power operations in higher semiadditive categories and study their characters. We develop the general framework in~\cref{subsec:semiadditive-characters}.
        In~\cref{subsec:higher-cats-and-alternating} we give a formula for the characters of the permutation representation $\Tm V \in \cC^{\B(G\wr \Sm)}$, in terms of the character of $V \in \cC\dblA[\BG]$. Finally, in~\cref{subsec:categorical-alternating} we give some computations of the iterated dimension in low heights for characters induced by functionals on stable stems.
        
        \cref{subsec:chromatic-induction} turns to the chromatic setting. 
        In~\cref{subsec:categorification-of-functions} we show that (iterated) characters of permutation representations admit a categorification. In~\cref{subsec:HKR-and-characters}, we utilize this to recover the transchromatic character as an iterated monoidal character.
        
        \cref{sec:computations} works out concrete height~1 examples, computing the dimensions and iterated dimensions of alternating powers in $\ModEn[1]$ and in $\Mod_{\sVect}$.
        
        Finally,~\cref{app:conjugacy-of-wreath} provides a semiadditive proof of the classification of conjugacy classes and centralizers in wreath products $G \wr \Sm$, via monoidal dimensions.

    %%%%%%%%%%%%%%%%%%%%%%%%%%%%%%%%%%%%%%%%%%%%%%%%%%%%%%%%%%%%%%%%%%%%%%%%%%%%%%%%
    \subsection{Conventions}
    %%%%%%%%%%%%%%%%%%%%%%%%%%%%%%%%%%%%%%%%%%%%%%%%%%%%%%%%%%%%%%%%%%%%%%%%%%%%%%%%
        We use the following terminology and notations:
        \begin{enumerate}
            \item The category of spaces (or animae, or groupoids) is denoted by $\spc$.
            \item The category of spectra is denoted $\Sp$ and the full subcategory of connective spectra is denoted $\cnSp$.
            \item We denote by $\cC\core\subseteq \cC$ the maximal subgroupoid of a category $\cC$.
            \item We denote the space of morphisms between two objects $X,Y$ in a category $\cC$ by $\Map_{\cC}(X,Y)$ and omit $\cC$ when it is clear from context. If $\cC$ is stable we denote the mapping spectrum of $X,Y$ by $\hom_{\cC}(X,Y)$ or by $\hom(X,Y)$ if $\cC$ is clear from context.
            \item The category of presentable categories with colimit-preserving functors is denoted by $\PrL$. For $\cC \in \CAlg(\PrL)$ we denote its category of modules by $\Mod_{\cC} \coloneqq \Mod_{\cC}(\PrL)$.
            \item For a category $\cC$, an object $X \in \cC$ and a space $A \in \spc$, we denote the constant limit and colimit of $X$ along $A$ (if they exist) by $X^A$ and $X[A]$ respectively.
            \item We denote the free commutative monoid (in $\spc$) by $\MM \coloneqq (\Fin\core, \sqcup) \simeq \bigsqcup_{\degree}\B\Sm$.
            \item For a symmetric monoidal category $\cC$ we denote its full subcategory spanned by dualizable objects by $\cC\dbl$ and the maximal subgroupoid by $\cC\dblspace$. For a space $A$ we denote the category of dualizable $A$-local systems by $\cC\dblA \coloneqq (\cC\dbl)^A \simeq (\cC^A)\dbl$.
            \item We denote the loops functor of connective spectra by $\Omega \colon \cnSp \to \cnSp$ and distinguish it from the desuspension of (not necessarily connective) spectra which we denote $\Sigma^{-1} \colon \Sp \to \Sp$.
            \item We denote the connective free loop functor by $\L$. We use it both for the functor $\Map(\TT,-) \colon \spc \to \spc$ and for $\tau_{\ge 0}\hom(\SS[\TT], -) \colon \cnSp \to \cnSp$.
            \item For a height $\chrHeight$ and a prime $p$ we denote by $\Kn$, $\Tn$ the corresponding Morava $K$-theory and any telescope of height $\chrHeight$. For a formal group $\mbb{G}$ of height $\chrHeight$ over $\Fpbar$ and an algebraically-closed field $L$ we denote by $\En(L) = \En(L, \mbb{G})$ the corresponding Morava $E$-theory. When $L$ does not play an important role we omit it from the notation.
            \item Inspired by~\cite{Burklund-Schlank-Yuan-2022-Nullstellensatz}, we call maps to Nullstellensatzian objects \quotes{geometric points}.
            \item We denote by $C^{\chrHeight}_{\chrHeight-t}(L)$ the $\Kn[\chrHeight - t]$-localization of the splitting algebra of the $p$-divisible group on $L_{K(n-t)}\En(L)$, as constructed in~\cite{Stapleton-2013-HKR, Lurie-2019-Elliptic3}. For any $\pi$-finite space $A$, we denote the corresponding transchromatic character by
            \begin{equation*}
                \chi^{t,\HKR}_{(-)} \colon \En(L)^A \to C^{\chrHeight}_{\chrHeight-t}(L)^{\Lp^t A}.
            \end{equation*}
        \end{enumerate}

    %%%%%%%%%%%%%%%%%%%%%%%%%%%%%%%%%%%%%%%%%%%%%%%%%%%%%%%%%%%%%%%%%%%%%%%%%%%%%%%%
    \subsection{Acknowledgements}
    %%%%%%%%%%%%%%%%%%%%%%%%%%%%%%%%%%%%%%%%%%%%%%%%%%%%%%%%%%%%%%%%%%%%%%%%%%%%%%%%
        We are grateful to Nathaniel Stapleton and Tomer Schlank for inspiring conversations that helped initiate this project. We are especially thankful to Tomer Schlank for his guidance and many helpful discussions. We also thank Millie Rose for insightful conversations in the early stages, and Lior Yanovski for numerous valuable discussions throughout. We thank Achim Krause, Beckham Myers and Leor Neuhauser for comments on a previous draft.

        The first author thanks the University of Chicago for its hospitality during the writing of this paper. The second author acknowledges the Hebrew University, where much of the project was developed.

        %%%%%%%%%%%%%%%%%%%%%%%%%%%%%%%%%%%%%%%%%%%%%%%%%%%%%%%%%%%%%%%%%%%%%%%%%%%%%%%%
        \subsubsection*{AI disclosure}
        After the first draft of this paper was completed, Claude Opus 5.5 found a counting error in section 5, and helped resolving it.

%%%%%%%%%%%%%%%%%%%%%%%%%%%%%%%%%%%%%%%%%%%%%%%%%%%%%%%%%%%%%%%%%%%%%%%%%%%%%%%%
%%%%%%%%%%%%%%%%%%%%%%%%%%%%%%%%%%%%%%%%%%%%%%%%%%%%%%%%%%%%%%%%%%%%%%%%%%%%%%%%
\section{Preliminaries}
\label{sec:prelims}
%%%%%%%%%%%%%%%%%%%%%%%%%%%%%%%%%%%%%%%%%%%%%%%%%%%%%%%%%%%%%%%%%%%%%%%%%%%%%%%%
%%%%%%%%%%%%%%%%%%%%%%%%%%%%%%%%%%%%%%%%%%%%%%%%%%%%%%%%%%%%%%%%%%%%%%%%%%%%%%%%
In this section, we review and slightly extend a range of results of differing nature. In~\cref{subsec:orientations}, we revisit the theory of semiadditive orientations developed in~\cite{BCSY-Fourier}. In~\cref{subsec:integrals-p-complete}, we recall a theorem of Li~\cite{Li-2025-cardinalities} about computations of semiadditive integrals in $p$-complete 1-semiadditive stable categories. Finally, in~\cref{subsec:higher-cats}, we extend the induced character formula of~\cite{Ponto-Shulman-2014-induced-character, Carmeli-Cnossen-Ramzi-Yanovski-2022-characters} to $\chrHeight$-categories, using iterated characters.

    %%%%%%%%%%%%%%%%%%%%%%%%%%%%%%%%%%%%%%%%%%%%%%%%%%%%%%%%%%%%%%%%%%%%%%%%%%%%%%%%
    \subsection{Semiadditive orientations}
    \label{subsec:orientations}
    %%%%%%%%%%%%%%%%%%%%%%%%%%%%%%%%%%%%%%%%%%%%%%%%%%%%%%%%%%%%%%%%%%%%%%%%%%%%%%%%
        Following the terminology of~\cite{BCSY-Fourier}:
    
        \begin{definition}
            Let $\In \coloneqq \tau_{\ge 0} \Sigma^{\chrHeight} I_{\QQ / \ZZ}$ be the truncated and shifted $p$-typical Brown--Comenetz dual of the sphere.
        \end{definition}
        $\In$ is a higher analog of the group $\In[0] = \mu_{\infty}$ of roots of unity. Its $n$-th homotopy group corresponds to higher roots of unity (at all primes) as in~\cite{CSY-cyclotomic}. In particular, it defines a notion of height $\chrHeight$ Pontryagin duality:
    
        \begin{definition}
            Let $M$ be a connective spectrum. Its height $\chrHeight$ Pontryagin dual is defined to be 
            \begin{equation*}
                \ndual{M} \coloneqq \tau_{\ge 0}\hom(M, \In).
            \end{equation*}    
        \end{definition}

        \begin{remark}
            Note that the homotopy groups of the $\chrHeight$-th Pontryagin dual of $M$ are the Pontryagin dual of the homotopy groups of $M$.
            \begin{equation*}
                \pi_i \ndual{M} \simeq \widehat{\pi}_{\chrHeight-i}(M) \coloneqq \Hom_{\Ab}(\pi_{\chrHeight-i}(M), \QQ/\ZZ).
            \end{equation*}
            for $0 \le i \le \chrHeight$.
        \end{remark}
    
        \begin{definition}
            Let $\cC$ be an $\infty$-semiadditive, presentably symmetric monoidal category and let $R \in \CAlg(\cnSp)$. An $(R,\chrHeight)$-pre-orientation of $\cC$ is a map of connective spectra
            \begin{equation*}
                \omega \colon \ndual{R} \to \ounit_{\cC}\units.
            \end{equation*}
        \end{definition}
    
        \begin{definition}
            An $R$-module $M$ is $[0,\chrHeight]$-finite if it is connective, $\chrHeight$-truncated and all its homotopy groups are finite. We denote by $\Mod_R^{[0,\chrHeight]\hyphen\mrm{fin}}$ the full subcategory of $[0,\chrHeight]$-finite $R$-modules.
        \end{definition}
    
        \begin{proposition}[{\cite[Proposition~3.10]{BCSY-Fourier}}]
            The space of $(R,\chrHeight)$-pre-orientations of $\cC$ is isomorphic to the space of natural transformations 
            \begin{equation*}
                \ounit_{\cC}[-] \To \ounit_{\cC}^{\ndual{(-)}} \qin \Fun(\Mod_R^{[0,\chrHeight]\hyphen\mrm{fin}}, \CAlg(\cC)).
            \end{equation*}
        \end{proposition}
        For a pre-orientation $\omega$, we denote the corresponding natural transformation by $\mscr{F}_{\omega}$ and call it the associated Fourier transform.
    
        \begin{definition}\label{def:orientation}
            A pre-orientation $\omega$ is an orientation if the associated Fourier transform is a natural isomorphism. If the space of $(R,\chrHeight)$-orientations of $\cC$ is non-empty we say that $\cC$ is $(R,\chrHeight)$-orientable.
        \end{definition}
    
        For any $R$, $\ndual{R} \simeq \ndual{(\tau_{\le \chrHeight} R)}$. Therefore an $(R,\chrHeight)$-orientation is equivalent to a $(\tau_{\le \chrHeight} R, \chrHeight)$-orientation.
    
        \begin{example}
            A $(\ZZ/N, 0)$-orientation of $\Vect$ is a choice of a primitive $N$-th root of unity in $\field$. $\Vect$ is $(\SS,0)$-orientable if and only if $\field$ is cyclotomically-closed.
        \end{example}
    
        \begin{example}
            Let $\cC$ be $\infty$-semiadditive. A $(\ZZ/p^r,\chrHeight)$-orientation of $\cC$ is equivalent to a primitive $p^r$-th root of unity of height $\chrHeight$, as in~\cite{CSY-cyclotomic}.
        \end{example}
    
        \begin{example}[{\cite[Proposition~7.27, Corollary~4.21]{BCSY-Fourier}}]
            At the prime $p=2$, for any $R\in\CAlg(\SpTn)$, the category $\Mod_R(\SpTn)$ is $(\Fq{2},\chrHeight)$-orientable. In particular, $\SpTn$, $\SpKn = \Mod_{\SSKn}(\SpTn)$, $\ModEn = \Mod_{\En}(\SpTn)$ are $(\Fq{2},\chrHeight)$-orientable.
        \end{example}
        
        We interpret the $(\SS, \chrHeight)$-orientability of $\cC$  as expressing that the unit $\ounit_{\cC}$ is spherically cyclotomically closed. Similarly, $(\SS_{(p)}, \chrHeight)$-orientability is interpreted as admitting all $p$-typical spherical roots of unity. 
        
        \begin{example}[{\cite[Theorem~7.8]{BCSY-Fourier}}]
            $\ModEn$ is $(\SS_{(p)},\chrHeight)$-orientable.
        \end{example}   
        
        Orientation also behaves well under categorification:
        \begin{theorem}[{\cite[Corollary~5.16]{BCSY-Fourier}}]
            Assume $R$ is $\chrHeight$-truncated. Then $\cC$ is $(R,\chrHeight)$-orientable if and only if $\Mod_{\cC}$ is $(R, \chrHeight+1)$-orientable.
        \end{theorem}
    
        \begin{definition}\label{def:Vectn}
            Let $\field$ be a field of characteristic 0. 
            Define $\Vectn[1] \coloneqq \Vect \in \CAlg(\PrL)$ and $\Vectnp \coloneqq \Mod_{\Vectn} \in \CAlg(\PrL)$.
    
            To be more formal, we choose a sequence of inaccessible cardinals $\kappa_1 \subseteq \kappa_2 \subseteq \cdots$, and consider $\Vectn[1]$ as lying in $\CAlg(\PrL_{\kappa_1})$. We then define $\Vectnp \coloneqq \Mod_{\Vectn}(\PrL_{\kappa_{\chrHeight}}) \in \CAlg(\PrL_{\kappa_{\chrHeight+1}})$. We will usually omit this careful description and always assume we work in $\PrL_{\kappa_{\chrHeight}}$ for large enough $\chrHeight$.
        \end{definition}
    
        \begin{corollary}
            $\Vectnp$ is $(\Fq{2},\chrHeight)$-orientable.
        \end{corollary}
    
        \begin{definition}\label{def:cyc-closure}
            An $\chrHeight$-cyclotomic-closure of $\ounit_{\cC}$ is an algebra $\ounit_{\cC}\cyc \in \CAlg(\cC)$ such that:
            \begin{enumerate}
                \item $\Mod_{\ounit_{\cC}\cyc}(\cC)$ is $(\SS,\chrHeight)$-orientable.
                \item $\ounit_{\cC} \to \ounit_{\cC}\cyc$ is a pro-Galois extension in the sense of Rognes~\cite{Rognes-2008-Galois}.
                \item It is the minimal such extension: For any pro-Galois extension $\ounit_{\cC} \to R$ such that $\Mod_R(\cC)$ is $(\SS,\chrHeight)$-orientable, there exists a map $\ounit_{\cC}\cyc \to R$.
            \end{enumerate}
        \end{definition}
        We do not discuss the existence of cyclotomic closures in this paper but instead assume it, and specifically the existence of $\cVectn = (\Vectn)\cyc$. The existence of these cyclotomic closures will be addressed in a forthcoming paper by the first author.

    %%%%%%%%%%%%%%%%%%%%%%%%%%%%%%%%%%%%%%%%%%%%%%%%%%%%%%%%%%%%%%%%%%%%%%%%%%%%%%%%
    \subsection{Semiadditive integrals in \texorpdfstring{$p$}{p}-complete categories}
    \label{subsec:integrals-p-complete}
    %%%%%%%%%%%%%%%%%%%%%%%%%%%%%%%%%%%%%%%%%%%%%%%%%%%%%%%%%%%%%%%%%%%%%%%%%%%%%%%%

        We use a result of Li, which shows that in $p$-complete 1-semiadditive stable categories, integrals over finite classifying spaces of groups can be expressed as linear combinations of integrals over their $p$-subgroups. To establish an analogous statement for free loop spaces, we must appeal to specific arguments from the proof. We therefore briefly recall the key steps below:

        For a group $G$ we denote by $\mrm{A}(G)$ its Burnside ring, i.e.\ the group completion of the semi-ring generated by isomorphism classes of finite $G$-sets. The Segal conjecture (\cite{Lewis-May-McClure-1982-Segal,Carlsson-1984-Segal,Adams-Gunawardena-Miller-1992-Segal}), says that $\pi_0(\SS_{p}^{\BG}) \simeq \ApIG$, where $I_G \subseteq \mrm{A}(G)$ is the augmentation ideal.

        \begin{theorem}[{\cite[Theorem~3.1]{Yoshida-1983-Burnisde-ring},~\cite[Theorem~5.2.5]{Li-2025-cardinalities}}]\label{lem:p-completed-Burnside-generators}
            Let $G$ be a finite group. Let $\{P_1,\dots, P_r\}$ be the set of its $p$-Sylow subgroups. Then in $\Aploc{}[\frac{1}{[G/P_1]}]$
            \begin{equation*}
                1 = \sum_{i_1 < \cdots < i_k} (-1)^{k-1} \frac{\left|P_{i_1} \cap \cdots \cap P_{i_k}\right|}{|G|} [ G / (P_{i_1} \cap \cdots \cap P_{i_k})].
            \end{equation*}
        \end{theorem}

        Let $G$ be a group with a $p$-Sylow subgroup $P_1 \subseteq G$.
        The evaluation at the base point
        \begin{equation*}
            e^* \colon \pi_0(\SS_{p}^{\BG}) \to \pi_0(\SS_{p}) = \Zp
        \end{equation*}
        detects invertibility and sends $[G/H]$ to $[G:H]$ and in particular $[G/P_1]$ to a prime to $p$ integer.
        Therefore, after completing by the augmentation ideal, $[G/P_1]$ is invertible in $\ApIG \simeq \pi_0(\SS_{p}^{\BG})$.  
        
        Assume that $\cC$ is a $p$-complete stable symmetric monoidal category. 
        Then $\pi_0(\ounit_{\cC}^{\BG})$ is a module over $\pi_0(\SS_{p}^{\BG}) \simeq \ApIG$, and in particular over $\Aploc{}[\frac{1}{[G/P_1]}]$.

        \begin{corollary}[{\cite[Theorem 5.2.5]{Li-2025-cardinalities}}]\label{cor:integral-of-BG-is-linear-combination-along-p-subgroups}
            Let $\cC$ be a $p$-complete, $1$-semiadditive, symmetric monoidal stable category. Let $G$ be a finite group and $\{P_1,\dots, P_r\}$ be the set of its $p$-Sylow subgroups. Then, for  $f\in \ounit_{\cC}^{\BG}$,
            \begin{equation*}
                \int_{\BG} f = \sum_{i_1 < \cdots < i_k} (-1)^{k-1} \frac{\left|P_{i_1} \cap \cdots \cap P_{i_k}\right|}{|G|}\int_{\B(P_{i_1} \cap \cdots \cap P_{i_k})} f.
            \end{equation*} 
        \end{corollary}

        \begin{proof}
            Using the $\Aploc{}[\frac{1}{[G/P_1]}]$-module structure, by~\cref{lem:p-completed-Burnside-generators}, it is enough to verify that for a $p$-subgroup $P$ of $G$,
            \begin{equation}\label{eqn:G/P-is-restriction}
                \int_{\BG} ([G/P] \cdot f)= \int_{\BP} f \tag{$\star$}.
            \end{equation}
            This can be done by translating all the operations to the language of spans of spaces.
            We denote by $\Span(\cC; L, R)$ the span category of $\cC$, with left morphisms in $L \subseteq \cC^{[1]}$ and right morphisms in $R\subseteq \cC^{[1]}$ (see~\cite[Definition~5.7]{Barwick-2017-spectral-Mackey-functors}). In particular, we will denote by $\Span(\spc; \mfin, \mrm{all})$ the category whose objects are spaces and its morphisms are spans where the wrong way maps have $m$-finite fibers.
            As $\SS_{p}^{\BG}$ is $0$-semiadditive (i.e.\ a commutative monoid) it promotes to a functor
            \begin{equation*}
                (\SS_{p}^{\BG})^{(-)} \colon \Span(\spc, \mfin[0], \textrm{all}){\op} \to \Sp^{\wedge}_{p}.
            \end{equation*}
            The action of $[G / P]$ on $\SS_{p}^{\BG}$ is represented by the span
            \begin{equation*}
                \begin{tikzcd}
                    & {\BP} \\
                    {\BG} & {} & {\BG.}
                    \arrow[from=1-2, to=2-1]
                    \arrow[from=1-2, to=2-3]
                \end{tikzcd}
            \end{equation*}
            Recall that the integral over $\BG$ is defined by evaluating $\ounit_{\cC}^{(-)} \colon \Span(\spcpi[1])\op\to \cC$ against the span
            \begin{equation*}
                \begin{tikzcd}
                    & \BG \\
                    \pt && \BG.
                    \arrow[from=1-2, to=2-1]
                    \arrow[equals, from=1-2, to=2-3]
                \end{tikzcd}
            \end{equation*}
            Thus, the left hand side of \labelcref{eqn:G/P-is-restriction} is represented by the composition of spans
            \begin{equation*}
                \begin{tikzcd}
                    & {\BG} && {\BP} \\
                    \pt && {\BG} & {} & {\BG}
                    \arrow[from=1-2, to=2-1]
                    \arrow[equals, from=1-2, to=2-3]
                    \arrow[from=1-4, to=2-3]
                    \arrow[from=1-4, to=2-5]
                \end{tikzcd} \simeq
                \begin{tikzcd}
                    & {\BP} \\
                    \pt & {} & {\BG,}
                    \arrow[from=1-2, to=2-1]
                    \arrow[from=1-2, to=2-3]
                \end{tikzcd}
            \end{equation*}
            which in turn represents the integral $\int_{\BP}$ of the restriction to $\BP$.
        \end{proof}

        % In particular, power operations in $p$-typical categories are determined by the restiction to $p$-subgroups:
        % \begin{corollary}\label{cor:power-operations-p-typical}
        %     Let $\cC$ be symmetric monoidal, 1-semiadditive stable category. Let $G \to \Sm$ be a homomorphism and $\hchar \colon \B\Sm \to \ounit_{\cC}\units$ a pointed map. Then
        %     \begin{equation*}
        %         \beta^{\degree}_{G,\hchar} = \sum_{i_1 < \cdots < i_k} (-1)^{k-1}\frac{|P_{i_1} \cap \cdots \cap P_{i_k}|}{|G|} \beta^{\degree}_{P_{i_1} \cap \cdots \cap P_{i_k}, \hchar},
        %     \end{equation*}
        %     where $\{P_1,\dots,P_r\}$ is the set of $p$-Sylow subgroups of $G$.
        % \end{corollary}
        
        \begin{remark}\label{rmrk:Burnside-ring-of-1-types}
            Note that the Burnside ring is the Grothendieck group of the monoid of finite $G$-sets up to isomorphisms, and therefore it depends on $\BG$ rather than on $G$. The space of finite $G$-sets is equivalent to the space of $0$-finite maps of spaces $A \to \BG$. We therefore extend the notion of the Burnside ring of a finite group, to that of a 1-finite space. For a 1-finite space $X$ we define its Burnside ring $\mrm{A}(X)$ as the Grothendieck group of isomorphism classes of 0-finite maps to $X$. If $X = \BG$, then $\mrm{A}(\BG) = \mrm{A}(G)$ is the usual Burnside ring, and if $X = \bigsqcup_i \BG_i$, then $\mrm{A}(X) \simeq \bigsqcap_{i} \mrm{A}(G_i)$.
        \end{remark}

        \begin{lemma}\label{cor:integral-of-LBG-is-linear-combination-along-p-subgroups}
            Let $\cC$ be \emph{$p$-complete}, $1$-semiadditive, symmetric monoidal stable category. Let $G$ be a finite group and $\{P_1,\dots, P_r\}$ be the set of its $p$-Sylow subgroups. Then, for every $f \in \ounit_{\cC}^{\Lp\BG}$,
            % , for $t \ge 0$ and $f\in \ounit_{\cC}^{\Lp^t\BG}$
            \begin{equation*}
                \int_{\Lp \BG} f = \sum_{i_1 < \cdots < i_k} (-1)^{k-1} \frac{\left|P_{i_1} \cap \cdots \cap P_{i_k}\right|}{|G|}\int_{\L \B(P_{i_1} \cap \cdots \cap P_{i_k})} f.
            \end{equation*} 
        \end{lemma}
    
        \begin{proof}
            Choose a base-point of $\B\Zp$ and let $\ev \colon \Lp\BG \simeq \Map(\B\Zp, \BG) \to \BG$ be the evaluation map.
            Using the notations of~\cref{rmrk:Burnside-ring-of-1-types}, consider the two maps of rings
            \begin{align*}
                \Lp \colon \A(\BG) & \to \A(\Lp\BG) & \text{and} && \ev^* \colon \A(\BG) & \to \A(\Lp\BG) \\
                [X \to \BG] & \mapsto [\Lp X \to \Lp\BG] &&& [X \to \BG] & \mapsto [X \times_{\BG} \Lp\BG \to \Lp\BG] .
            \end{align*}
            Denote by $G^{(p)}$ the set of $p$-torsion elements in $G$. Then 
            \begin{equation*}
                \Lp\BG \simeq \bigsqcup_{[g] \in G^{(p)} / \conj} \B\cent_G(g),
            \end{equation*}
            and therefore
            \begin{equation*}
                \A(\Lp\BG) \simeq \bigsqcap_{[g] \in G^{(p)} / \conj} \A(\B\cent_G(g)).
            \end{equation*}
            Composing $\ev^*$ with the projection to $A(\B\cent_G(g))$ is the restriction map
            \begin{equation*}
                \A(G) \to \A(\cent_G(g)).
            \end{equation*}
            In particular, it sends $[G / P_1]$ to the $\cent_G(g)$-set $[G / P_1]$ which factors as
            \begin{equation*}
                \sum_{[x] \in P_1 \backslash G / \cent_G(g)} [\cent_G(g) / (x^{-1} P_1 x \cap \cent_G(g))].
            \end{equation*}
            The map $\Lp$ sends $[G/P_1]$ to
            \begin{equation*}
                [\L\BP_1 \simeq \bigsqcup_{[q] \in P_1 / \conj} \B\cent_{P_1}(q) \to \bigsqcup_{[g] \in G / \conj} \B\cent_G(g)],
            \end{equation*}
            therefore the image under the restriction to $\A(\B\cent_G(g))$ is
            \begin{equation*}
                \sum_{\substack{[q] \in P_1 / \conj \\ q \sim g}}[\B\cent_{P_1}(q) \to \B\cent_G(g)] = 
                \sum_{\substack{[x] \in P_1 \backslash G / \cent_G(g) \\ xgx^{-1} \in P_1}} [\B\cent_{P_1}(xgx^{-1}) \to \B\cent_G(g)].
            \end{equation*}
            Written as a sum of $\cent_G(g)$-sets, $[G/P_1] \in \A(G)$ is sent to
            \begin{equation*}
                \sum_{\substack{[x] \in P_1 \backslash G / \cent_G(g) \\ xgx^{-1} \in P_1}} [\cent_G(g) / x^{-1}\cent_{P_1}(xgx^{-1})x] =                 
                \sum_{\substack{[x] \in P_1 \backslash G / \cent_G(g) \\ xgx^{-1} \in P_1}} [\cent_G(g) / (x^{-1}P_1 x \cap \cent_G(g))].
            \end{equation*}
            In particular, the composition
            \begin{equation*}
                \A(G) = \A(\BG) \xto{\ev^* - \Lp} \A(\Lp\BG) \onto \A(\B\cent_G(g)) = \A(\cent_G(g)) \xto{e^*} \ZZ
            \end{equation*}
            sends $[G/P_1]$ to 
            \begin{equation*}
                \sum_{\substack{[x] \in P_1 \backslash G / \cent_G(g) \\ xgx^{-1} \not\in P_1}} [\cent_G(g) : (x^{-1}P_1 x \cap \cent_G(g))].
            \end{equation*}
            As $xgx^{-1} \not\in P_1$ and is central in $x \cent_G(g) x^{-1}$, the group $P_1 \cap x \cent_G(g) x^{-1}$ can not be a $p$-Sylow subgroup in $x \cent_G(g) x^{-1}$, therefore the above sum is 0 modulo $p$. 
            
            The map $\ev^*$, after $p$-completion and completion by augmentation ideals, has a geometric interpretation as $\ev^* \colon \SS_{p}^{\BG} \to \SS_p^{\Lp\BG}$. As $[G/P_1]$ is invertible in $\pi_0 (\SS_p^{\BG})$ it is sent to an invertible element in $\A(\cent_G(g))^{\wedge}_{p, I_{\cent_G(g)}}$ for any $[g] \in G^{(p)} / \conj$, and in particular, evaluating at a base point of $\B\cent_G(g)$, it is sent to a prime to $p$ integer in $\Zp$.

            By the above computation, the composition
            \begin{equation*}
                \A(G) \xto{\Lp} \A(\Lp\BG) \onto \A(\cent_G(g)) \to \A(\cent_G(g))^{\wedge}_{p, I_{\cent_G(g)}}
            \end{equation*}
            sends $[G/P_1]$ to an invertible element.
            
            Therefore, in the $\Aploc{}[\frac{1}{[G/P_1]}]$-algebra $\pi_0(\SS_p^{\Lp \BG}) \simeq \bigsqcap_{[g] \in G^{(p)} / \conj} \A(\cent_G(g))^{\wedge}_{p, I_{\cent_G(g)}}$,
            \begin{equation*}
                1 = \Lp(1) = \sum_{i_1 < \cdots < i_k} (-1)^{k-1} \frac{\left|P_{i_1} \cap \cdots \cap P_{i_k}\right|}{|G|}[\Lp \B(P_{i_1} \cap \cdots \cap P_{i_k}) \to \Lp \BG].
            \end{equation*}
            The proof now follows, as in the proof of~\cref{cor:integral-of-BG-is-linear-combination-along-p-subgroups}.
        \end{proof}

    %%%%%%%%%%%%%%%%%%%%%%%%%%%%%%%%%%%%%%%%%%%%%%%%%%%%%%%%%%%%%%%%%%%%%%%%%%%%%%%%
    \subsection{Higher categories and iterated characters}
    \label{subsec:higher-cats}
    %%%%%%%%%%%%%%%%%%%%%%%%%%%%%%%%%%%%%%%%%%%%%%%%%%%%%%%%%%%%%%%%%%%%%%%%%%%%%%%%

        We begin by recalling the theory of higher categories and the monoidal trace. We then use the 1-categorical induced character formula to derive an iterated version suited to the higher categorical setting. 

        Let $\cC$ be a symmetric monoidal $\chrHeight$-category. An object $V \in \cC$ is 1-dualizable (or just dualizable) if there exists $V\dual \in \cC$ along with evaluation and coevaluation maps
        \begin{equation*}
            \ev \colon V \otimes V\dual \to \ounit_{\cC}, \qquad \coev \colon \ounit_{\cC} \to V\dual \otimes V
        \end{equation*}
        satisfying the zig-zag identities, i.e.\ the maps
        \begin{equation*}
            V \xto{\id \otimes \coev} V \otimes V\dual \otimes V \xto{\ev \otimes \id} V, \qquad V\dual \xto{\coev \otimes \id} V\dual \otimes V \otimes V\dual \xto{\id \otimes \ev} V
        \end{equation*}
        are isomorphisms.
        
        \begin{definition}
            We denote the $(\chrHeight-1)$-category spanned by dualizable objects with (higher) morphisms which are left adjoint in $\cC$ by $\cC\ndbl[1]$.
            For $t > 1$ we define the $(\chrHeight-t)$-category
            \begin{equation*}
                \cC\tdbl \coloneqq (\cC\ndbl[(t-1)])\ndbl[1].
            \end{equation*}
            An object $V \in \cC\tdbl$ is called $t$-dualizable. Objects in $\cC\ndbl$ are called fully-dualizable.
        \end{definition}

        % \begin{remark}
        %     In our notations, an object is 2-dualizable if and only if both the evaluation and coevaluation maps are left-adjointable. This differs from the standard notion of 2-dualizability, which requires full adjointability, that is, that $\ev$ and $\coev$ both admit infinite chains of left and right adjoints. Both definitions yield the same notion of full adjointability.
        % \end{remark}

        \begin{definition}
            Given an $\chrHeight$-category $\cC$, we define its loops as
            \begin{equation*}
                \laxloops \cC \coloneqq \ounit_{\cC} \simeq \End_{\cC}(\ounit_{\cC}) \qin \Mon(\Catn[\chrHeight-1]).
            \end{equation*}
            When $\cC$ is symmetric monoidal, $\laxloops \cC$ is symmetric monoidal and the tensor product of $\ounit_{\cC}$ identifies with the composition of $\End_{\cC}(\ounit_{\cC})$.
        \end{definition}
        \begin{corollary}
            A morphism $f \colon \ounit_{\cC} \to \ounit_{\cC}$ is right $t$-adjointable\footnote{If and only if $f$ is left $t$-adjointable, and in this case it is ambidextrous.} in $\cC$ if and only if it is $t$-dualizable in $\laxloops \cC$.
        \end{corollary}
        In particular
        \begin{equation*}
            (\laxloops(\cC\tdbl))\core \simeq ((\laxloops \cC)\tdbl)\core.
        \end{equation*}

        \begin{definition}[{\cite[Definition 2.9]{Hoyois-Scherozke-Sibilla-2017-traces},~\cite[Definition~2.24]{Carmeli-Cnossen-Ramzi-Yanovski-2022-characters}}]
            Let $\cC$ be a symmetric monoidal $n$-category. There exists a symmetric monoidal functor\footnote{In the above references they actually construct the functor from a larger category (denoted by $\End(\cC)$ in~\cite{Hoyois-Scherozke-Sibilla-2017-traces} and by $\cC^{\mrm{trl}}$ in~\cite{Carmeli-Cnossen-Ramzi-Yanovski-2022-characters}) consisting of (not necessarily invertible) endomorphisms of dualizable objects and left adjoint morphisms with lax commuting squares between them. We restrict their construction to the simpler category consisting only of isomorphisms $\cC\ndbl[1]$.}
            \begin{equation*}
                \tr \colon \cC\ndblA[\TT]{1} \to \laxloops \cC
            \end{equation*}
            sending $f \colon V \to V$ to its trace
            \begin{equation*}
                \tr(f \mid V) \colon \ounit_{\cC} \xto{\coev} V \otimes V\dual \xto{f \otimes \id} V \otimes V\dual \isoto V\dual \otimes V \xto{\ev} \ounit_{\cC}.
            \end{equation*}
            The dimension of $V$ is defined to be
            \begin{equation*}
                \dim(V) \coloneqq \tr(\id_V \mid V).
            \end{equation*}
        \end{definition}

        \begin{lemma}[{\cite[Lemma~3.2.1]{Keidar-Ragimov-2025-twisted}}]\label{lem:trace-is-T-invariant}
            The trace map 
            \begin{equation*}
                \tr \colon (\cC\dblspace)^{\TT} \to (\laxloops \cC)\core
            \end{equation*}
            is $\TT$-invariant. Its $\TT$-fixed points induce the dimension map, remembering the $\TT$-action
            \begin{equation*}
                \dim \simeq \tr^{h\TT} \colon \cC\core \to ((\laxloops \cC)\core)^{\B\TT}.
            \end{equation*}
        \end{lemma}

        \begin{corollary}\label{cor:tr-of-k-dualizable}
            Assume that $V \in \cC$ is $t$-dualizable and $f \in \Aut(V)$. Then $\tr(f \mid V)$ is $(t-1)$-dualizable.
        \end{corollary}
        \begin{proof}
            The trace functor of automorphisms of $t$-dualizable objects factors through
            \begin{equation*}
                \cC\tdblA[\TT] \simeq (\cC\ndbl[(t-1)])\ndblA[\TT]{1} \xto{\tr} \laxloops (\cC\ndbl[(t-1)]).
            \end{equation*}
        \end{proof}

        We can therefore define inductively:
        \begin{definition}
            Let $\cC$ be a symmetric monoidal $\chrHeight$-category. Then by~\cref{cor:tr-of-k-dualizable}, one can define an iterated trace symmetric monoidal functor
            \begin{equation*}
                \tr^t \colon \cC\tdblA[\TT^t] \to \laxloops^t \cC
            \end{equation*}
            sending $t$ commuting automorphisms $f_1,\dots, f_t$ of $V$ to $\tr(f_1,\dots,f_t \mid V)$.
            We also define the $t$-fold dimension of $V$ as
            \begin{equation*}
                \dim^t(V) \coloneqq \tr^t(\id_V, \dots, \id_V \mid V).
            \end{equation*}
        \end{definition}

        \begin{example}
            $\dim^t(\ounit_{\cC}) \simeq \laxloops^{t+1} \cC$.
        \end{example}
        
        \begin{example}\label{exm:dim^4-of-picard}
            If $V$ is invertible, then it lies in the connective spectrum $\cC\units = \pic(\cC)$. Its $t$-fold dimension lies in its $t$-th loops 
            \begin{equation*}
                \dim^t(V) \in \Omega^t \cC \units = (\laxloops^t \cC)\units.
            \end{equation*}
            By~\cite[Proposition~3.20]{CSY-cyclotomic}, $\dim(V)$ is identified with $\eta \cdot V$, where $\eta \in \pi_1 \SS$ is the Hopf element. Therefore
            \begin{equation*}
                \dim^t V = \eta^t \cdot V \in \Omega^t \cC\units.
            \end{equation*}
            In particular, if $t \ge 4$, $\dim^t(V) = \laxloops^{t+1} \cC$, which is the unit of $\Omega^t \cC \units$.
        \end{example}

        By an iterated application of~\cref{lem:trace-is-T-invariant}, we deduce:

        \begin{corollary}\label{cor:tr^k-is-T^k-invariant}
            The $t$-fold trace map
            \begin{equation*}
                \tr^t \colon (\cC\tdblspace)^{\TT^t} \to (\laxloops^t \cC)\core
            \end{equation*}
            is $\TT^t$-invariant and its fixed point is the $t$-fold dimension, remembering a $\TT^t$-action
            \begin{equation*}
                \dim^t \simeq (\tr^t)^{h\TT^t} \colon \cC\tdblspace \to ((\laxloops^t \cC)\core)^{\B\TT^t}.
            \end{equation*}
        \end{corollary}

        % \begin{example}
        %     Let $A$ be a space such that $\ounit_{\cC}[A]$ is $k$-dualizable. Then
        %     \begin{equation*}
        %         \dim^k(\ounit_{\cC}[A]) \simeq (\laxloops^{k+1}\cC)[\L^k A].
        %     \end{equation*}
        % \end{example}

        % \begin{lemma}
        %     Assume that $\cC$ is presentable. Let $X \in \cC^A$. Then $X$ is $k$-dualizable if and only if the $\cC$-linear functor $\cC[A] \xto{X} \cC$ is $k$-adjointable.
        % \end{lemma}
        % \begin{proof}
        %     Following~\cite[Corollary 4.32]{Carmeli-Cnossen-Ramzi-Yanovski-2022-characters}, the proofs of~\cite[Lemma~4.12, Lemma~4.19]{Carmeli-Cnossen-Ramzi-Yanovski-2022-characters} carry for $k$-categories, showing it is enough to prove that $X$ is $k$-dualizable if and only if for any $a\in A$, $X \circ a_! \colon \cC \to \cC$ is $k$-adjointable. By~\cite[Corollary~4.12]{Carmeli-Cnossen-Ramzi-Yanovski-2022-characters}, $X \circ a_!$ corresponds to $a^* X \in \cC$. So it is enough to assume $A = \pt$.

        %     In this case, $\cC \simeq \End_{\Mod_{\cC}}(\cC)$ as symmetric monoidal categories, identifying the tensor product with the composition. Thus $X \in \cC$ is $k$-dualizable if and only if $\cC \xto{X} \cC$ is $k$-adjointable.
        % \end{proof}

        \begin{definition}
            Let $\cC$ be a symmetric monoidal $\chrHeight$-category. Let $A$ be a space and $V \in \cC\tdblA$ be an $A$-local system of $t$-dualizable objects. Define its $t$-fold character as the composition
            \begin{equation*}
                \chi^t_V \colon \L^t A \simeq A^{\TT^t} \xto{V^{\TT^t}} \cC\tdblA[\TT^t] \xto{\tr^t} \laxloops^t \cC.
            \end{equation*}
        \end{definition}

        \begin{remark}
            Let $V \in \cC\tdblA$ and let $(\gamma_1,\dots,\gamma_t) \colon \TT^t \to A$ be $t$ commuting loops based at $a \in A$. The character of $V$ at $(\gamma_1,\dots,\gamma_t)$ is the $t$-fold trace of the induced automorphisms 
            \begin{equation*}
                \tr^t(\gamma_1, \dots, \gamma_t \mid V_a).
            \end{equation*}
            Note in particular that the $t$-fold character is a $t$-fold iteration of the 1-character functor.
        \end{remark}
        
        \begin{corollary}\label{cor:characters-are-T-invariant}
            Let $\cC$ be a symmetric monoidal $\chrHeight$-category, $A \in \spc$ and $V\in \cC\tdblA$. Then the $t$-fold character of $V$
            \begin{equation*}
                \chi^t_V \colon \L^t A \to \laxloops^t \cC                
            \end{equation*}
            is $\TT^t$-invariant.
        \end{corollary}
    
        \begin{proof}
            The $t$-fold character map is the composition
            \begin{equation*}
                \chi^t_V \colon \L^t A = A^{\TT^t} \xto{V^{\TT^t}} \cC\tdblA[\TT^t] \xto{\tr^t} \laxloops^t \cC.
            \end{equation*}
            The first map is $\TT^t$-equivariant by definition and the second by~\cref{cor:tr^k-is-T^k-invariant}.            
        \end{proof}

        %%%%%%%%%%%%%%%%%%%%%%%%%%%%%%%%%%%%%%%%%%%%%%%%%%%%%%%%%%%%%%%%%%%%%%%%%%%%%%%%
        \subsubsection*{Induced character formula}
        Recall the 1-categorical induced character formula as developed in~\cite{Ponto-Shulman-2014-induced-character, Carmeli-Cnossen-Ramzi-Yanovski-2022-characters}:

        \begin{theorem}[{\cite[Theorem~5.20]{Carmeli-Cnossen-Ramzi-Yanovski-2022-characters}}]\label{thm:induced-character-formula}
            Let $\cC \in \CAlg(\PrL)$ and let $V \in \cC^A$ be a local system of dualizable objects. For a map of spaces $f \colon A \to B$ such that $f^* \colon \cC[B] \to \cC[A]$ is right adjointable, the induced character $\chi_{f_! V}$ is given by the composition
            \begin{equation*}
                \ounit_{\cC}[\L B] \xto{\tr(f^*)} \ounit_{\cC}[\L A] \xto{\chi_V} \ounit_{\cC}.
            \end{equation*}
            In particular, if $\cC$ is $m$-semiadditive and $f$ is a map of $m$-finite spaces, then
            \begin{equation*}
                \chi_{f_! V} = \int_{\L f} \chi_V.
            \end{equation*}
        \end{theorem}
        \begin{proof}
            The first part is exactly~\cite[Theorem~5.20]{Carmeli-Cnossen-Ramzi-Yanovski-2022-characters}. The second is by observing that in the $m$-semiadditive case,
            \begin{equation*}
                \tr (f^*) \colon \ounit_{\cC}[\L B] \to \ounit_{\cC}[\L A]
            \end{equation*}
            is equivalent to the composition
            \begin{equation*}
                (\L f)^* \colon \ounit_{\cC}[\L B] \isoto \ounit_{\cC}^{\L B} \xto{(\L f)^*} \ounit_{\cC}^{\L A} \isogets \ounit_{\cC}[\L A].
            \end{equation*}
            The above composition then agrees with the definition of the semiadditive integral.
        \end{proof}

        We iterate this theorem to give a higher analogue, assuming that the category satisfies some $\chrHeight$-presentability condition.

        \begin{definition}
            Let $\cC$ be an $\chrHeight$-category.
            \begin{enumerate}
                \item We say that $\cC$ is $\chrHeight$-presentable if $\laxloops^i \cC$ is presentable for $0 \le i < \chrHeight$ (\cite{stefanich-2020-presentable}).
                \item We say that $\cC$ is $(\chrHeight,m)$-semiadditive if $\laxloops^i \cC$ is $m$-semiadditive for $0 \le i < \chrHeight$.
            \end{enumerate}
        \end{definition}

        \begin{corollary}\label{cor:induced-k-character-formula}
            Let $\cC$ be a symmetric monoidal $\chrHeight$-presentable category.
            Let $V \in \cC\tdblA$ be a local system of $t$-dualizable objects. Then for a map of spaces $f \colon A \to B$ such that $f^* \colon \cC[B] \to \cC[A]$ is $t$-adjointable, the induced iterated character $\chi^t_{f_! V}$ is given by the composition
            \begin{equation*}
                \laxloops^{t}\cC[\L^t B] \xto{\tr^t(f^*)} \laxloops^{t}\cC[\L^t A] \xto{\chi^t_V} \laxloops^{t} \cC.
            \end{equation*}
            In particular, if $\cC$ is $(\chrHeight,m)$-semiadditive and $f$ is a map of $m$-finite spaces, then
            \begin{equation*}
                \chi^t_{f_! V} \simeq \int_{\L^t f} \chi^t_{V}.
            \end{equation*}
        \end{corollary}

%%%%%%%%%%%%%%%%%%%%%%%%%%%%%%%%%%%%%%%%%%%%%%%%%%%%%%%%%%%%%%%%%%%%%%%%%%%%%%%%
%%%%%%%%%%%%%%%%%%%%%%%%%%%%%%%%%%%%%%%%%%%%%%%%%%%%%%%%%%%%%%%%%%%%%%%%%%%%%%%%
\section{Characters and alternating powers in the higher categorical settings}
\label{sec:categorical}
%%%%%%%%%%%%%%%%%%%%%%%%%%%%%%%%%%%%%%%%%%%%%%%%%%%%%%%%%%%%%%%%%%%%%%%%%%%%%%%%
%%%%%%%%%%%%%%%%%%%%%%%%%%%%%%%%%%%%%%%%%%%%%%%%%%%%%%%%%%%%%%%%%%%%%%%%%%%%%%%%

    In~\cref{subsec:semiadditive-characters} we introduce alternating powers and power operations in the context of higher semiadditive categories. In the higher categorical framework, one can use the iterated induced character formula~\cref{cor:induced-k-character-formula} to compute their iterated dimension.

    In~\cite[Lemma~3.3.10]{Keidar-Ragimov-2025-twisted}, independently proven in~\cite[Lemma~4.7]{Ramzi-2023-seperability} (see also~\cref{lem:character-of-Tm}), we provide a formula for the character of a permutation representation. By this formula, to compute the iterated character, under the assumption that the $\TT$-action on the dimension is trivial, it suffices to extend the character computation from permutation representations $\Tm V \in \cC^{\B\Sm}$ to $\Tm (\BG)^*V \in \cC^{\B(G \wr \Sm)}$. We perform this extension in~\cref{subsec:higher-cats-and-alternating}, where we compute the character of $\Tm W \in \cC^{\B(G \wr \Sm)}$ in terms of the character of $W$, for any $W \in \cC\dblA[\BG]$.

    In~\cref{subsec:categorical-alternating}, we briefly examine the categorical case $\Mod_{\cVectn}$. In this setting, we derive a formula for the alternating power of the unit. We also provide generating functions that compute iterated dimensions of characters induced from functionals on stable stems in low heights, yielding results analogous to~\cite[\textsection~5.2]{Keidar-Ragimov-2025-twisted}.

    %%%%%%%%%%%%%%%%%%%%%%%%%%%%%%%%%%%%%%%%%%%%%%%%%%%%%%%%%%%%%%%%%%%%%%%%%%%%%%%%
    \subsection{Semiadditive characters and alternating powers}
    \label{subsec:semiadditive-characters}
    %%%%%%%%%%%%%%%%%%%%%%%%%%%%%%%%%%%%%%%%%%%%%%%%%%%%%%%%%%%%%%%%%%%%%%%%%%%%%%%%
        
        \begin{definition}
            Let $\cC$ be a symmetric monoidal category and let $H \in\Grp(\spc)$ be an $\EE_1$-group. 
            A character of $H$ in $\cC$ is a map of $\EE_1$-groups
            \begin{equation*}
                \hchar \colon H \to \ounit_{\cC}\units.
            \end{equation*}
            Equivalently a character is a map of pointed connected spaces $\BH\to \B\ounit_{\cC}\units \simeq \B\Aut(\ounit_{\cC})$ and thus gives an $H$-action on $\ounit_{\cC}$. We denote the unit with this action by $\ounit_{\cC}[\hchar]\in\cC^{\BH}$.
        \end{definition}

        \begin{definition}
            Let $H \in \Grp(\spc)$ and $\hchar \colon H \to \ounit_{\cC}\units$ be a character. Define the inverse character $\overline{\hchar} \colon H \to \ounit_{\cC}\units$ as the composition
            \begin{equation*}
                \overline{\hchar} \colon H \xto{\hchar} \ounit_{\cC}\units \xto{(-)^{-1}} \ounit_{\cC}\units.
            \end{equation*}
        \end{definition}
        
        \begin{definition}
            Let $\cC\in\CAlg(\PrL)$, $H\in\Grp(\spc)$ and $\hchar \colon H \to \ounit_{\cC}\units$ a character. For $V\in\cC^{\BH}$ we define its $\hchar$-isotypic component by
            \begin{equation*}
                V_{\hchar} \coloneqq (V\otimes \ounit_{\cC}[\overline{\hchar}])^{h H}.
            \end{equation*}
        \end{definition}

        \begin{definition}\label{def:alt-power}
            Let $\cC\in\CAlg(\PrL)$, $H$ a group with a homomorphism $H \to \Sm$ and a character $\hchar \colon H \to \ounit_{\cC}\units$. 
            For $V\in \cC$ define its $\hchar$-twisted $\degree$-th alternating power $\alt_{H,\hchar} V \coloneqq (\Tm V)_{\hchar}$ as the $\hchar$-isotypic component of $\Tm V\in \cC^{\BH}$. That is:
            \begin{equation*}
                \alt_{H, \hchar} V = (\Tm V \otimes \ounit_{\cC}[\overline{\hchar}])^{h H}.
            \end{equation*}
        \end{definition}

        \begin{notation}
            Our main interest will be in the case $H = \Sm$ with the identity morphism. In this case we drop $H$ from the notation
            \begin{equation*}
                \alt_{\hchar} V = (\Tm V \otimes \ounit_{\cC}[\overline{\hchar}])^{h\Sm}.
            \end{equation*}
        \end{notation}

        \begin{example}
            For any $\cC$ there exists a trivial character $\hchar \equiv 1$. Its corresponding alternating powers are identified with the symmetric powers.
        \end{example}
        
        \begin{example}
        \label{exm:classical-characters}
            Let $\field$ be a field of characteristic different than 2. Then $\Sm$ has exactly two characters in $\Vect$, which coincide when $\degree = 0,1$: the trivial and the sign characters. The corresponding alternating powers are the symmetric powers and the classical alternating powers respectively.
        \end{example}

        Both the trivial and the sign representations factor through $\ZZ/2$. More generally, if $H$ is a finite group, then any character of $H$ in $\Vect$ factors through $\mu_{\infty}(\field)$ --- the group of roots of unity in $\field$. We extend this idea to the higher semiadditive setting, replacing roots of unity with their higher analogue, as captured by semiadditive orientations (\cref{subsec:orientations}).

        \begin{definition}\label{def:characters-from-orientation}
            Let $H$ be a group, $R$ a connective ring spectrum and $\cC$ an $(R,\chrHeight)$-oriented category. 
            A character is said to be an $(R,\chrHeight)$-orientation character if it factors through the orientation map $\ndual{R} \to \ounit_{\cC}\units$.
        \end{definition}

        % Replacing $\cC$ with $\cC^{\BG}$, we get the following simple generalization:
        % \begin{definition}
        %     Let $\cC \in \CAlg(\PrL)$, $G, H$ groups with a homomorphism $H \to \Sm$. Let $\hchar \in \Ch(H; \cC)$. For $V \in \cC^{\BG}$ we define its $\hchar$-twisted $\degree$-th alternating power $\alt_{H,\hchar} \in \cC^{\BG}$ as
        %     \begin{equation*}
        %         \alt_{H, \hchar} V = (\Tm V \otimes \ounit_{\cC}[\overline{\hchar}])^{h H}.
        %     \end{equation*}
        % \end{definition}

        % \begin{remark}
        %     Note that for $V \in \cC^{\BG}$, $\Tm V$ is an $H$-representation in $\cC^{\BG}$, but in $\cC$ it admits naturally an action of $G \wr H \coloneqq G^{\degree} \rtimes H$. Taking $H$-fixed points, it admits a residual action by the Weil group
        %     \begin{equation*}
        %         \mrm{N}_{G \wr H}(H) / H \supseteq (G \times H) / H \simeq H.
        %     \end{equation*}
        %     When $H = \Sm$, the inequality becomes an equality.
        % \end{remark}

        % \begin{corollary}
        %     Assume that $F \colon \cC \to \cD \in \CAlg(\PrL)$ commutes with $\BH$-shaped limits. Let $\hchar \in \Ch(H; \cC)$. Then for any $V \in \cC$,
        %     \begin{equation*}
        %         F \alt_{H, \hchar} V = \alt_{H, F \hchar} FV.
        %     \end{equation*}
        % \end{corollary}

        %%%%%%%%%%%%%%%%%%%%%%%%%%%%%%%%%%%%%%%%%%%%%%%%%%%%%%%%%%%%%%%%%%%%%%%%%%%%%%%%
        \subsubsection*{Twisted power operations}
        Let $\hchar \colon \BH \to \B\ounit_{\cC}\units$ be a character. The $\hchar$-twisted alternating power functor
        \begin{equation*}
            \alt_{\hchar} \colon \cC \to \cC
        \end{equation*}
        can be thought of as a (twisted) power operation in $\Mod_{\cC}$. Under semiadditivity assumptions, it introduces operations on the ring $\K(\cC\dbl)$. 

        % $\Mod_{\cC}$ is $\infty$-semiadditive and the semiadditive integral is identified with the colimit in the corresponding functor categories.
        The character $\hchar$ can also be viewed as a map of pointed spaces
        \begin{equation*}
            \hchar \colon \BH \to \cC\units.
        \end{equation*}
        Observing that $\cC$ is the unit object in $\Mod_{\cC}$ motivates the following definition:
                
        \begin{definition}\label{cons:power-operation}
            Let $\cC$ be a 1-semiadditive, symmetric monoidal category. Let $H$ be a finite group with a homomorphism $H \to \Sm$ and let $\hchar \colon \BH \to \ounit_{\cC}\units$ be a pointed map.
            Define the $\hchar$-twisted power operation as the composition
            \begin{equation*}
                \beta_{H,\hchar}^{\degree} \colon \underlying \ounit_{\cC} \xto{(-)^{\degree}} (\underlying \ounit_{\cC})^{\BH} \xto{-\cdot \overline{\hchar}} (\underlying \ounit_{\cC})^{\BH} \xto{\int_{\BH}} \underlying \ounit_{\cC}.
            \end{equation*}
        \end{definition}

        \begin{notation}
            In the case where $H = \Sm$ with the identity morphism we drop the group from the notation
            \begin{equation*}
                \beta_{\hchar}^{\degree} \colon \underlying \ounit_{\cC} \xto{(-)^{\degree}} (\underlying \ounit_{\cC})^{\B\Sm} \xto{-\cdot \overline{\hchar}} (\underlying \ounit_{\cC})^{\B\Sm} \xto{\int_{\B\Sm}} \underlying \ounit_{\cC}.
            \end{equation*}
        \end{notation}
    
        \begin{example}
            When $\hchar \colon \B\Sm \to \ounit_{\cC}\units$ is the trivial map, the $\hchar$-twisted power operation is the usual semiadditive power operation (see e.g.~\cite[Definition~4.1.1]{CSY-teleambi}).
        \end{example}

        \begin{example}
            Let $\cC = \Mod_{\KUp[2]}(\SpKn[1]) = \Mod_{\KUp[2]}(\Sp^{\wedge}_2)$.\footnote{We avoid writing $\En[1] = \KUp[2]$ as our $\En[1]$ is Galois closed. That is, $\En[1]$ is the tensor of $\KUp[2]$ with the spherical Witt vectors of $\Fpbar$.} Let
            \begin{equation*}
                \hchar \colon \B\Sm \xto{\sgn} \Sigma \ZZ/2 \to (\KUp[2])\units.
            \end{equation*}
            Then the $\hchar$-twisted power operation $\beta^{\degree}_{\hchar} \colon \KUp[2] \to \KUp[2]$ is the $\degree$-th $\lambda$-operation $\lambda_{\degree}$.
        \end{example}

        % \begin{definition}
        %     A decategorified character of $\MM$ is a map of spaces $\hchar \colon \MM \to \ounit_{\cC}\units$.
        %     Such a map defines a decategorified character for each $\Sm$, which we denote by $\hchar_{\degree}$, or just by $\hchar$.
        % \end{definition}
    
        % \begin{definition}
        %     Let $\cC \in\CAlg(\Prsam[1])$ and $\hchar \colon \MM \to \ounit_{\cC}\units$ be a decategorified character. 
        %     Define 
        %     \begin{equation*}
        %         \beta_\hchar \coloneqq \bigsqcap_{\degree} \beta_{\hchar}^{\degree} \colon \underlying \ounit_{\cC} \to \underlying(\ounit_{\cC})^{\NN}.
        %     \end{equation*}
        %     Thinking of $(\ounit_{\cC})^{\NN}$ as $\ounit_{\cC}[\![t]\!]$, we would sometime write it formally as the generating function 
        %         \begin{equation*}
        %             \beta_{\hchar}(x) = \sum_{\degree} \beta_{\hchar}^{\degree}(x) t^{\degree}.
        %         \end{equation*}
        % \end{definition}

        \begin{remark}\label{rmrk:decategorification-of-alternating-powers}
            The category $\Mod_\cC$ is $\infty$-semiadditive and the semiadditive integral over a map $A \to \FunL_{\cC}(X,Y)$ is computed as the colimit over $A$ in the functor category. 
            We therefore deduce the following equality for any finite group $H$ and $V \in \cC^A$:
            \begin{equation*}
                \int^{\Mod_{\cC}}_{\BH} V= \colim_{\BH} V = V_{hH} \qin \cC.
            \end{equation*}
            In particular, for any character $\hchar \colon \BH \to \B\ounit_{\cC}\units$ the $\hchar$-twisted power operation in $\Mod_{\cC}$ is the $\hchar$-twisted alternating power in $\cC$:
            \begin{equation*}
                \beta^{\degree}_{H,\hchar}(-) = \int_{\BH}^{\Mod_\cC} (-)^{\degree} \cdot \overline{\hchar} =\alt_{H, \hchar} (-).
            \end{equation*}
        \end{remark}

        %%%%%%%%%%%%%%%%%%%%%%%%%%%%%%%%%%%%%%%%%%%%%%%%%%%%%%%%%%%%%%%%%%%%%%%%%%%%%%%%
        \subsubsection*{Induced character formula for alternating powers}

        In the context of $\chrHeight$-presentable $(\chrHeight,1)$-semiadditive categories, one can use the iterated induced character formula to compute the dimensions of alternating powers.

        \begin{lemma}
            Let $\cC$ be a symmetric monoidal $\chrHeight$-presentable, $(\chrHeight,1)$-semiadditive category.
            Let $H$ be a finite group, $H \to \Sm$ a homomorphism, $V \in \cC\tdbl$ and $\hchar \colon \Sm \to \ounit_{\cC}\units$ a character. Then 
            \begin{equation*}
                \dim^t(\alt_{H,\hchar}) = \int_{\L^t \BH} \chi^t_{\Tm V \otimes \ounit_{\cC}[\overline{\hchar}]} = \int_{\L^t \BH} \chi^t_{\Tm V} \cdot \chi^t_{\ounit_{\cC}[\overline{\hchar}]} \qin \laxloops^{t} \cC.
            \end{equation*}
        \end{lemma}

        \begin{proof}
            This follows by~\cref{cor:induced-k-character-formula}, as $\alt_{H,\hchar} V$ is the induced representation of $\Tm V \otimes \ounit_{\cC}[\overline{\hchar}]$ under the map $\BH \to \pt$.
        \end{proof}

        \begin{definition}\label{def:transgression}
            Let $A\in\spc$, $X\in \cnSp$ and $\hchar \colon A\to X$ be a map of spaces. We define the transgression of $\hchar$ as
            \begin{equation*}
                \tg(\hchar) \colon \L A \xto{\L\hchar} \L X\simeq  X \times \Omega X \xto{\eta +\id} \Omega X
            \end{equation*}
            where $\eta \colon X \to \Omega X$ is the Hopf map. 
        \end{definition}
        
        \begin{remark}
            In the case where $A = \BH$ and $X = \Sigma^n M$ for an $\EE_1$-group $H$ and an abelian group $M$, the $\pi_0$ of the above transgression defines an operation
            \begin{equation*}
                \H^n(\BH; M) \to \H^{n-1}(\L\BH; M)
            \end{equation*}
            which is named \quotes{transgression} in~\cite{Willerton-2008-transgression}. Composing with the map $H = \Omega \BH \to \L\BH$, we get an operation
            \begin{equation*}
                \H^n(\BH; M) \to \H^{n-1}(H; M)
            \end{equation*}
            which is called the transgression in~\cite{Dijkgraff-Witten-1990-transgression}. See also \href{https://ncatlab.org/nlab/show/transgression+in+group+cohomology}{nLab} for further discussion of the transgression. 
        \end{remark}

        \begin{lemma}\label{lem:tg-of-const}
            Let $x \colon A \to X$ be a constant map choosing $x\in X$. Then its transgression is the constant map
            \begin{equation*}
                \tg(x) \equiv \eta x \qin \Omega X.
            \end{equation*}
        \end{lemma}
        \begin{proof}
            Under the identification $\L X \simeq X \oplus \Omega X$, $\L x \colon \L A \to X \oplus \Omega X$ is the constant map on $(x,0)$.
        \end{proof}

        % \begin{lemma}\label{lem:tg-additive}
        %     Transgression is additive. That is, for $\hchar,\hchar' \colon A \to X$,
        %     \begin{equation*}
        %         \tg(\hchar + \hchar') = \tg(\hchar) + \tg(\hchar').
        %     \end{equation*}
        % \end{lemma}
        % \begin{proof}
        %     The addition is defined as the composition
        %     \begin{equation*}
        %         A \xto{\Delta} A \times A \xto{\hchar \times \hchar'} X \times X \xto{+} X.
        %     \end{equation*}
        %     The claim now follows as $\L(X \times X) \simeq \L X \times \L X$, and using the naturality of the map $\eta + \id$.
        % \end{proof}
        
        \begin{proposition}\label{prop:transgersion-is-charcter-of-character}
            Let $\cC$ be a symmetric monoidal $\chrHeight$-presentable category.
            Let $H$ be a finite group and $\hchar \colon H \to \ounit_{\cC}\units$ a character. Then for $1 \le t \le \chrHeight$
            \begin{equation*}
                \chi^t_{\ounit_{\cC}[\hchar]}=\tg^t(\hchar) \qin (\laxloops^t \cC)^{\L^t \BH}.
            \end{equation*}
        \end{proposition}
        
        \begin{proof}
            It is enough to show that, for a 1-finite space $A$ and $\hchar \colon A \to \cC\units$, we have
            \begin{equation*}
                \chi_{\hchar} \simeq \tg(\hchar).
            \end{equation*}

            In this case the character map $\chi_{\hchar} \colon \L A \to \Omega \cC\core$ factors as 
            \begin{equation*}
                \Map(\TT, A) \xto{\hchar \circ -} \Map(\TT, \cC\units) \xto{\tr} \ounit_{\cC} \units \to \Omega \cC\core.
            \end{equation*}
            By~\cite[Lemma~3.2.7]{Keidar-Ragimov-2025-twisted}, under the identification $\Map(\TT, \cC\units) \simeq \cC\units \oplus \ounit_{\cC}\units$ the map $\tr$ becomes $\eta + \id$. Therefore 
            \begin{equation*}
                \chi_{\hchar} = \tg(\hchar).
            \end{equation*}            
            % It also suffices to assume $A = \BK$ is connected. Choose a basepoint $x_0 \in \BK$ and let $V \coloneqq \hchar(x_0)$ be the corresponding invertible element. Consider
            % \begin{equation*}
            %     \hchar' \coloneqq V^{-1} \otimes \hchar \colon \BK \to \cC\units.
            % \end{equation*}
            % Then $\hchar'$ factors through the connected cover
            % \begin{equation*}
            %     \hchar' \colon \BK \to \B\ounit_{\cC}\units.
            % \end{equation*}
            
            % Now using the multiplicativity of the character and the additivity of the transgression (addition in the spectrum $\cC\units$ corresponds to $\otimes$ and in $\ounit_{\cC}\units$ to $\cdot$) we learn that
            % \begin{equation*}
            %     \dim(V^{-1}) \cdot \chi_{\hchar} = \tg(V^{-1}) \cdot \tg(\hchar).
            % \end{equation*}
            % By~\cref{lem:tg-of-const}, $\tg(V^{-1}) = \eta V^{-1}$ which identifies with $\dim(V^{-1})$ by~\cite[Proposition~3.20]{CSY-cyclotomic}.
        \end{proof}

        \begin{remark}
            Note that when $\cC$ is $(R,\chrHeight)$-oriented, then $\laxloops \cC$ is $(R,\chrHeight-1)$-oriented. If $\hchar$ factors through $\ndual{R}$, then its transgression factors as
            \begin{equation*}
                \tg(\hchar) \colon \L\BH \to \L\ndual{R} \xto{\eta + \id} \Omega \ndual{R} \simeq \ndual[\chrHeight-1]{R}.
            \end{equation*}
            That is, transgression sends orientation characters to maps factoring through the orientation.
        \end{remark}
        
        \begin{corollary}\label{cor:k-character-of-alternating}
            Let $\cC$ be a symmetric monoidal $\chrHeight$-presentable, $(\chrHeight,1)$-semiadditive category. Then for $V \in \cC\tdbl$, a homomorphism $H \to \Sm$, and a character $\hchar \colon H \to \ounit_{\cC}\units$,
            \begin{equation*}
                \dim^t (\alt_{H, \hchar} V) \simeq \int_{\L^t \BH} \chi^t_{\Tm V} \cdot \tg^t(\overline{\hchar}) \qin \laxloops^{t} \cC.
            \end{equation*}
        \end{corollary}

    %%%%%%%%%%%%%%%%%%%%%%%%%%%%%%%%%%%%%%%%%%%%%%%%%%%%%%%%%%%%%%%%%%%%%%%%%%%%%%%%
    \subsection{Wreath products and permutation representations}
    \label{subsec:higher-cats-and-alternating}
    %%%%%%%%%%%%%%%%%%%%%%%%%%%%%%%%%%%%%%%%%%%%%%%%%%%%%%%%%%%%%%%%%%%%%%%%%%%%%%%%
        Our goal in this subsection is to express the character of $\Tm W \in \cC^{\B(G \wr \Sm)}$ for any $W \in \cC\dblA[\BG]$ in terms of the character of $W$ itself.
        We begin by recalling the structure of stabilizers in wreath products. Using this, we deduce the general result from the special case $G = e$, which was treated in~\cite{Keidar-Ragimov-2025-twisted}.        
        
        \begin{notation}\label{not:not1}
            Let $G$ be a group. We denote elements of $G^{\degree}$ by $\underline{g} = (g_1, \dots, g_{\degree})$. 
            Recall that the wreath product $G \wr \Sm$ is the semidirect product $G^{\degree} \rtimes \Sm$. We denote its elements by $(\underline{g}; \sigma)$ where $\underline{g} \in G^{\degree}$ and $\sigma \in \Sm$.
        \end{notation}
        
        \begin{notation}\label{not:not2}
            For $\underline{g}\in G^{\degree}$, we say that $x\in \underline{g}$ if there exists $i$ for which $g_i = x$. We define its multiplicity as the number of indices equal to $x$
            \begin{equation*}
                \mul_{\underline{g}}(x) \coloneqq |\{i \mid g_i = x\}|.
            \end{equation*}
        \end{notation}
    
        \begin{notation}\label{not:sigma}
            For $\sigma\in \Sm$. We denote by
            \begin{enumerate}
                \item $\Nk$ --- The number of $k$ cycles in the decomposition into disjoint cycles of $\sigma$.
                \item $\numcyc = \sum_k \Nk$ --- the number of cycles in the decomposition of $\sigma$ into disjoint cycles.
            \end{enumerate}
        \end{notation}
        
        \begin{definition}\label{def:root-of-central-element}
            Let $G$ be a group, $z \in \Z(G)$ and $k \in \NN$. We define the group
            \begin{equation*}
                G\kz{z} \coloneqq G \langle x \mid x^k = z, \ gx = xg\  \forall g \in G \rangle.
            \end{equation*}
            It is an extension of $\Ck$ by $G$
            \begin{equation*}
                1 \to G \to G\kz{z} \to \Ck \to 1.
            \end{equation*}
            In particular, $G\kz{e} = G \times \Ck$.
        \end{definition}

        % \begin{remark}\label{rmrk:adding-roots-to-abelian-group}
        %     By definition, if $G$ is a finite abelian group then so is $G\kz{z}$ for any $z \in G$ and $k \in \NN$.
        % \end{remark}
    
        \begin{lemma}\label{lem:decomposition-of-wreath}
            Let $G$ be a finite group. Then
            \begin{equation*}
                \begin{split}
                    \L\B(G \wr \Sm) 
                    & = \bigsqcup_{[\sigma] \in \Sm/\mathrm{conj}} \ 
                    \bigsqcap_{k=1}^\degree \ 
                    \bigsqcup_{[\underline{g}^k] \in \Sm[\Nk] \backslash G^{\Nk} / \conj} \ 
                    \bigsqcap_{x \in \underline{g}^k} \B (\cent_G(x)\kz{x} \wr \Sm[\mul_{\underline{g^k}}(x)]) \\
                    & = \bigsqcup_{\sigma \in \Sm/\mathrm{conj}} \ 
                    \bigsqcup_{[\underline{g}^1], \dots, [\underline{g}^{\degree}]} \ 
                    \bigsqcap_{k=1}^\degree \
                    \bigsqcap_{x \in \underline{g}^k} \B (\cent_G(x)\kz{x} \wr \Sm[\mul_{\underline{g}^k}(x)]).
                \end{split}
            \end{equation*}
            In particular,
            \begin{equation*}
                \L\B\Sm  = \bigsqcup_{[\sigma] \in \Sm / \conj} \ \bigsqcap_{k=1}^{\degree} \B(\ZZ/k \wr \Sm[\Nk]),
            \end{equation*}
            and the group $\bigsqcap_{k=1}^{\degree} (\ZZ/k \wr \Sm[\Nk])$ is the centralizer $\Csigma$. The groups $\ZZMod{k}$ are generated by the $k$-cycles and the groups $\Sn[\Nk]$ permutes the different $\Nk$ $k$-cycles.
        \end{lemma}
        
        \begin{proof}
            When $G = e$, this result is classical. For a proof in the general case, see~\cite{Bernhardt-Niemeyer-Rober-Wollenhaupt-2022-wreath}. A deduction of the general case from the special case $G = e$ using the symmetric monoidal dimension in $\Span(\spcpi[1])$ can be found in~\cref{app:conjugacy-of-wreath}. See also~\cite{Barthel-Stapleton-2017-power-operations} for a discussion of the $\pi_0$-level version of this result.
        \end{proof}

        \begin{remark}\label{rmrk:conjugacy-classes-to-connected-components}
            Let $(\underline{h}; \sigma) \in G \wr \Sm$. We can explicitly describe the connected component corresponding to the conjugacy class of this element. Suppose that $\sigma$ decomposes into disjoint cycles as
            \begin{equation*}
                \sigma = (a^1_1, \dots, a^1_{m_1}) (a^2_1, \dots, a^2_{m_2}) \cdots (a^r_1, \dots, a^r_{m_r}).
            \end{equation*}
            Then the conjugacy class of $(\underline{h}; \sigma)$ corresponds to the connected component of $[\sigma], [\underline{g}^1], \dots, [\underline{g}^{\degree}]$, where
            \begin{equation*}
                \underline{g}^k = (\, \sqcap_{j} h_{a^i_j} \mid m_i = k\,).
            \end{equation*}
        \end{remark}

        \begin{definition}\label{def:stabilizer-action}
            Let $\cC$ be a symmetric monoidal category.
            Let $\sigma \in \Sm$ and $V \in \cC \dbl$. By abuse of notations, we denote by
            \begin{equation*}
                (\dim V)^{\numcyc} \coloneqq \bigsqcap_{k=1}^{\degree} \Tm[\Nk] \Res^{\TT}_{\ZZ/k} \dim(V) = \bigsqcap_{k=1}^{\degree} \dim(V)^{\Nk} \qin \ounit_{\cC}^{\B\Csigma}
            \end{equation*}
            the object with the corresponding $\Csigma$-action.
            It is the action of $\Csigma \simeq \bigsqcap_{k = 1}^{\degree} ((\ZZ / k)^{\Nk} \rtimes \Sm[\Nk])$ on $(\dim V)^{\numcyc} \in \ounit_{\cC}$, where the component $(\ZZ / k)_{\tau} \subseteq (\ZZ / k)^{\Nk}$ corresponding to a cycle $\tau \subseteq \sigma$ of length $k$, acts on the component $(\dim (V))_{\tau}$ by restriction along the natural map $\ZZ/k \into \TT$. $\Sm[\Nk]$ acts on $\bigsqcap_{\substack{\tau \subseteq \sigma, \\ |\tau| = k}} (\dim V)_{\tau}$ by permutation.
        \end{definition}

        \begin{remark}
            When $\cC$ is an $\chrHeight$-category, as we denote the product in $\ounit_{\cC} = \laxloops \cC$ by $\otimes$, we will denote this element by $(\dim V)\om[\numcyc] \in (\laxloops \cC)^{\B\Csigma}$.
        \end{remark}
    
        \begin{lemma}[{\cite[Lemma~4.7]{Ramzi-2025-endomorphisms-of-THH},~\cite[Lemma~3.3.10]{Keidar-Ragimov-2025-twisted}}]\label{lem:character-of-Tm}
           Let $\cC$ be a symmetric monoidal category, $V\in \cC\dbl$, $[\sigma] \in \Sm / \conj$. Then 
           \begin{equation*}
               \chi_{\Tm V}(\sigma) \simeq \dim(V)^{\numcyc} \qin \ounit_{\cC}^{\B \Csigma}.
           \end{equation*}
        \end{lemma}

        \begin{corollary}\label{cor:dim-of-alternating-depends-only-on-dim}
            Assume that $\cC$ is 1-semiadditive and presentable. Then for any $V \in \cC\dbl$ and a character $\hchar \colon \Sm \to \ounit_{\cC}\units$, the dimension $\dim(\alt_{\hchar} V)$ depends only on $\dim(V) \in \ounit_{\cC}^{\B\TT}$ and $\hchar$.
        \end{corollary}
        \begin{proof}
            By~\cref{lem:character-of-Tm}, $\chi_{\Tm V}$ depends only on $\dim(V)$. The result now follows by~\cref{cor:k-character-of-alternating}.
        \end{proof}

        % \begin{remark}\label{rmrk:dim-of-alternating-depends-cyclic-group}
        %     More precisely, the dimension $\dim(\alt_{\hchar} V)$ depends only on $\dim(V) \in (\laxloops \cC)^{\vee_k \B\Ck}$ --- the restriction of $\dim(V)$ to any finite subgroup $\Ck \subseteq \TT$.
        % \end{remark}
      
        The proof of~\cref{lem:character-of-Tm} uses Baez's and Dolan's cobordism hypothesis~\cite{Baez-Dolan-1995-cobordism} in dimension 1. Its proof sketched in~\cite{Lurie-2009-cobordism} and proved (in dimension 1) in~\cite{Harpaz-2012-cobordism}. To generalize this lemma to the case where $V$ admits a group action, we use the cobordism hypothesis with tangential structures, which was proved (assuming the cobordism hypothesis) in~\cite[Theorem~2.4.18]{Lurie-2009-cobordism}. We use the slightly less general theorem, when the space is connected:
       \begin{theorem}[{The cobordism hypothesis for $G$-manifolds in dimension 1~\cite{Harpaz-2012-cobordism},~\cite[Theorem~2.4.26]{Lurie-2009-cobordism}}]\label{thm:cobordism-hyp-in-dim-1}
           Let $G$ be an $\EE_1$-group and $\chi \colon G \to O(1)$ a homomorphism of groups. Then for any symmetric monoidal category $\cC$, the evaluation at the point defines an isomorphism of spaces
           \begin{equation*}
               \Map^{\otimes}(\Bord_1^G, \cC) \simeq (\cC\dblspace)^{h G}.
           \end{equation*}
       \end{theorem}

       \begin{remark}\label{rmrk:G-cobordism}
           If $G$ is discrete and the map $G \to O(1)$ is trivial, we can explicitly describe the category $\Bord_1^G$ (see~\cite[Notation~2.4.16, Definition~Sketch~2.4.17]{Lurie-2009-cobordism}). Its objects are liftings
           \begin{equation*}
               \begin{tikzcd}
                    & \BG \\
                    \pt && {\B O(1),}
                    \arrow["e", from=1-2, to=2-3]
                    \arrow[dashed, from=2-1, to=1-2]
                    \arrow[from=2-1, to=2-3]
                \end{tikzcd}
           \end{equation*}
           i.e.\ framed points with a lift to $\BG$, and the morphisms are framed 1-dimensional manifolds with a lift to $\BG$, respecting the boundary. 
           Informally, we can think of $\Bord_1^G$ as consisting of framed points, and the morphisms are 1-dimensional manifolds where on each connected component $M$, we add a weight which is an element in $G$ if $M$ is a line or an element in $G$ up to conjugation if $M$ is a circle. The weight of the composition of bordisms will be the product of the weights.
       \end{remark}

        \begin{recollection}\label{recol:T-action-on-BG}
            Let $G$ be a discrete group. A $\TT$-action on the space $\BG$ is the same as a map of pointed spaces $\B \TT \to \B\Aut(\BG)$, which, as $\B\TT$ is simply connected and $\B\Aut(\BG)$ is 2-truncated, is the same as an element in $\pi_2(\B\Aut(\BG)) = \Z(G)$. That is, a $\TT$-action on $\BG$ is the data of a central element of $G$.
        \end{recollection}
    
        \begin{lemma}
            Let $G$ be a discrete group and $z\in \Z(G)$ corresponding to a $\TT$-action on $\BG$. Then, restricting the $\TT$-action to $\Ck\subseteq \TT$, we have
            \begin{equation*}
                (\BG)_{h \Ck} \simeq \B G\kz{z}.
            \end{equation*}
        \end{lemma}
    
        \begin{proof}
            This is proved in~\cref{lem:cyclic-fixed-points-of-BG}. We remark that this result is not needed for the rest of the paper if one replaces $\BG\kz{z}$ with $(\BG)_{h\Ck}$ everywhere.
        \end{proof}
    
        \begin{remark}\label{rmrk:roots-of-centralizer-acts-on-character}
           Let $\cC$ be a symmetric monoidal $\chrHeight$-category, $G$ a discrete group and $V \in \cC\dblA[\BG]$.
           For any $x \in G$, the character map restricted to $\B\cent_G(x)$
            \begin{equation*}
                \chi_V \colon \B\cent_G(x) \into \L\BG \to \ounit_{\cC},
            \end{equation*}
            is $\TT$-invariant by~\cref{cor:characters-are-T-invariant}, and in particular $\Ck$-invariant for any $k \ge 1$. Therefore, it is equivalent to a map
            \begin{equation*}
                \B\cent_G(x)\kz{x} \simeq (\B\cent_G(x))_{h\Ck} \to \ounit_{\cC},
            \end{equation*}
            choosing $\chi_V(x)$. That is, $\chi_V(x)$ admits a natural action of $\cent_{G}(x)\kz{x}$.

            In the case where $G = e$, this is the $\Ck$-action on $\dim(V)$ restricted from the $\TT$-action.
        \end{remark}
        
        \begin{notation}
             Let $A$ be a space and $B\in \spc^A$ be an $A$-local system of spaces. 
             By embedding $\spc$ into the 2-category of spans $\Span_{1.5}(\spc)$\footnote{The notation $\Span_{1.5}$ indicates that the 2-arrows are maps between spans, to distinguish $\Span_2$ where the 2-arrows are spans of spans.}, sending morphisms to the right way maps, we get an $A$-local system with values in $\Span_{1.5}(\spc)$. We will abuse notation and call this local system $B$ as well. 
        \end{notation}
    
        \begin{lemma}\label{lem:colimt-of-character-in-spans}
            Let $B \in \spc^A$ and let 
            \begin{equation*}
                \chi^{\Span}_B \colon \L A \to \laxloops \Span_{1.5}(\spc) \simeq \spc
            \end{equation*}
            be the character of $B$ thought of as a $\Span_{1.5}(\spc)$-valued local system. Then 
            \begin{equation*}
                \colim_{\L A} \chi^{\Span}_B =\L(\colim_A B),
            \end{equation*}
            where both colimits are computed in $\spc$. 
        \end{lemma}
    
        \begin{proof}
            We observe that the colimit of $B$ in $\Span_{1.5}(\spc)$ is computed as the colimit of $B$ in $\spc$. As $\dim_{\Span_{1.5}(\spc)}(\colim_A B) = \L(\colim_A B)$, by the induced character formula
            \begin{equation*}
                \colim_{\L A} \chi_B = \int_{\L A} \chi_B \simeq \dim_{\Span_{1.5}(\spc)}(\colim_A B) = \L(\colim_A B).
            \end{equation*}
        \end{proof}

        % \begin{construction}[Categorical power operation]\label{cons:action-of-wreath}
        %     Let $\cC\in \CAlg(\PrL)$ and $I \in \Cat$. There is a natural map
        %     \begin{equation*}
        %         \Fun(I,\cC) \to \Fun((I^{\degree})_{h\Sm},(\cC^{\degree})_{h\Sm}) \to \Fun((I^{\degree})_{h\Sm}, (\cC\om)_{h\Sm}) \xto{\mu} \Fun((I^{\degree})_{h\Sm}, \cC)
        %     \end{equation*}
        %     where $\mu$ is the multiplication map $\cC\om \to \cC$, which is $\Sm$-invariant as $\cC$ is commutative.
        %     In particular, if $I=\BG$ we get the map
        %     \begin{equation*}
        %         \Fun(\BG, \cC) \to \Fun(\B(G \wr \Sm),\cC)
        %     \end{equation*}
        %     sending $V$ to $\Tm V = V\om$.
        % \end{construction}
            
        \begin{proposition}\label{lem:character-of-wreath-product}
            Let $\cC$ be a symmetric monoidal category and $G$ a finite group. Let $V\in \cC\dblA[\BG]$. Then the character of the $G \wr \Sm$-representation $\Tm V$ is the map
            \begin{gather*}
                \chi_{\Tm V} 
                \colon \L\B(G \wr \Sm) 
                \simeq 
                    \bigsqcup_{[\sigma] \in \Sm / \conj} \ 
                    \bigsqcup_{[\underline{g}^1],\dots,[\underline{g}^{\degree}]} \ 
                    \bigsqcap_{k=1}^{\degree} \ 
                    \bigsqcap_{x \in \underline{g}^k} \B(\cent_G(x)\kz{x} \wr \Sm[\mul_{\underline{g}^k}(x)])
                \to \ounit_{\cC} \\
                ((x)_{x \in \underline{g}^k})_{k=1}^{\degree} 
                \mapsto \bigsqcap_{k=1}^{\degree} \bigsqcap_{x \in \underline{g}^k} \Tm[\mul_{\underline{g}^k}(x)]\chi_V(x) 
                \simeq \bigsqcap_{k=1}^{\degree} \bigsqcap_{x \in \underline{g}^k} \chi_V(x)^{\mul_{\underline{g}^k}(x)}
            \end{gather*}
            where the $\cent_G(x)\kz{x}$-action is as described in~\cref{rmrk:roots-of-centralizer-acts-on-character}
        \end{proposition}

        \begin{proof}
            First note that $\L\B(G \wr \Sm) \simeq \L((\BG^{\degree})_{h\Sm})$ and therefore, by~\cref{lem:colimt-of-character-in-spans} 
            \begin{equation*}
                \L\B(G \wr \Sm) \simeq \colim_{\L\B\Sm} \chi^{\Span}_{\BG^{\degree}}.
            \end{equation*}
            By~\cref{lem:character-of-Tm}, 
            \begin{equation*}
                \chi^{\Span}_{\BG^{\degree}}(\sigma) \simeq (\L \BG)^{\numcyc} \qin \spc^{\B \Csigma}.
            \end{equation*}
            The character map $\chi_{\Tm V} \colon \L\B(G \wr \Sm) \to \ounit_{\cC}$ can therefore be described as a system of $\Csigma$-invariant maps
            \begin{equation*}
                \chi_{\Tm V}(\sigma) \colon (\L\BG)^{\numcyc} \to \ounit_{\cC}.
            \end{equation*}
            It is enough to check that the map $\chi_{\Tm V}(\sigma)$ is the claimed map, and it is enough to verify it in the universal case $\Bord_1^G$ (\cref{thm:cobordism-hyp-in-dim-1,rmrk:G-cobordism}). In that case, it is straightforward using~\cref{rmrk:conjugacy-classes-to-connected-components}.
        \end{proof}
        
        \begin{corollary}\label{cor:character-when-no-T-action}
            Let $\cC$ be a 1-semiadditive symmetric monoidal category, $G$ a finite group and  $V \in  \cC\dbl$. Assume that the $\TT$-action on $\dim(V)$ is trivial. Then the character of the $G \wr \Sm$-representation $\Tm (\BG)^* V$ is
            \begin{gather*}
                \chi_{\Tm (\BG)^* V}
                \colon \L\B(G \wr \Sm) 
                \simeq 
                    \bigsqcup_{[\sigma] \in \Sm / \conj} \ 
                    \bigsqcup_{[\underline{g}^1],\dots,[\underline{g}^{\degree}]} \ 
                    \bigsqcap_{k=1}^{\degree} \ 
                    \bigsqcap_{x \in \underline{g}^k} \B(\cent_G(x)\kz{x} \wr \Sm[\mul_{\underline{g}^k}(x)])
                \to \ounit_{\cC} \\
                ((x)_{x \in \underline{g}^k})_{k=1}^{\degree} 
                \mapsto \bigsqcap_{k=1}^{\degree} \bigsqcap_{x \in \underline{g}^k} \Tm[\mul_{\underline{g}^k}(x)] (\B\cent_G(x)\kz{x})^*(\dim V)
                \simeq \bigsqcap_{k=1}^{\degree} \bigsqcap_{x \in \underline{g}^k} (\dim V)^{\mul_{\underline{g}^k}(x)}.
            \end{gather*}
            That is, the $\cent_G(x)\kz{x}$-action is trivial, and the action of $\Sm[\mul_{\underline{g}^k}(x)]$ is by permuting the corresponding elements. 
        \end{corollary}

        \begin{remark}
            Let
            \begin{equation*}
                p \colon \B(G \wr \Sm) \to \B\Sm
            \end{equation*}
            be the natural projection. Noting that $\Tm (\BG)^* V \simeq p^* \Tm V$ for $V \in \cC\dbl$, one can alternatively prove~\cref{cor:character-when-no-T-action} using only~\cref{lem:character-of-Tm}, via the composite
            \begin{equation*}
                \chi_{p^* \Tm V} \colon \L\B(G \wr \Sm) \xto{(\L p)^*} \L\B\Sm \xto{\chi_{\Tm V}} \ounit_{\cC}.
            \end{equation*}
            Although in this paper we consider only characters of representations of the form $\Tm(\BG)^* V$, we chose to present the computation in~\cref{lem:character-of-wreath-product} as we find it conceptually satisfying.
        \end{remark}

        \begin{corollary}\label{cor:dim^k-of-alternating-depends-only-on-dim^k}
            Let $\cC$ be a symmetric monoidal $\chrHeight$-presentable, $(\chrHeight,1)$-semiadditive category. Let $V \in \cC\tdbl$, $H \to \Sm$ a homomorphism of finite groups and $\hchar \colon H \to \ounit_{\cC}\units$ a character. Assume that the $\TT$-action on $\dim^i(V)$ is trivial for every $0\le i \le t$. Then $\dim^t(\alt_{H,\hchar} V) \in (\laxloops^t \cC)^{\B\TT^t}$ depends only on $\dim^t(V) \in (\laxloops^t \cC)^{\B\TT^t}$.
        \end{corollary} 

        % %%%%%%%%%%%%%%%%%%%%%%%%%%%%%%%%%%%%%%%%%%%%%%%%%%%%%%%%%%%%%%%%%%%%%%%%%%%%%%%%
        % \subsubsection*{Higher categories and power operations}

        % As with alternating powers, we can compute the iterated dimensions of twisted power operations using the induced character formula:
        % \begin{corollary}
        %     Let $\cC$ be a symmetric monoidal $\chrHeight$-presentable, $(\chrHeight,1)$-semiadditive category. Let $H \to \Sm$ be a homomorphism of finite groups and $\hchar \colon \BH \to \ounit_{\cC}\units$ be a pointed map. Then for $V \in (\laxloops \cC)\tdbl$,
        %     \begin{equation*}
        %         \dim^t \beta^{\degree}_{H, \hchar}(V) = \dim^t \alt_{H,\hchar} V = \int_{\L^t \BH} \chi^t_{\Tm V} \cdot \tg^t (\hchar) \qin \laxloops^{t+1}\cC.
        %     \end{equation*}
        % \end{corollary}

        % \begin{corollary}\label{cor:dim^t-of-alternating-depends-only-on-dim^t}
        %     Let $\cC$ be a symmetric monoidal $\chrHeight$-presentable, $(\chrHeight,1)$-semiadditive category. Let $V \in (\laxloops \cC)\tdbl$, $H \to \Sm$ and $\hchar \in \Ch(H; \cC)$. Assume that the $\TT^i$-action on $\dim^i(V)$ is trivial for every $0\le i \le t$. Then $\dim^t(\beta^{\degree}_{H,\hchar} V) \in \laxloops^{t+1} \cC^{\B\TT^t}$ depends only on $\dim^t(V) \in \laxloops^{t+1} \cC$.
        % \end{corollary}

        \begin{remark}\label{rmrk:enough-cyclic-subgroups}
            The character does not depend on the full $\TT^i$-action, but on its restriction to any finite subgroup. Therefore the corollary also holds if the restriction of the $\TT^i$-action on $\dim^i(V)$ to any finite subgroup is trivial, for every $1 \le i \le t$.
        \end{remark}

    %%%%%%%%%%%%%%%%%%%%%%%%%%%%%%%%%%%%%%%%%%%%%%%%%%%%%%%%%%%%%%%%%%%%%%%%%%%%%%%%
    \subsection{Characters induced from duals of stable stems}
    \label{subsec:categorical-alternating}
    %%%%%%%%%%%%%%%%%%%%%%%%%%%%%%%%%%%%%%%%%%%%%%%%%%%%%%%%%%%%%%%%%%%%%%%%%%%%%%%%
        We now specialize to the case $\cC = \Mod_{\cVectn}$ (\cref{def:cyc-closure}), which is a symmetric monoidal, $(\chrHeight+1)$-presentable, and $(\chrHeight+1,\infty)$-semiadditive category.
        
        We compute the iterated dimension of the alternating power of the unit:
        \begin{notation}
            Let $\Pi \field \in \sVect$ be the $(0|1)$-dimensional super vector space, so that $\pi_0 (\sVect\units) = \{\field, \Pi\field\}$

            Let $f \colon A \to \sVect\units$ be an $A$-local system of invertible super vector spaces. Define
            \begin{itemize}
                \item $d_+(f)$ to be the number of connected components of $A$ on which $f$ is $\field$ with the trivial action;
                \item $d_-(f)$ to be the number of connected components of $A$ on which $f$ is $\Pi \field$ with the trivial action.
            \end{itemize}
        \end{notation}
        \begin{proposition}\label{cor:inductive-formula-Vectn}
            Let $H$ be a finite group with a homomorphism $H \to \Sm$. Let $\hchar \colon \B\Sm \to \B(\cVectn)\units$ be a character in $\Mod_{\cVectn}$. 
            Then the $(\chrHeight+1)$-fold dimension
            \begin{equation*}
                \dim^{\chrHeight+1}(\alt_{H,\hchar} \cVectn) = d_+(\tg^{\chrHeight}(\overline{\hchar})) - d_-(\tg^{\chrHeight}(\overline{\hchar})) \qin \field.
            \end{equation*}
        \end{proposition}
    
        \begin{proof}
            By~\cref{cor:k-character-of-alternating}
            \begin{equation*}
                \begin{split}
                    \dim^{\chrHeight}(\alt_{H,\hchar} \cVectn) 
                    & \simeq \colim_{\L^{\chrHeight}\BH} \tg^{\chrHeight}(\overline{\hchar}) \\
                    & \simeq \bigoplus_{x \in \pi_0 \L^{\chrHeight}\BH} \tg^{\chrHeight}(\overline{\hchar})(x)^{h\Omega_x \L^{\chrHeight}\BH} \qin \sVect.
                \end{split}
            \end{equation*}
            For any $x\in \L^{\chrHeight} \BH$, $\tg^{\chrHeight}(\overline{\hchar})(x)$ lands in $\sVect\units$ and therefore is a one-dimensional vector space (up to $\Pi$). Its fixed-points is therefore one-dimensional if the action is trivial and trivial otherwise. The proof follows as $\dim(\Pi \field) = -1$.
        \end{proof}

        %%%%%%%%%%%%%%%%%%%%%%%%%%%%%%%%%%%%%%%%%%%%%%%%%%%%%%%%%%%%%%%%%%%%%%%%%%%%%%%%
        \subsubsection*{Characters induced from duals of stable stems}

        \label{characters-from-dual-stable-stems}
        The $(\chrHeight+1)$-category $\Mod_{\cVectn}$ is $(\SS,\chrHeight)$-oriented. Recall that $(\SS,\chrHeight)$-orientation characters of symmetric groups are maps of connected spaces $\B\Sm \to \B\In$. 
        \begin{equation*}
            \pi_0 \Map_*(\B\Sm, \B\In) \simeq \pi_0 \Map_{\cnSp}(\redSS[\B\Sm], \Sigma \In) \simeq \pi_0\Map_{\cnSp}(\Omega \redSS[\B\Sm], \In) \simeq \widehat{\pi}_{\chrHeight+1}(\redSS[\B\Sm]).
        \end{equation*}
        Therefore orientation characters are classified by $\widehat{\pi}_{\chrHeight+1}(\redSS[\B\Sm])$.
        The composition
        \begin{equation*}
            \redSS[\B\Sm] \to \redSS[\B\Sinfty] \simeq \redSS[\tau_{\ge 1}\SS] \xto{c} \tau_{\ge 1}\SS,
        \end{equation*}
        defines a map
        \begin{equation*}
            \widehat{\pi}_{\chrHeight+1}(\SS) = \widehat{\pi}_{\chrHeight+1}(\tau_{\ge 1}\SS) \to \widehat{\pi}_{\chrHeight+1}(\redSS[\B\Sm]).
        \end{equation*}
        That is, each functional $\pichar$ on $\pin$ defines a character of $\Sinfty$, which can be restricted to a character of $\Sm$. We denote the corresponding alternating powers by $\alt_{\pichar}$.
        
        These characters and their corresponding alternating powers were studied in~\cite{Keidar-Ragimov-2025-twisted}.
        There it was shown that for \emph{additive} $(\SS,\chrHeight)$-oriented categories in low heights, there exists a simple generating function computing the dimensions of the alternating powers (\cite[Proposition 5.2.8, Lemma 5.2.11, Lemma 5.2.12]{Keidar-Ragimov-2025-twisted}), under some conditions.
        
        In $\Mod_{\cVectn}$ we can get similar results for the iterated dimension:
        Recall that for $\chrHeight \le 2$, $\widehat{\pi}_{\chrHeight+1}(\SS)$ is a cyclic group.
        \begin{proposition}\label{prop:generating-function}
            Let $0 \le \chrHeight \le 2$. Let $\pichar \in \widehat{\pi}_{\chrHeight+1}(\SS)$ be a generator and\footnote{Here $\Mod_{\cVectn[0] } = \catname{Vect}_{\field\cyc}$.}  $V \in (\Mod_{\cVectn})\ndbl[\chrHeight+1]$ such that the $\TT^i$-action on $\dim^i(V)$ is trivial for $1\le i \le \chrHeight+1$. Then
            \begin{equation*}
                (\sum_{\degree} \dim^{\chrHeight+1}(\Sym V) t^{\degree}) (\sum_{\degree} \dim^{\chrHeight+1}(\alt_{\pichar} V) (-t)^{\degree}) = 1.
            \end{equation*}
        \end{proposition}
        \begin{proof}
            For convenience, denote $\cC \coloneqq \Mod_{\cVectn}$, and $\ounit_{\cC} = \cVectn$.
            Consider the twisted graded category $\Gr^{\pichar}_{\ZZ} \cC$ as in~\cite[Definition~2.4.25]{Keidar-Ragimov-2025-twisted}. That is, it is the Thom construction along the map
            \begin{equation*}
                \ZZ \xonto{\pichar} \widehat{\pi}_{\chrHeight+1}(\SS) \to \Sigma^2 \In \to \Sigma^2 \ounit_{\cC}\units \to (\Mod_{\cC})\units.
            \end{equation*}
            In particular, it is a commutative algebra in $\Mod_{\cC}$, which as a monoidal category is isomorphic to $\FunDay(\ZZ, \cC)$, and thus admits a natural structure of an $(\chrHeight+1)$-category.
    
            In this category, consider the objects $V\shift{0} \oplus V\shift{1}$ i.e.\ the graded object which is $V$ at places 0 and 1, and zero everywhere else. 
            By~\cite[Lemma~5.2.5, Lemma~5.2.6]{Keidar-Ragimov-2025-twisted}, the restricted action to finite subgroups of $\TT^i$ on $\dim^i(\ounit_{\cC}\shift{1})$ is trivial. Therefore so is the restricted action on
            \begin{equation*}
                \dim^i(V\shift{0} \oplus V\shift{1}) = \dim^i(V\shift{0}) + \dim^i(V\shift{0})\cdot \dim^i(\ounit_{\cC}\shift{1}) = \dim^i(V) \cdot (1 + \dim^{i}(\ounit_{\cC}\shift{1})).
            \end{equation*}
    
            By~\cite[Lemma~5.2.3]{Keidar-Ragimov-2025-twisted}
            \begin{equation*}
                \dim^{\chrHeight+1}(\ounit_{\cC}\shift{1}) = \eta^{\chrHeight+1} \cdot \pichar \qin \widehat{\pi}_{0}(\SS) = \mu_{\infty} \subseteq \field\cyc = \laxloops^{\chrHeight+1} \cC.
            \end{equation*}
            
            It is straightforward to verify that $\eta^{\chrHeight+1} \cdot \pichar = -1 \in \widehat{\pi}_{0}(\SS)$.
            Therefore,
            \begin{equation*}
                \dim^{\chrHeight+1}(V\shift{0} \oplus V\shift{1}) = \dim^{\chrHeight+1}(V) \cdot (1 + \dim^{\chrHeight+1}(\ounit_{\cC}\shift{1})) = 0 \qin \field\cyc.
            \end{equation*}
            By~\cref{cor:dim^k-of-alternating-depends-only-on-dim^k},~\cref{rmrk:enough-cyclic-subgroups},
            \begin{equation*}
                \dim^{\chrHeight+1} \Sym(V\shift{0} \oplus V\shift{1}) = \dim^{\chrHeight+1} \Sym (0) = \begin{cases}
                    1, & \degree = 0 \\
                    0, & \degree > 0.
                \end{cases}
            \end{equation*}
            On the other hand
            \begin{equation*}
                \begin{split}
                    \Sym(V\shift{0} \oplus V\shift{1}) 
                    & = (\bigoplus_{i+j = \degree} \Ind_{\Sm[i] \times \Sm[j]}^{\Sm} V\shift{0}\om[i] \otimes V\shift{1}\om[j])_{h\Sm} \\
                    & \simeq \bigoplus_{i + j = \degree} (V\shift{0}\om[i])_{h\Sm[i]} \otimes (V\shift{1}\om[j])_{h\Sm[j]}.
                \end{split}
            \end{equation*}
            By definition,
            \begin{equation*}
                (V\shift{0}\om[i])_{h\Sm[i]} \simeq (\Sym[i] V)\shift{0}.
            \end{equation*}
            By~\cite[Proposition~4.1.8, Remark~4.3.6]{Keidar-Ragimov-2025-twisted}
            \begin{equation*}
                (V\shift{1}\om[j])_{h\Sm[j]} \simeq (\alt[j]_{\pichar} V)\shift{j} \simeq (\alt[j]_{\pichar} V)\shift{0} \otimes \ounit_{\cC}\shift{j}.
            \end{equation*}
            Thus,
            \begin{equation*}
                \begin{split}
                    \dim^{\chrHeight+1} \Sym(V\shift{0} \oplus V\shift{1}) 
                    & = \sum_{i+j = \degree} \dim^{\chrHeight+1} \Sym[i](V) \cdot \dim^{\chrHeight+1} \alt[j]_{\pichar} V \cdot (\dim^{\chrHeight+1} \ounit_{\cC}\shift{1})^j \\
                    & = \sum_{i+j = \degree} \dim^{\chrHeight+1} \Sym[i](V) \cdot (-1)^j\dim^{\chrHeight+1}\alt[j]_{\pichar} V,
                \end{split}
            \end{equation*}
            which is the requested formula.
        \end{proof}

%%%%%%%%%%%%%%%%%%%%%%%%%%%%%%%%%%%%%%%%%%%%%%%%%%%%%%%%%%%%%%%%%%%%%%%%%%%%%%%%
\section{Transchromatic characters as monoidal characters}
\label{subsec:chromatic-induction}
%%%%%%%%%%%%%%%%%%%%%%%%%%%%%%%%%%%%%%%%%%%%%%%%%%%%%%%%%%%%%%%%%%%%%%%%%%%%%%%%

    This section moves from the higher-categorical world to the chromatic one. Working in the category $\ModEn$, we extend the induced-character machinery using the transchromatic character of Hopkins--Kuhn--Ravenel and Stapleton. We isolate a family of functions that can be categorified, prove that their iterated characters coincide with the corresponding transchromatic characters, and derive an explicit formula for these characters. This formula, in turn, lets us compute the dimensions of twisted alternating powers.

    % In~\cref{subsec:categorification-of-functions}, we study a family of functions to $\Enp$ that admit categorification in the chromatic setting, and develop character theory for them, relating it to transchromatic character theory. We use it to establish explicit formulas for dimensions and power operations, analogous to those in~\cref{sec:categorical}. Finally, in~\cref{subsec:minus-one}, we give an explicit computation in height~1.    

    %%%%%%%%%%%%%%%%%%%%%%%%%%%%%%%%%%%%%%%%%%%%%%%%%%%%%%%%%%%%%%%%%%%%%%%%%%%%%%%%
    \subsection{Characters and categorification of functions}
    \label{subsec:categorification-of-functions}
    %%%%%%%%%%%%%%%%%%%%%%%%%%%%%%%%%%%%%%%%%%%%%%%%%%%%%%%%%%%%%%%%%%%%%%%%%%%%%%%%
        We begin by constructing a decategorification map from representations in $\ModEn$ to $\Enp$-valued functions, using the chromatic Nullstellensatz.
        Since the chromatic Nullstellensatz (\cite{Burklund-Schlank-Yuan-2022-Nullstellensatz}) will be used often throughout this subsection, we introduce the following terminology:
        
        \begin{notation}
            Let $R$ be a nonzero $\Tn$-local ring.
            We refer to maps $R \to \En(L)$ of $\Tn$-local commutative algebras, where $L$ is algebraically closed, as \emph{geometric points}.
        \end{notation}
        
        Recall that any object $V \in (\ModEn)\dbl$ defines a class in the $\Tnp$-localized $K$-theory spectrum
        \begin{equation*}
            \KTnp(\En) \coloneqq \LTnp \K((\ModEn)\dbl).
        \end{equation*}
        By the chromatic redshift for $\En$~\cite{Yuan-2024-redshift-En}, there exists a geometric point
        \begin{equation*}
            \phi \colon \KTnp(\En) \to \Enp(L),
        \end{equation*}
        for some algebraically closed field $L$.

        \begin{definition}\label{def:increasing-height}
             Let $A$ be a 1-finite space. We define the decategorification map as the composition
             \begin{equation*}
                 \de \colon ((\ModEn)\dblspace)^A \to \KTnp(\En)^{A} \xto{\phi^A} \Enp(L)^A.
             \end{equation*}
        \end{definition}

        \begin{remark}\label{rmrk:dim-categorifiable}
            Let $V \in (\ModEn)\dblA$. Then 
            \begin{equation*}
                \de(V) \colon A \to \Enp(L)
            \end{equation*}
            chooses $\dim(V_a) \in \ZZ \subseteq \pi_0 \Enp(L)$ at each $a \in A$ (where the dimension of $V_a$ is an integer by~\cite[Proposition~10.11]{Mathew-2016-Galois}).
        \end{remark}

        \begin{definition}\label{exm:decategorifying-through-K}
            In particular, for a character $\hchar \colon \BH \to \B\En\units$ we define
            \begin{equation*}
                \de(\hchar) \coloneqq \de(\En{}[\hchar]) \in (\Enp(L)\units)^{\BH}.
            \end{equation*}
            Then $\de(\hchar) \colon \BH \to \Enp(L)\units$ is a pointed map.%, which we think of as a \quotes{decategorified character} of height $\chrHeight+1$.
        \end{definition}

        \begin{lemma}\label{lem:decategorification-of-orientation-induced-characters}
            If $\hchar$ is a $(\SS,\chrHeight)$-orientation character, then $\de(\hchar)$ factors through the $(\tau_{\leq n} \SS,\chrHeight+1)$-orientation $\Sigma\In \to\Inplus \to \Enp(L)\units$.
        \end{lemma}
        \begin{proof}
            By~\cite[Lemma 5.3.23]{Keidar-Ragimov-2025-twisted}, the composition of the suspension of the orientation map with the decategorification
            \begin{equation*}
                \Sigma \In \to \Sigma \En\units \to (\ModEn)\units \xto{\de} \Enp(L)\units
            \end{equation*}
            is a $(\tau_{\le \chrHeight} \SS, \chrHeight+1)$-orientation of $\Enp(L)$. That is, it is a connected cover of an orientation
            \begin{equation*}
                \Inplus \to \Enp(L)\units.
            \end{equation*}
            The claim now follows.
        \end{proof}

        The decategorification behaves well with respect to ($p$-typical) semiadditive integrals.

        \begin{lemma}\label{lem:integral-of-decategorification-of-En}
            Let $A$ be a 1-finite space and $V \in (\ModEn)\dblA$. Then both semiadditive integrals
            \begin{equation*}
                \int^{\Enp(L)}_A \de(V) \in \pi_0 \Enp(L) \qand \int^{\En}_{\Lp A}\chi_V \in \pi_0 \En
            \end{equation*}
            land in $\ZZ_{(p)}$ and are equal.
        \end{lemma}
        \begin{proof}
            As $\Lp$ respects disjoint unions, we can assume $A = \BG$. Let $P \subseteq G$ be a $p$-subgroup so that $\Lp\BP = \L\BP$. By~\cite[Lemma~5.3.10]{Keidar-Ragimov-2025-twisted}
            \begin{equation*}
                \int^{\Enp(L)}_{\BP} \de(V) = \int^{\En}_{\Lp\BP} \chi_V \qin \ZZ.
            \end{equation*}
            The result now follows from~\cref{cor:integral-of-BG-is-linear-combination-along-p-subgroups} and~\cref{cor:integral-of-LBG-is-linear-combination-along-p-subgroups}.
        \end{proof}        

        By~\cref{lem:integral-of-decategorification-of-En}, maps $f \colon A \to \Enp$ that admit a categorification, i.e.\ have a preimage $V_f \in (\ModEn)\dblA$ under the decategorification map, admit a character $\widetilde{\chi}_f \coloneqq \chi_{V_f} \in \En^{\L A}$ (which may depend on the choice of the preimage), satisfying
        \begin{equation*}
            \int^{\Enp}_A f = \int^{\En}_{\Lp A} \widetilde{\chi}_f.
        \end{equation*}
        Assuming $\widetilde{\chi}_f$ is also categorifiable, this process can be iterated.
        
        We define a family of categorifiable functions $A \to \Enp$ that includes the characters of permutation representations and is closed under the operation $\widetilde{\chi}$. In particular, this allows us to define iterated characters of permutation representations via categorification.
        
        \begin{definition}\ 
            \begin{enumerate}
                \item A function $f \colon \B(G \wr \Sm) \to \Enp$ is \emph{a permutation function of height $\chrHeight+1$} if there exists a $d \in \ZZ$ such that $f = \pm \Tm (\BG)^* d$. That is, 
                \begin{equation*}
                    f = \pm d^{\degree} \in \Enp^{\B(G \wr \Sm)}
                \end{equation*}
                where $\Sm$ acts on the coordinates and $G^{\degree}$ acts trivially. 
                \item A function 
                \begin{equation*}
                    f = \bigsqcup_i \bigsqcap_j f_{i,j} \colon  \bigsqcup_i \bigsqcap_j \B(G_{i,j} \wr \Sm[\degree_{i,j}]) \to \Enp
                \end{equation*}
                is a permutation function of height $\chrHeight+1$ if each $f_{i,j}$ is.
                \item Let $A$ be a 1-finite space. A function $f \colon A \to \Enp$ is called a \emph{twisted permutation function} of height $\chrHeight+1$ if it is of the form $f = f' \cdot \hchar$, where $f' \colon A \to \Enp$ is a permutation function, and
                \begin{equation*}
                    \hchar \colon A \to \B \In
                \end{equation*}
                is a map of spaces, regarded as a map to $\Enp$ via the orientation map $\Sigma \In \to \Enp\units$.
            \end{enumerate}
        \end{definition}

        We now note that all twisted permutation functions admit a preimage under $\de$:
        \begin{construction}\ 
            \begin{enumerate}
                \item\label{item:1} Let $f \colon \B(G \wr \Sm) \to \Enp$ be of the form $f = d^{\degree}$ for $d \in \ZZ$. Define 
                \begin{equation*}
                    V_f \coloneqq \begin{cases}
                        \Tm(\BG)^* \En^{\oplus d}, & d \ge 0 \\
                        \Tm(\BG)^* \Sigma \En^{\oplus (-d)}, & d < 0
                    \end{cases}
                    \qin (\ModEn)\dblA[\B(G\wr \Sm)].
                \end{equation*}
                \item\label{item:2} Let $f \colon \B(G\wr \Sm) \to \Enp$ be of the form $f = -f'$, where $f'$ is as above. Then define
                \begin{equation*}
                    V_f \coloneqq \Sigma V_{f'} \qin (\ModEn)\dblA[\B(G \wr \Sm)].
                \end{equation*}
                \item\label{item:5} Let 
                \begin{equation*}
                    f = \bigsqcup_i \bigsqcap_j f_{i,j} \colon \bigsqcup_i \bigsqcap_j \B(G_{i,j} \wr \Sm[\degree_{i,j}]) \to \Enp
                \end{equation*}
                be a permutation function. Define
                \begin{equation*}
                    V_f \coloneqq \bigoplus_i \bigotimes_j V_{f_{i,j}}.
                \end{equation*}
                \item\label{item:3} Let $\hchar \colon A \to \B\In$ be a map of spaces. In particular 
                \begin{equation*}
                    A \xto{\hchar} \B\In \to \B\En\units \to (\ModEn)\units
                \end{equation*}
                is a local system choosing $\En$ at each connected component. Denote this local system by $\En{}[\hchar] \in (\ModEn)\dblA$.
                \item\label{item:4} Let $f \colon A \to \Enp$ be a twisted permutation function. Then $f$ is of the form $f' \cdot \hchar$ where $f'$ is a permutation function and $\hchar \colon A \to \B\In$. Define
                \begin{equation*}
                    V_f \coloneqq V_{f'} \otimes \En{}[\hchar] \qin (\ModEn)\dblA.
                \end{equation*}
                
            \end{enumerate}            
        \end{construction}

        \begin{lemma}
            Let $f \colon A \to \Enp$ be a twisted permutation function. Then, after composition with the map $\Enp^A \to \Enp(L)^A$, $f$ is identified with $\de(V_f)$.
        \end{lemma}
        \begin{proof}
            It is enough to verify the claim for (\labelcref{item:1}), (\labelcref{item:2}) and (\labelcref{item:3}). 
            (\labelcref{item:1}) and (\labelcref{item:2}) are clear. 
            (\labelcref{item:3}) follows from~\cref{lem:decategorification-of-orientation-induced-characters}.
        \end{proof}
    
        In particular, any twisted permutation function $f$ of height $\chrHeight+1$ admits a character $\widetilde{\chi}_f = \chi_{V_f}$. 

        \begin{lemma}
            Let $f$ be a twisted permutation function of height $\chrHeight+1$. Then $\widetilde{\chi}_f$ is a twisted permutation function of height $\chrHeight$.
        \end{lemma}
        \begin{proof}
            Let $f \colon A \to \Enp$ be a twisted permutation function. Therefore $f = f' \cdot \hchar$ where $f'$ is a permutation function and $\hchar \colon A \to \Sigma \In$. Thus $V_f = V_{f'} \otimes \En{}[\hchar]$, and therefore 
            \begin{equation*}
                \widetilde{\chi}_f = \chi_{V_{f'}} \cdot \chi_{\En{}[\hchar]}.
            \end{equation*}
            Since the $\TT$-action on $\dim(V_{f'})$ is trivial (see e.g.~\cite[Corollary~3.2.5]{Keidar-Ragimov-2025-twisted}), by~\cref{cor:character-when-no-T-action}, 
            \begin{equation*}
                \chi_{V_{f'}} \colon \L A \to \En
            \end{equation*}
            is a permutation function of height $\chrHeight$. 
            By~\cref{prop:transgersion-is-charcter-of-character} $\chi_{\En{}[\hchar]} = \tg(\hchar)$. The transgression $\tg(\hchar)$ is the composition
            \begin{equation*}
                \tg(\hchar) \colon \L A \xto{\L\hchar} \L\Sigma \In \simeq \Sigma \In \oplus \In \xto{\eta + 1} \In \to \En\units.
            \end{equation*}
            For each $a \in \pi_0 \L A$, the value of $\tg(\hchar)(a) \in \pi_0\En\units$ lies in the image of $\pi_0\In \to \pi_0\En\units$ and in particular is a $p$-power torsion. As
            \begin{equation*}
                \pi_0\En\units[p^{\infty}] = \Witt(\Fpbar)[\![u_1,\dots,u_{\chrHeight-1}]\!]\units[p^{\infty}] = \begin{cases}
                    \{1\}, & p \neq 2 \\
                    \{\pm 1\}, & p = 2
                \end{cases},
            \end{equation*}
            each connected component lands in $\{\pm 1\} \subseteq \En\units$. In particular, on each connected component, we may assume the transgression is of the form $\pm 1 \cdot \hchar'$, where $\hchar' \colon \B\Omega_a A \to \In$ is pointed, therefore factors through the connected cover map $\hchar' \colon \B\Omega_a A \to \B\In[\chrHeight-1]$. Thus $\tg(\hchar)$ is a twisted permutation function.
        \end{proof}

        \begin{definition}
            Let $f \colon A \to \Enp$ be a twisted permutation function. Define inductively for $0 \le t \le \chrHeight+1$ a twisted permutation function $\widetilde{\chi}^t_f \colon \L^t A \to \En[\chrHeight+1-t]$ as follows:
            \begin{align*}
                & \widetilde{\chi}^0_f \coloneqq f \colon A \to \Enp, \\
                & \widetilde{\chi}^t_f \coloneqq \widetilde{\chi}_{\tilde{\chi}^{t-1}_f} \colon \L^t A \to \En[\chrHeight+1-t].
            \end{align*}
        \end{definition}

        % \begin{corollary}\label{cor:iterated-characters-multiplicative}
        %     Let $f, g \colon A \to \Enp$ be twisted permutation functions. Then for any $0 \le t \le \chrHeight+1$
        %     \begin{equation*}
        %         % \widetilde{\chi}^t_{f + g} = \widetilde{\chi}^t_f + \widetilde{\chi}^t_g, \qquad
        %         % \widetilde{\chi}^t_{-f} = -\widetilde{\chi}^t_f, \qquad
        %         \widetilde{\chi}^t_{f\cdot g} = \widetilde{\chi}^t_f \cdot \widetilde{\chi}^t_g.
        %     \end{equation*}
        % \end{corollary}

        \begin{corollary}\label{cor:iterated-characters-integrals}
            Let $f \colon A \to \Enp$ be a twisted permutation function. Then for any map of $1$-finite spaces $B \to A$, and $0 \le t \le \chrHeight+1$
            \begin{equation*}
                \int^{\Enp}_B f = \int^{\En[\chrHeight+1-t]}_{\Lp^t B} \widetilde{\chi}^t_f.
            \end{equation*}
        \end{corollary}
        \begin{proof}
            This follows by induction from~\cref{lem:integral-of-decategorification-of-En}.
        \end{proof}

        \begin{corollary}\label{cor:iterated-characters-transgression}
            Let $\hchar \colon A \to \Sigma \In \to \Enp\units$ be a map of spaces. Then $\widetilde{\chi}^t_{\hchar} = \tg^t(\hchar)$, i.e.\ it is the composition
            \begin{equation*}
                 \widetilde{\chi}^t_{\hchar} \colon \L^t A \xto{\tg^t \hchar} \In[\chrHeight+1-t] \to \En[\chrHeight+1-t]\units.
            \end{equation*}
        \end{corollary}

    %%%%%%%%%%%%%%%%%%%%%%%%%%%%%%%%%%%%%%%%%%%%%%%%%%%%%%%%%%%%%%%%%%%%%%%%%%%%%%%%
    \subsection{Transchromatic character theory for twisted permutation functions}
    \label{subsec:HKR-and-characters}
    %%%%%%%%%%%%%%%%%%%%%%%%%%%%%%%%%%%%%%%%%%%%%%%%%%%%%%%%%%%%%%%%%%%%%%%%%%%%%%%%

        The transchromatic character theory, developed by Hopkins--Kuhn--Ravenel, Stapleton and Lurie (\cite{Hopkins-Kuhn-Ravenel-2000-HKR,Stapleton-2013-HKR,Lurie-2019-Elliptic3}) provides for any $\pi$-finite space $A$ a character map 
        \begin{equation*}
            \chi^{t, \HKR}_{(-)} \colon \Enp(L)^A \to C^{\chrHeight+1}_{\chrHeight+1-t}(L)^{\Lp^{t} A},
        \end{equation*}
        where $\LKn[\chrHeight+1-t] \Enp(L) \to C^{\chrHeight+1}_{\chrHeight+1-t}(L)$ is the $\Kn[\chrHeight+1-t]$-local splitting algebra at height $\chrHeight+1-t$. By the chromatic Nullstellensatz, there exists a geometric point
        \begin{equation*}
            C^{\chrHeight+1}_{\chrHeight+1-t}(L) \to \En[\chrHeight+1-t](K_t).
        \end{equation*}
        We also denote the composition
        \begin{equation*}
            \Enp(L)^A \xrightarrow{\chi^{t, \HKR}_{(-)}} C^{\chrHeight+1}_{\chrHeight+1-t}(L)^{\Lp^{t} A} \to \En[\chrHeight+1-t](K_t)^{\Lp^{t} A}
        \end{equation*}
        again by $\chi^{t, \HKR}_{(-)}$.
        For $t = 1$ we write $\chi^{\HKR} = \chi^{1, \HKR}$.

        When $f$ admits a preimage $V_f$ under $\de$, we have defined
        \begin{equation*}
            \widetilde{\chi}_f \coloneqq \chi_{V_f} \in \En^{\L A}.
        \end{equation*}
        On the other hand, the transchromatic character provides a map
        \begin{equation*}
            \chi^{\HKR}_f \in \En(K_1)^{\Lp A}.
        \end{equation*}
        We expect that, after composing with a map $\En \to \En(K_1)$ and restricting along $\Lp A \to \L A$, the two notions agree. Although we do not prove this in full generality, we verify that it holds in the cases relevant to our applications.

        % \begin{lemma}\label{lem:finite-wr-Sm-is-good}
        %     Let $G$ be a finite abelian group, then $G \wr \Sm$ is $\En$-good in the sense of~\cite[Definition~7.1]{Hopkins-Kuhn-Ravenel-2000-HKR} for every $\chrHeight$.
        % \end{lemma}
        % \begin{proof}
        %     By~\cite[Proposition~7.2(iii)]{Hopkins-Kuhn-Ravenel-2000-HKR} it is enough to verify its $p$-Sylow subgroup is good. Let $P$ be a $p$-Sylow subgroup of $G$. For any $k$ let $Q_k \subseteq \Sm[p^k]$ be a $p$-Sylow subgroup. Write $\degree = \sum_{k=0}^{M} a_k p^k$ where $0\le a_k < p$. Then $G \wr \Sm$ has a $p$-Sylow subgroup of the form 
        %     \begin{equation*}
        %         P^{\degree} \rtimes (Q_0^{a_0} \times Q_1^{a_1} \times \cdots \times Q_M^{a_M}) \simeq
        %         \bigsqcap_{k=0}^{M} (P^{p^k} \rtimes Q_k)^{a_k} = \bigsqcap_{k=0}^M (P \wr Q_k)^{a_k}.
        %     \end{equation*}
        %     Moreover $Q_k \simeq \ZZ/p \wr \ZZ/p \wr \cdots \wr \ZZ/p$.
        %     As $P \subseteq G$ is abelian, it is good by~\cite[Proposition~7.2(i)]{Hopkins-Kuhn-Ravenel-2000-HKR}, and therefore 
        %     \begin{equation*}
        %         P \wr Q_k \simeq P \wr \ZZ/p \wr \ZZ/p \wr \cdots \wr \ZZ/p
        %     \end{equation*}
        %     is good by~\cite[Theorem~7.3]{Hopkins-Kuhn-Ravenel-2000-HKR}. Finally, the $p$-Sylow subgroup of $G \wr \Sm$ is good by~\cite[Proposition~7.2(ii)]{Hopkins-Kuhn-Ravenel-2000-HKR}.
        % \end{proof}
        
        \begin{lemma}[{\cite[Lemma~5.3.21]{Keidar-Ragimov-2025-twisted}}]\label{lem:HKR-vs-twice-HKR}
             Let $A$ be a $\pi$-finite space and $0 \le t \le s \le \chrHeight+1$. Then for any geometric point $C^{\chrHeight+1}_{\chrHeight+1-t} \to \En[\chrHeight+1-t](K_t)$ there exists a map of $\Kn[\chrHeight+1-s]$-local commutative algebras $C^{\chrHeight+1}_{\chrHeight+1-s} \to C^{\chrHeight+1-t}_{\chrHeight+1-s}(K_t)$ rendering the following diagram commutative
             \begin{equation*}
                \begin{tikzcd}
                    {\Enp^{A}} & {(C^{\chrHeight+1}_{\chrHeight+1-t})^{\Lp^{t}A}} & {\En[\chrHeight+1-t](K_t)^{\Lp^{t}A}} \\ \\
                    {(C^{\chrHeight+1}_{\chrHeight+1-s})^{\Lp^s A}} && {C^{\chrHeight+1-t}_{\chrHeight+1-s}(K_t)^{\Lp^{s}A}.}
                    \arrow["{\chi^{t,\HKR}}", from=1-1, to=1-2]
                    \arrow["{\chi^{s,\HKR}}", from=1-1, to=3-1]
                    \arrow[from=1-2, to=1-3]
                    \arrow["{\chi^{s-t,\HKR}}", from=1-3, to=3-3]
                    \arrow[from=3-1, to=3-3]
                \end{tikzcd}
            \end{equation*}
         \end{lemma}
         \begin{proof}
            By the proof of~\cite[Lemma~5.3.21]{Keidar-Ragimov-2025-twisted}, there exists a map $C^{\chrHeight+1}_{\chrHeight+1-s} \to C^{\chrHeight+1-t}_{\chrHeight+1-s}(K_t)$ of \emph{tempered function spectra} as in~\cite[Notation~4.0.1]{Lurie-2019-Elliptic3}, giving a commutative diagram of tempered function spectra
            \begin{equation*}
                \begin{tikzcd}
                    {(\Enp)_{\mbf{G}^{\mcal{Q}}}} & {(C^{\chrHeight+1}_{\chrHeight+1-t})_{\mbf{G}^{\mcal{Q}}\oplus(\QQ_p/\ZZ_p)^t}} & {{\En[\chrHeight+1-t](K_t)}_{\mbf{G}^{\mcal{Q}}\oplus(\QQ_p/\ZZ_p)^t}} \\
                    {(C^{\chrHeight+1}_{\chrHeight+1-s})_{\mbf{G}^{\mcal{Q}}\oplus(\QQ_p/\ZZ_p)^s}} && {{C^{\chrHeight+1-t}_{\chrHeight+1-s}(K_t)}_{\mbf{G}^{\mcal{Q}}\oplus(\QQ_p/\ZZ_p)^s}}
                    \arrow[from=1-1, to=1-2]
                    \arrow[from=1-1, to=2-1]
                    \arrow[from=1-2, to=1-3]
                    \arrow[from=1-3, to=2-3]
                    \arrow[from=2-1, to=2-3]
                \end{tikzcd}
            \end{equation*}
            See also the proof of~\cite[Lemma~5.3.21]{Keidar-Ragimov-2025-twisted}.
            Applying~\cite[Theorem 4.2.5, Theorem 4.3.2]{Lurie-2019-Elliptic3} gives the desired diagram.
         \end{proof}

        \begin{proposition}\label{lem:chi-tilde-1-is-HKR}
            Let $f \colon A \to \Enp\units$ a twisted permutation function. Then, after composing with the map $\En \to \En(K_{\chrHeight})$,
            \begin{equation*}
                \res{\widetilde{\chi}_{f}}{\Lp A} \simeq \chi^{\HKR}_f \qin \En(K_{\chrHeight})^{\Lp A}.
            \end{equation*}
        \end{proposition}

        \begin{proof}
            The statement is exactly that the image of $V_f$ under both compositions in the diagram below agree
            \begin{equation}\label{eq:monoidal-vs-HKR}
                \begin{tikzcd}
                    {(\ModEn)^{A}} && {\Enp(L)^{A}} \\
                    {\En^{\L A}}&& {C^{\chrHeight+1}_{\chrHeight}(L)^{\Lp A}} \\
                    {\En^{\Lp A}} && {\En(K_{\chrHeight})^{\Lp A}.}
                    \arrow["\de", from=1-1, to=1-3]
                    \arrow["\chi", from=1-1, to=2-1]
                    \arrow[from=2-1, to=3-1]
                    \arrow["{\chi^{\HKR}}", from=1-3, to=2-3]
                    \arrow[from=2-3, to=3-3]
                    \arrow[from=3-1, to=3-3]
                \end{tikzcd}\tag{$\diamondsuit$}
            \end{equation}
            Write $f$ as $f = f' \cdot \hchar$ where $f'$ is a permutation function and $\hchar \colon A \to \B\In$. Then $V_f = V_{f'} \otimes \En{}[\hchar]$. It is enough to verify it for $V_{f'}$ and for $\En{}[\hchar]$ separately. The latter is~\cite[Corollary~5.3.24]{Keidar-Ragimov-2025-twisted}. So it is enough to assume $f$ is a permutation function. We may therefore reduce to the case $A = \B(G\wr \Sm)$ and
            \begin{equation*}
                f = \pm d^{\degree} \colon \B(G \wr \Sm) \to \Enp\units.
            \end{equation*}
            As $(-1)$ categorifies to the constant local system $\Sigma \En$, and both compositions in \labelcref{eq:monoidal-vs-HKR} send $\Sigma \En$ to the constant function $(-1)$, we may further reduce to the case $f = d^{\degree}$.
            Thus, $V_f = \Tm (\BG)^* V$ for some $V \in (\ModEn)\dbl$.
            Consider the projection
            \begin{equation*}
                p \colon \B(G \wr \Sm) \to \B\Sm.
            \end{equation*}
            Then $\Tm(\BG)^* V \simeq p^* \Tm V$. As both $\chi$ and $\chi^{\HKR}$ are functorial with respect to restriction, \labelcref{eq:monoidal-vs-HKR} is given by $p^*$ applied to the diagram
            \begin{equation}\label{eq:monoidal-vs-HKR2}
                \begin{tikzcd}
                    {(\ModEn)^{\B\Sm}} && {\Enp(L)^{\B\Sm}} \\
                    {\En^{\L \B\Sm}}&& {C^{\chrHeight+1}_{\chrHeight}(L)^{\Lp\B\Sm}} \\
                    {\En^{\Lp \B\Sm}} && {\En(K_{\chrHeight})^{\Lp \B\Sm}.}
                    \arrow["\de", from=1-1, to=1-3]
                    \arrow["\chi", from=1-1, to=2-1]
                    \arrow[from=2-1, to=3-1]
                    \arrow["{\chi^{\HKR}}", from=1-3, to=2-3]
                    \arrow[from=2-3, to=3-3]
                    \arrow[from=3-1, to=3-3]
                \end{tikzcd}\tag{$\diamondsuit'$}
            \end{equation}
            It is therefore enough to verify \labelcref{eq:monoidal-vs-HKR2} commutes. 

            The proof is now the same as in~\cite[Lemma~5.3.25]{Keidar-Ragimov-2025-twisted}, using that $\Sm$ is a good group (e.g.\ by~\cite[Theorem~E]{Hopkins-Kuhn-Ravenel-2000-HKR}).
        \end{proof}

        \begin{theorem}\label{thm:chi-tilde-t-is-HKR}
            Let $f \colon A \to \Enp$ be a twisted permutation function. Then for $0 \le t \le \chrHeight+1$, after composition with a geometric point 
            \begin{equation*}
                \En[\chrHeight+1-t] \to \En[\chrHeight+1-t](K_t) \to \En[\chrHeight+1-t](K'_t)
            \end{equation*}
            for some algebraically-closed field $K'_t$,
            \begin{equation*}
                \res{\widetilde{\chi}^{t}_f}{\Lp^{t} A} \simeq \chi^{t,\HKR}_f \qin \En[\chrHeight+1-t](K'_t)^{\Lp^{t} A}.
            \end{equation*}
        \end{theorem}
        \begin{proof}
            We prove this by induction on $t$. For $t = 0$ it is trivial. Assume 
            \begin{equation*}
                \res{\widetilde{\chi}^{t}_f}{\Lp^{t} A} \simeq \chi^{t,\HKR}_f \qin \En[\chrHeight+1-t](K'_{t})^{\Lp^{t} A}.
            \end{equation*}
            By~\cref{lem:chi-tilde-1-is-HKR},~\cref{lem:decomposition-of-wreath}, there exists a geometric point $C^{\chrHeight+1-t}_{\chrHeight-t}(K'_{t}) \to \En[\chrHeight-t](K''_{t+1})$ such that
            \begin{equation*}
                \res{\widetilde{\chi}^{t+1}_f}{\Lp^{t+1} A} \simeq \chi^{\HKR}_{\chi^{t,\HKR}_f} \qin \En[\chrHeight-t](K''_{t+1})^{\Lp^{t+1} A}.
            \end{equation*}
            By~\cref{lem:HKR-vs-twice-HKR}
            \begin{equation*}
                \chi^{\HKR}_{\chi^{t,\HKR}_f} \simeq \chi^{t+1, \HKR}_f.
            \end{equation*}
            By~\cite[Lemma 7.14, Theorem 7.2(2)]{Burklund-Schlank-Yuan-2022-Nullstellensatz} there exists an algebraically-closed field $K'_{t+1}$ and geometric points $\En[\chrHeight-t](K''_{t+1}) \to \En[\chrHeight-t](K'_{t+1})$ and $\En[\chrHeight-t](K_{t+1}) \to \En[\chrHeight-t](K'_{t+1})$, proving the claim.
        \end{proof}

        \begin{remark}\label{rmrk:dim-alt-iterative}
            \cref{thm:chi-tilde-t-is-HKR} gives an alternative proof of the shifted semiadditivity of the transchromatic character (\cite[Remark 7.4.8, Remark~7.4.9]{Lurie-2019-Elliptic3},~\cite{Ben-Moshe-2024-shifted-semiadditive}) for twisted permutation functions, using the usual shifted semiadditivity of the monoidal dimension (\cref{cor:iterated-characters-integrals}).

            In particular, for $V\in (\ModEn)\dbl$, a character $\hchar \colon \B\Sm \to \B\En\units$ factoring through the spherical orientation and $1 \le t \le \chrHeight$
            \begin{equation*}
                \dim(\alt_{\hchar} V) = \int^{C^{\chrHeight}_{\chrHeight+1-t}}_{\Lp^{t-1} \L\B\Sm} \chi^{t-1,\HKR}_{\chi_{\Tm V}} \cdot \tg^t(\overline{\hchar}) \qin \ZZ.
            \end{equation*}
        \end{remark}

        \begin{remark}
            \cref{thm:chi-tilde-t-is-HKR} along with the proof of~\cref{lem:chi-tilde-1-is-HKR}, provide an algorithm for computing the dimension of $\alt_{\hchar} V$ for $V \in (\ModEn)\dbl$. First,
            \begin{equation*}
                \dim(\alt_{\hchar} V) = \int^{\En[0]}_{\Lp^{\chrHeight} \L \B\Sm} \widetilde{\chi}^{\chrHeight+1}_{d^{\degree}} \cdot \tg^{\chrHeight+1}(\overline{\hchar}).
            \end{equation*}
            This semiadditive integral is taken in a rational ring, therefore it is computed as
            \begin{equation*}
                \dim(\alt_{\hchar} V) = \sum_{x \in \pi_0 \Lp^{\chrHeight} \L \B\Sm} \frac{1}{|\Omega_x \Lp^{\chrHeight} \L \B\Sm|} \widetilde{\chi}^{\chrHeight+1}_{d^{\degree}} (x) \cdot \tg^{\chrHeight+1}(\overline{\hchar})(x).
            \end{equation*}
            The character $\widetilde{\chi}^{\chrHeight+1}_{d^{\degree}} $ can be computed using the character of the total power operation~\cite{Barthel-Stapleton-2017-power-operations} which has a closed formula. More concretely, if we consider an element $x \in \pi_0(\Lp^{\chrHeight} \L\B\Sm)$ as a $\Zp^\chrHeight \times \ZZ$-set of size $\degree$. Then, $\chi^{\chrHeight+1}_{d^{\degree}}(x)$ is equal to $d$ raised to the number of orbits of this $\Zp^{\chrHeight} \times \ZZ$-set.

            We do not know of a general closed formula for computing the transgression. However, such a formula exists when $\hchar$ factors through  $(\ZZ,\chrHeight)$-orientation, that is,
            \begin{equation*}
                \hchar \colon \B\Sm \to \B^{\chrHeight+1}\QQ/\ZZ
            \end{equation*}
            and therefore corresponds to a cohomology class $\hchar \in \H^{\chrHeight+1}(\Sm; \QQ/\ZZ)$.
            If $\hchar$ is represented by a cocycle
            \begin{equation*}
                \hchar:\Sm^{\chrHeight+1}\to \QQ/\ZZ
            \end{equation*}
            then, by~\cite[Theorem~3]{Willerton-2008-transgression} its transgression on $(g_1,\dots,g_{\chrHeight}) \in \Csigma^\chrHeight$ is given by the alternating sum:
            \begin{equation*}
                \tg(\hchar)(g_1,\dots,g_{\chrHeight})= \hchar(\sigma,g_1,g_2,\dots,g_{\chrHeight})-\hchar(g_1,\sigma,g_2,\dots,g_\chrHeight)+\cdots+(-1)^{\chrHeight+1} \hchar(g_1,\dots,g_{\chrHeight} ,\sigma)
            \end{equation*}
            This formula can be iterated in order to compute $\tg^{\chrHeight+1}(\hchar)(x)$.
        \end{remark}

        \begin{corollary}\label{cor:power-op-iterative}
            Let $\hchar \colon \B\Sm \to \Sigma \In$ be a map of pointed spaces. Let $H \to \Sm$ be a homomorphism of finite groups and $d \in \ZZ \subseteq \pi_0 \En$. Then for $1\le t \le \chrHeight+1$:
            \begin{equation*}
                \beta^{\degree}_{H, \hchar}(d) = \begin{cases}
                    \int\limits^{C^{\chrHeight}_{\chrHeight+1-t}}_{\Lp^t \BH} \chi^{t-1,\HKR}_{\chi_{\Tm \En^{\oplus d}}} \cdot \tg^t(\overline{\hchar}), & d \ge 0 \\
                     & \\
                    \int\limits^{C^{\chrHeight}_{\chrHeight+1-t}}_{\Lp^t \BH} \chi^{t-1,\HKR}_{\chi_{\Tm \Sigma \En^{\oplus (-d)}}} \cdot \tg^t(\overline{\hchar}), & d < 0
                \end{cases}
            \end{equation*}
        \end{corollary}
        
        \begin{proof}
            Define the twisted permutation function
            \begin{equation*}
                f_d \coloneqq d^{\degree} \cdot \overline{\hchar} \colon \B\Sm \to \Enp.
            \end{equation*}
            Recall that
            \begin{equation*}
                \beta^{\degree}_{H,\hchar}(d) = \int^{\Enp}_{\BH} f_d.
            \end{equation*}
            The result now follows by unraveling the categorification 
            \begin{equation*}
                V_{f_d} = \begin{cases}
                    \En^{\oplus d} \otimes \En{}[\hchar], & d \ge 0 \\
                    \Sigma \En^{\oplus (-d)} \otimes \En{}[\hchar], & d < 0.
                \end{cases}
            \end{equation*}
        \end{proof}

        \begin{remark}
            In particular, if $P \subseteq \Sm$ is a $p$-subgroup and $\hchar \colon \B\Sm \to \Sigma \In$ is a pointed map of spaces, then for $d \in \ZZ$
            \begin{equation*}
                \beta^{\degree}_{P, \hchar}(d) = \begin{cases}
                    \dim(\alt_{P,\hchar} \En^{\oplus d}), & d \ge 0, \\
                    \dim(\alt_{P,\hchar} \Sigma \En^{\oplus (-d)}), & d < 0,
                \end{cases}
            \end{equation*}
            and the iterative formulas of~\cref{cor:power-op-iterative} and~\cref{rmrk:dim-alt-iterative} agree.
        \end{remark}

%%%%%%%%%%%%%%%%%%%%%%%%%%%%%%%%%%%%%%%%%%%%%%%%%%%%%%%%%%%%%%%%%%%%%%%%%%%%%%%%
\section{Computations in height 1}
\label{sec:computations}
%%%%%%%%%%%%%%%%%%%%%%%%%%%%%%%%%%%%%%%%%%%%%%%%%%%%%%%%%%%%%%%%%%%%%%%%%%%%%%%%

    In this section, we give a concrete computation in height~1, computing the dimensions of alternating powers in $\ModEn[1]$ and $\Mod_{\sVect}$ corresponding to $(\Fq{2},1)$-orientation characters.
    
    Recall that a primitive height 1 minus one, or equivalently an $(\Fq{2},1)$-orientation, is a non-degenerate map of connective spectra
    \begin{equation*}
        (-1)^{(1)} \colon \Sigma \ZZ/2 \simeq \ndual[1]{\Fq{2}} \to \ounit_{\cC}\units.
    \end{equation*}
    An $(\Fq{2},1)$-orientation character of $\Sm$ is a map
    \begin{equation*}
        \hchar \colon \B\Sm \to \Sigma^2 \ZZ/2,
    \end{equation*}
    or equivalently a cohomology class $\hchar \in \H^2(\Sm; \ZZ/2)$.
    
    Given a height 1 minus one and a cohomology class $\hchar \in \H^2(\Sm; \ZZ/2)$ we have constructed the corresponding alternating power functor $\alt_{\hchar}$.
    Recall the following result of Schur:
    \begin{theorem}[\cite{Schur-1911-alternating-groups}]
        Let $\degree > 1$. Then
        \begin{equation*}
            \H^2(\Sm; \QQ/\ZZ) = \begin{cases}
                0, & m \le 3 \\
                \ZZ/2, & m \ge 4
            \end{cases}, \qquad
            \H^2(\Sm; \ZZ/2) = \begin{cases}
                \ZZ/2, & m \le 3 \\
                (\ZZ/2)^2, & m \ge 4
            \end{cases},
        \end{equation*}
        and the map $\H^2(\Sm; \ZZ/2)\to \H^2(\Sm; \QQ/\ZZ)$ is surjective.
    \end{theorem}

    \begin{definition}\label{def:sgn^(1)}
        Let $\degree \ge 4$. Let $\sgn^{(1)} \colon \B\Sm \to \B^2\QQ/\ZZ$ be the unique non-trivial map.
        By Schur's theorem, we know it factors (non-uniquely) through a map, which we also denote
        \begin{equation*}
            \sgn^{(1)} \colon \B\Sm \to \B^2 \ZZ/2.
        \end{equation*}
    \end{definition}

    \begin{remark}
        As we are interested in the character $\B\Sm \to \B^2\Ct \to \B\En[1]\units$, which factors through $\B^2\QQ/\ZZ$, the choice of the decomposition through $\B^2\ZZ/2$ does not matter.
    \end{remark}

    % \begin{definition}
    %     :et $\cC$ be $(\Fq{2},1)$-oriented category. For a choice of $\sgn^{(1)} \colon \B\Sm \to \B^2 \ZZ/2$ we denote the corresponding alternating power functor by
    %     \begin{equation*}
    %         \alt_{(1)} \coloneqq \alt_{\sgn^{(1)}} \colon \cC \to \cC.
    %     \end{equation*}
    % \end{definition}

    Our goal is to calculate $\dim(\alt_{\sgn^{(1)}} V)$ for $V$ dualizable in the chromatic case. Note that as $\sgn^{(1)}$ factors through $\Sigma^2\ZZ/2$, we have $\overline{\sgn^{(1)}} \simeq \sgn^{(1)}$, and therefore
    \begin{equation*}
        \alt_{\sgn^{(1)}}V = (V\om \otimes \En[1][\sgn^{(1)}])^{h\Sm}.
    \end{equation*}

    In height 1, the algebra $\En[1][\B\ZZ/2]$ plays a role analogous to that of the regular representation: Recall (\cite[Proposition 4.5]{CSY-cyclotomic}) that there exists a $\B\ZZ/2$-equivariant idempotent $\varepsilon \in \En[1][\B \ZZ/2]$ such that $\En[1][\varepsilon^{-1}] \simeq \En[1]$. The complement is called a height 1 cyclotomic extension of $\En$ and is denoted
    \begin{equation*}
        \En[1][(-1)^{(1)}] \coloneqq \En[1][(1-\varepsilon)^{-1}].
    \end{equation*}
    By~\cite[Proposition 5.2]{CSY-cyclotomic}, $\En[1][(-1)^{(1)}]$ is an $e = (\ZZ/2)\units$-Galois extension of $\En[1]$. That is, $\En[1][(-1)^{(1)}] \simeq \En[1]$ as commutative algebras, but admits a non-trivial $\B\ZZ/2$-action. Restricting along $\sgn^{(1)} \colon \B\Sm \to \B^2 \ZZ/2$, we get a $\Sm$-equivariant decomposition
    \begin{equation*}
        \En[1][\B \ZZ/2] \simeq \En[1][\triv] \oplus \En[1][\sgn^{(1)}].
    \end{equation*}
    Therefore, in order to compute $\alt_{\sgn^{(1)}} V$ for $V\in \ModEn[1]$, we can replace $\En[1][\sgn^{(1)}]$ by $\En[1][\B \ZZ/2]$:
    \begin{equation*}
        (V\om \otimes \En[1][\B \ZZ/2])^{h\Sm} \simeq \Sym V \oplus \alt_{\sgn^{(1)}} V.
    \end{equation*}
    When $V = \En[1][X]$ is a free module,
    \begin{equation*}
        (V\om \otimes \En[1][\B \ZZ/2])^{h\Sm} \simeq (\En[1][X^\degree \times \B\ZZ/2])_{h\Sm} \simeq \En[1][(X^{\degree} \times \B\ZZ/2)_{h\Sm}].
    \end{equation*}
    As $\En[1][-] \colon \Span(\spcpi) \to \ModEn[1]$ is symmetric monoidal, when $X$ is $\pi$-finite, $\En[1][X]$ is dualizable and we can compute the dimension of the above module as the dimension of $(X^{\degree} \times \B\ZZ/2)_{h\Sm}$ in $\Span(\spcpi)$.

    \begin{definition}
        Let $\Sm^+$ be the $\ZZ/2$-central extension of $\Sm$ corresponding to $\sgn^{(1)}$, i.e.\footnote{Note again that the group $\Sm^+$ depends on the choice of factorization $\B\Sm \to \B^2\ZZ/2$. However, the computations are the same for both choices.}
        \begin{equation*}
            \B\Sm^+ = \fib(\B\Sm \xto{\sgn^{(1)}} \B^2 \ZZ/2).
        \end{equation*}
    \end{definition}

    \begin{lemma}\label{lem:Sm+-orbits}
        Let $X$ be a space with $\Sm$-action. Then, restricting along $\Sm^+ \onto \Sm$,
        \begin{equation*}
            X_{h\Sm^+} \simeq (X \times \B\ZZ/2)_{h\Sm}.
        \end{equation*}
    \end{lemma}

    \begin{proof}
        Equipping $\B\ZZ/2$ with the transitive $\B\ZZ/2$-action, the space $X \times \B\ZZ/2$ admits a $(\Sm \times \B\ZZ/2)$-action and
        \begin{equation*}
            (X \times \B\ZZ/2)_{h(\Sm \times \B\ZZ/2)} \simeq X_{h\Sm}.
        \end{equation*}
        Moreover, the natural map
        \begin{equation*}
            X_{h\Sm} \to \B\Sm \times \B^2\ZZ/2,
        \end{equation*}
        induced as the orbits of $X \times \B\ZZ/ 2 \to \pt$, is trivial on the second coordinate.
    
        By unstraightening, we get a pullback square
        \begin{equation*}
            \begin{tikzcd}
                {(X \times \B \ZZ/2)_{h\Sm}} & {X_{h\Sm}} \\
                {\B\Sm} & {\B\Sm \times \B^2 \ZZ/2.}
                \arrow[from=1-2, to=2-2]
                \arrow[from=1-1, to=1-2]
                \arrow[from=1-1, to=2-1]
                \arrow["{(\id,\sgn^{(1)})}", from=2-1, to=2-2]
                \arrow["\lrcorner"{anchor=center, pos=0.125}, draw=none, from=1-1, to=2-2]
            \end{tikzcd}
        \end{equation*}
        The map $(\id, \sgn^{(1)}) \colon \B\Sm \to \B\Sm \times \B^2 \ZZ/2$ factors through the diagonal $\Delta \colon \B\Sm \to \B\Sm \times \B\Sm$. Together with the product of pullbacks
        \begin{equation*}
            \begin{tikzcd}
                {X_{h\Sm}} & {X_{h\Sm}} \\
                {\B\Sm} & {\B\Sm}
                \arrow[from=1-2, to=2-2]
                \arrow["{\id}", from=1-1, to=1-2]
                \arrow[from=1-1, to=2-1]
                \arrow["{\id}", from=2-1, to=2-2]
                \arrow["\lrcorner"{anchor=center, pos=0.125}, draw=none, from=1-1, to=2-2]
            \end{tikzcd}
            \qand
            \begin{tikzcd}
                {\B\Sm^+} & {\pt} \\
                {\B\Sm} & {\B^2 \ZZ/2,}
                \arrow[from=1-2, to=2-2]
                \arrow[from=1-1, to=1-2]
                \arrow[from=1-1, to=2-1]
                \arrow["{\sgn^{(1)}}"', from=2-1, to=2-2]
                \arrow["\lrcorner"{anchor=center, pos=0.125}, draw=none, from=1-1, to=2-2]
            \end{tikzcd}
        \end{equation*}
        we get the pullback diagram
        \begin{equation*}
            \begin{tikzcd}
                {(X \times \B\ZZ/2)_{h\Sm}} & {X_{h\Sm} \times \B\Sm^+} & {X_{h\Sm}} \\
                {\B\Sm} & {\B\Sm \times \B\Sm} & {\B\Sm \times \B^2 \ZZ/2.}
                \arrow[from=1-1, to=1-2]
                \arrow[from=1-1, to=2-1]
                \arrow["\lrcorner"{anchor=center, pos=0.125}, draw=none, from=1-1, to=2-2]
                \arrow[from=1-2, to=1-3]
                \arrow[from=1-2, to=2-2]
                \arrow["\lrcorner"{anchor=center, pos=0.125}, draw=none, from=1-2, to=2-3]
                \arrow[from=1-3, to=2-3]
                \arrow["\Delta", from=2-1, to=2-2]
                \arrow["{\id \times \sgn^{(1)}}", from=2-2, to=2-3]
            \end{tikzcd}
        \end{equation*}
        The left hand side pullback square, is a pullback of a product of maps along the diagonal, and therefore is equivalent to the pullback square
        \begin{equation*}
            \begin{tikzcd}
                {(X \times \B\ZZ/2)_{h\Sm}} & {X_{h\Sm}} \\
                {\B\Sm^+} & {\B\Sm,}
                \arrow[from=1-2, to=2-2]
                \arrow[from=1-1, to=1-2]
                \arrow[from=1-1, to=2-1]
                \arrow[from=2-1, to=2-2]
                \arrow["\lrcorner"{anchor=center, pos=0.125}, draw=none, from=1-1, to=2-2]
            \end{tikzcd}
        \end{equation*}
        and by unstraightening, $X_{h\Sm^+} \simeq (X \times \B\ZZ/2)_{h\Sm}$.
    \end{proof}

    \begin{remark}
        In particular, for $X = \pt$, we get the well-known formula for group extensions
        \begin{equation*}
            (\B\ZZ/2)_{h\Sm} \simeq \B\Sm^+.
        \end{equation*}
    \end{remark}
    
    \begin{lemma}\label{lem:E1-of-smash-of-wedge-of-circles}
        Let $d \ge 0$ and $X=(\bigvee_d S^1)^{\wedge m}$. Then
        \begin{align*}
            & \En[1][X_{h\Sm}] \simeq \Sym \En[1] \oplus \Sym(\Sigma \En[1]^{\oplus d}), \\
            & \En[1][X_{h\Sm^+}] \simeq \Sym \En[1] \oplus \alt_{\sgn^{(1)}} \En[1] \oplus \Sym(\Sigma \En[1]^{\oplus d})\oplus \alt_{\sgn^{(1)}}(\Sigma \En[1]^{\oplus d}).
        \end{align*}
    \end{lemma}
    
    \begin{proof}
        By~\cref{lem:Sm+-orbits}
        \begin{equation*}
            \En[1][X_{h\Sm^+}] \simeq \En[1][(X \times \B\ZZ/2)_{h\Sm}] \simeq \En[1][X \times \B\ZZ/2]_{h\Sm}.
        \end{equation*}
        As before, this splits as
        \begin{equation*}
            \En[1][X_{h\Sm^+}] \simeq (\En[1][X] \otimes (\En[1] \oplus \En[1][\sgn^{(1)}]))_{h\Sm} \simeq (\En[1][X])_{h\Sm} \oplus (\En[1][X] \otimes \En[1][\sgn^{(1)}])_{h\Sm}.
        \end{equation*}
        Write $\redEn[1][-]$ for the reduced homology functor, that is
        \begin{equation*}
            \redEn[1][X]  = \cofib(\En[1] \to \En[1][X]).
        \end{equation*}
        The base point of $X$ is an equivariant section so that $\En[1][X] \simeq \redEn[1][X] \oplus \En[1]$ in $\Sp^{\B\Sm}$.
        The reduced homology functor is symmetric monoidal with respect to the smash product, therefore
        \begin{equation*}
            \En[1][X] \simeq \redEn[1][X] \oplus \En[1] \simeq \redEn[1][\bigvee_d S^1]\om \oplus \En[1] \simeq (\Sigma \En[1]^{\oplus d})\om \oplus \En[1].
        \end{equation*}
        Thus
        \begin{equation*}
            \En[1][X_{h\Sm}] \simeq \En[1][X]_{h\Sm} \simeq (\Sigma \En[1]^{\oplus d})\om_{h\Sm} \oplus (\En[1])_{h\Sm} = \Sym(\Sigma \En[1]^{\oplus d}) \oplus \Sym \En[1],
        \end{equation*}
        and
        \begin{equation*}
            \begin{split}
                (\En[1][X] \otimes \En[1][\sgn^{(1)}])_{h\Sm}  
                & = ((\Sigma \En[1]^{\oplus d})\om \otimes \En[1][\sgn^{(1)}])_{h\Sm} \oplus (\En[1][\sgn^{(1)}])_{h\Sm} \\
                & = \alt_{\sgn^{(1)}} (\Sigma \En[1]^{\oplus d}) \oplus \alt_{\sgn^{(1)}} \En[1].
            \end{split}
        \end{equation*}
    \end{proof}

    \begin{remark}
        Note that the discussion so far applies to any $(\Fq{2},1)$-oriented category, such as $\SpKn[1]$ or $\Mod_{\Vect}$.
    \end{remark}
    
    \begin{theorem}\cite[Corollary 4.3.4, Corollary 4.7.5]{Lurie-2019-Elliptic3}\label{thm:intertia-groupoid}
        Let $G$ be a finite group and $X$ be a compact genuine $G$-space. Then  
        \begin{equation*}
            \dim(\En[1][X_{hG}]) = \dim((\En[1] \otimes \QQ)[(\bigsqcup_{g \in G^{(2)}} X^{\langle g \rangle})_{hG}],
        \end{equation*} 
        where $G^{(2)}$ is the subset of $2$-power torsion elements in $G$ and the $G$-action on $\bigsqcup_{g \in G^{(2)}} X^{\langle g \rangle}$ is given by
        \begin{equation*}
            h.(g,x) = (hgh^{-1}, h.x).
        \end{equation*}
    \end{theorem}
    
    \begin{proof}
        Applying global sections to~\cite[Corollary 4.3.4, Corollary 4.7.5]{Lurie-2019-Elliptic3} to $A=\En[1]$ with the Quillen $2$-divisible group, the finite group $G$, the colattice $\hat{\Lambda} = \ZZ_p$ and choosing $B$ as the splitting algebra, gives a base change formula 
        \begin{equation*}
            \En[1]^{X_{h G}} \otimes B \simeq B^{(\bigsqcup_{g \in G^{(2)}}X^{\langle g \rangle})_{hG}}.
        \end{equation*}
        By ambidexterity
        \begin{equation*}
            \En[1][X_{h G}] \otimes B \simeq B[(\bigsqcup_{g \in G^{(2)}}X^{\langle g \rangle})_{hG}].
        \end{equation*}
        Using the fact that the splitting algebra is faithfully flat over $\En[1] \otimes \QQ$ we get 
        \begin{equation*}
            \dim(B[(\bigsqcup_{g \in G^{(2)}}X^{\langle g \rangle})_{hG}]) = \dim ((\En[1] \otimes \QQ)[(\bigsqcup_{g \in G^{(2)}}X^{\langle g \rangle})_{hG}]).
        \end{equation*}
        Finally, the claim follows as the symmetric monoidal functor $B \otimes -$ preserves dimensions and since $\pi_0\En[1]$ is torsion-free.
    \end{proof}

    \begin{notation}
        We denote the image of $1 \in \ZZ/2$ under the central homomorphism $\ZZ/2 \to \Sm^+$  by $z \in \Z(\Sm^+)$. We denote the preimages of $\sigma \in \Sm$ under the map $\Sm^+ \to \Sm$ by $\overline{\sigma}$ and $z\overline{\sigma}$.
    \end{notation}

    \begin{definition}
        Let $\sigma \in \Sm$. Define $\kappa_{\sigma} \colon \Csigma \to \{\pm 1\}$ by
        \begin{equation*}
            \kappa_{\sigma}(\tau) = \begin{cases}
                -1, & [\overline{\sigma}, \overline{\tau}] = z \\
                1, & [\overline{\sigma}, \overline{\tau}] = e.
            \end{cases}
        \end{equation*}
    \end{definition}
    
    \begin{lemma}\label{lem:kappa}
        The map $\kappa_{\sigma} \colon \Csigma \to \{\pm 1\}$ is well defined group homomorphism satisfying
        \begin{enumerate}[label=(\arabic*)]
            \item\label{item:kappa-hom} its kernel is the image of $\cent_{\Sm^+}(\overline{\sigma})$ in $\cent_{\Sm}(\sigma)$, and $\kappa_{\sigma}$ is not trivial if and only if $\overline{\sigma} \sim z\overline{\sigma}$.
            \item\label{item:kappa-cycle} If $c$ is a $k$-cycle of $\sigma$, then $\kappa_{\sigma}(c) = -1$ if and only if $k$ is even and $\sigma$ is even.
            \item\label{item:kappa-swap} If $c = (a_1 \cdots a_k)$ and $c' = (b_1 \cdots b_k)$ are two $k$-cycles of $\sigma$ and $w = (a_1\,b_1)\cdots(a_k\,b_k)$, then $w \in \Csigma$, and $\kappa_{\sigma}(w) = -1$ if and only if $k$ is even, or $k$ and $\sigma$ are both odd.
        \end{enumerate}
    \end{lemma}

    \begin{proof}
        As $z$ is central, $\kappa_{\sigma}(\tau)$ does not depend on the choice of the lifts, and $\kappa_{\sigma}$ is a homomorphism. 
        \begin{itemize}[leftmargin=*, align=left]
            \item[\labelcref{item:kappa-hom}] Its kernel consists of the $\tau \in \Csigma$ whose lifts commute with $\overline{\sigma}$, so it is the image of $\cent_{\Sm^+}(\overline{\sigma})$ in $\cent_{\Sm}(\sigma)$.
            Moreover, an element $\overline{\tau} \in \Sm^+$ conjugates $\overline{\sigma}$ to $z\overline{\sigma}$ if and only if its image $\tau \in \Csigma$ centralizes $\sigma$ and $\kappa_{\sigma}(\tau) = -1$. So $\kappa_{\sigma} \neq 1$ if and only if $\overline{\sigma} \sim z\overline{\sigma}$. 
        \end{itemize}

        Recall that lifts of disjoint transpositions anticommute in $\Sm^+$ (\cite{Schur-1911-alternating-groups}, see also~\cite{Bessenrodt-1994-covering-of-symmetric-groups}). Writing permutations as products of transpositions, it follows that for $a, b \in \Sm$ with disjoint supports, $[\overline{a},\overline{b}] = z$ if $a$ and $b$ are both odd, and $[\overline{a},\overline{b}] = e$ otherwise.

        \begin{itemize}[leftmargin=*, align=left]
            \item[\labelcref{item:kappa-cycle}] Write $\sigma = c\sigma'$ with $\sigma'$ disjoint from $c$.
            Then 
            \begin{equation*}
                [\overline{c}, \overline{\sigma}] = \overline{c} (\overline{c}\barprime{\sigma} \overline{c}^{-1}) \overline{\sigma}^{-1} = \begin{cases}
                    z, & \text{$c$ and $\sigma'$ are both odd}, \\
                    e, & \text{otherwise.}
                \end{cases}
            \end{equation*}
            The claim follows since $c$ is odd if and only if $k$ is even, and in this case $\sigma'$ is odd if and only if $\sigma$ is even.

            \item[\labelcref{item:kappa-swap}] Write $\sigma = cc'\sigma''$ with $\sigma''$ disjoint from $c$ and $c'$. Then $wcw^{-1} = c'$ and $w \in \Csigma$. Choosing the lift $\barprime{c} = \overline{w}\,\overline{c}\,\overline{w}^{-1}$ and using that $\overline{w}^{\,2} \in \{1, z\}$ is central, we get $\overline{w}\,\barprime{c}\,\overline{w}^{-1} = \overline{c}$, and hence $\overline{w}\,(\overline{c}\,\barprime{c})\,\overline{w}^{-1} = \barprime{c}\,\overline{c}$, which is $z\,\overline{c}\,\barprime{c}$ if $k$ is even and $\overline{c}\,\barprime{c}$ if $k$ is odd.Moreover, as $w$ is odd if and only if $k$ is odd, and $\sigma''$ has the same parity as $\sigma$, we have $\overline{w}\,\barpprime{\sigma}\,\overline{w}^{-1} = z\,\barpprime{\sigma}$ if $k$ and $\sigma$ are both odd, and $\overline{w}\,\barpprime{\sigma}\,\overline{w}^{-1} = \barpprime{\sigma}$ otherwise. The claim follows.

        \end{itemize}

    \end{proof}

    \begin{proposition}\label{prop:dim-of-alternating-in-height-1-ge0}
        Let $V\in (\ModEn[1])\dbl$ with $\dim(V) = d \ge 0$. Then
        \begin{equation*}
            \dim \alt_{\sgn^{(1)}} V = {d+\degree-1 \choose \degree} + \sum_{\substack{[\sigma]\in \Sm^{(2)} / \conj \\ \sgn(\sigma) = -1}} \bigsqcap_{k=1}^{\degree} {d \choose \Nk}.
        \end{equation*}
    \end{proposition}

    \begin{proof}
        Since the $\TT$-action on $\dim(V)$ is trivial (see e.g.~\cite[Corollary~3.2.5]{Keidar-Ragimov-2025-twisted}), by~\cref{cor:dim-of-alternating-depends-only-on-dim}
        \begin{equation*}
            \dim \alt_{\sgn^{(1)}} V = \dim \alt_{\sgn^{(1)}} \En[1]^{\oplus d}.
        \end{equation*}
        We may therefore assume $V = \En[1]^{\oplus d}$. Let $[d]$ be a set with $d$ elements, so that $[d]^{\degree}$ is a (Borel) $\Sm^+$-space, where $\Sm^+$ acts through $\Sm$. By~\cref{lem:Sm+-orbits}
        \begin{equation*}
            \En[1][[d]^{\degree}_{h\Sm^+}] \simeq ((\En[1]^{\oplus d})\om \otimes \En[1][\B\ZZ/2])^{h\Sm} \simeq \Sym (\En[1]^{\oplus d}) \oplus \alt_{\sgn^{(1)}} (\En[1]^{\oplus d}).
        \end{equation*}
        By~\cref{thm:intertia-groupoid}, and since finite groups have no rational homology,
        \begin{equation*}
            \begin{split}
                \dim(\En[1][[d]^{\degree}_{h\Sm^+}])
                & = \dim((\En[1] \otimes \QQ)[(\bigsqcup_{\overline{\sigma} \in (\Sm^+)^{(2)}}([d]^{\degree})^{\langle \overline{\sigma} \rangle})_{h\Sm^+}]) \\
                & = \sum_{[\overline{\sigma}]\in (\Sm^+)^{(2)} / \conj} \dim((\En[1] \otimes \QQ) [(([d]^{\degree})^{\langle \sigma \rangle})_{h\cent_{\Sm^+}(\overline{\sigma})}]) \\
                & = \sum_{[\overline{\sigma}]\in (\Sm^+)^{(2)} / \conj} |([d]^{\degree})^{\langle \sigma \rangle} / \cent_{\Sm^+}(\overline{\sigma})|,
            \end{split}
        \end{equation*}
        where we used that $z$ acts trivially on $[d]^{\degree}$. Note that $\cent_{\Sm^+}(\overline{\sigma})$ acts through its image in $\Csigma$, which by~\cref{lem:kappa} is the subgroup $\ker(\kappa_{\sigma})$, of index $2$ if $\overline{\sigma} \sim z\overline{\sigma}$, and all of $\Csigma$ if $\overline{\sigma} \not\sim z\overline{\sigma}$. 

        To compare with $\Sm$, let $[\sigma] \in \Sm^{(2)}/\conj$. As $(\Sm^+)^{(2)}$ is the preimage of $\Sm^{(2)}$, the classes in $(\Sm^+)^{(2)}/\conj$ lying over $[\sigma]$ are $[\overline{\sigma}]$ and $[z\overline{\sigma}]$, which coincide if and only if $\overline{\sigma} \sim z\overline{\sigma}$.
        
        If $\overline{\sigma} \not\sim z\overline{\sigma}$, then $\kappa_{\sigma} = 1$, and each of the two classes contributes $|([d]^{\degree})^{\langle \sigma \rangle} / \Csigma|$.
        
        If $\overline{\sigma} \sim z\overline{\sigma}$, the single class contributes the number of $\ker(\kappa_{\sigma})$-orbits of $([d]^{\degree})^{\langle \sigma \rangle}$. 
        As $\ker(\kappa_{\sigma}) \subseteq \Csigma$ is a normal subgroup of index $2$, the $\Csigma$-orbit of $x$ is a single $\ker(\kappa_{\sigma})$-orbit if $\kappa_{\sigma}$ is not trivial on $\mrm{Stab}_{\Csigma}(x)$, and it splits into two $\ker(\kappa_{\sigma})$-orbits otherwise. 
        In both cases, the classes lying over $[\sigma]$ contribute 
        \begin{equation*}
            |([d]^{\degree})^{\langle \sigma \rangle} / \Csigma| + \bigg|\bigg\{\begin{array}{cc} [x] \in ([d]^{\degree})^{\langle \sigma \rangle} / \Csigma \\ \kappa_{\sigma}(\mrm{Stab}_{\Csigma}(x)) = 1 \end{array}\bigg\}\bigg|.
        \end{equation*}
        Therefore,
        \begin{equation*}
            \dim(\En[1][[d]^{\degree}_{h\Sm^+}]) = \sum_{[\sigma]\in \Sm^{(2)} / \conj} \bigg(|([d]^{\degree})^{\langle \sigma \rangle} / \Csigma| + \bigg|\bigg\{\begin{array}{cc} [x] \in ([d]^{\degree})^{\langle \sigma \rangle} / \Csigma \\ \kappa_{\sigma}(\mrm{Stab}_{\Csigma}(x)) = 1 \end{array}\bigg\}\bigg|\bigg).
        \end{equation*}
        Similarly, by~\cref{thm:intertia-groupoid},
        \begin{equation*}
            \dim \Sym (\En[1]^{\oplus d}) = \dim(\En[1][[d]^{\degree}_{h\Sm}]) = \sum_{[\sigma] \in \Sm^{(2)}/\conj} |([d]^{\degree})^{\langle \sigma \rangle} / \Csigma|,
        \end{equation*}
        and therefore
        \begin{equation*}
            \dim \alt_{\sgn^{(1)}} V = \sum_{[\sigma]\in \Sm^{(2)} / \conj} \bigg|\bigg\{\begin{array}{cc} [x] \in ([d]^{\degree})^{\langle \sigma \rangle} / \Csigma \\ \kappa_{\sigma}(\mrm{Stab}_{\Csigma}(x)) = 1 \end{array}\bigg\}\bigg|.
        \end{equation*}

        It remains to compute the summands. 
        An element $x \in [d]^{\degree}$ is fixed by $\sigma$ if and only if it is constant on the support of each cycle $c$ of $\sigma$. We denote this value by $x_c \in [d]$. This identifies $([d]^{\degree})^{\langle \sigma \rangle}$ with $[d]^{\numcyc}$.
        By~\cref{lem:decomposition-of-wreath}, $\Csigma = \bigsqcap_k (\ZZ/k \wr \Sm[\Nk])$.
        The $\ZZ/k$ factors, are generated by the $k$-cycles of $\sigma$, and act trivially on $([d]^{\degree})^{\langle \sigma \rangle}$. 
        The $\Sm[\Nk]$ factors, which permute the $k$-cycles of $\sigma$, permute the coordinates of $([d]^{\degree})^{\langle \sigma \rangle} = \bigsqcap_k [d]^{\Nk}$.
        Hence $\mrm{Stab}_{\Csigma}(x)$ is generated by the cycles of $\sigma$ and by the elements $w = (a_1\,b_1)\cdots(a_k\,b_k)$ of~\cref{lem:kappa} exchanging two $k$-cycles $c = (a_1 \cdots a_k)$ and $c' = (b_1 \cdots b_k)$ of $\sigma$ with $x_c = x_{c'}$. Using the values of $\kappa_{\sigma}$ on these generators given in~\cref{lem:kappa}, we can now compute the summands:
        \begin{enumerate}
            \item If $\sigma = e$, then $\kappa_{\sigma} = 1$, and the summand is $|[d]^{\degree} / \Sm| = {d+\degree-1 \choose \degree}$.
            \item If $\sigma \neq e$ is even, then, as the cycle lengths of $\sigma$ are powers of $2$, $\sigma$ has a cycle $c$ of even length, so $\kappa_{\sigma}(c) = -1$ with $c \in \mrm{Stab}_{\Csigma}(x)$ for every $x$, and the summand vanishes.
            \item If $\sigma$ is odd, then by~\cref{lem:kappa}, $\kappa_{\sigma}(c) = 1$ for every cycle $c$ of $\sigma$, and $\kappa_{\sigma}(w) = -1$ for every $w$ as in~\cref{lem:kappa}\labelcref{item:kappa-swap}. As $\kappa_{\sigma}$ is a homomorphism, it is trivial on $\mrm{Stab}_{\Csigma}(x)$ if and only if none of these elements $w$ lies in $\mrm{Stab}_{\Csigma}(x)$. 
            That is, if and only if $x_c \neq x_{c'}$ for any two distinct cycles $c, c'$ of $\sigma$ of the same length. 
            As $\Csigma$ acts on $\bigsqcap_k [d]^{\Nk}$ through $\bigsqcap_k \Sm[\Nk]$, the orbit of such an $x$ is determined by the sets $\{x_c \mid \text{$c$ is a $k$-cycle of $\sigma$}\}$. These are subsets of $[d]$ of size $\Nk$, and every family of such subsets occurs, so the summand is $\bigsqcap_{k=1}^{\degree} {d \choose \Nk}$.\qedhere
        \end{enumerate}
    \end{proof}

        \begin{lemma}\label{lem:orbit-count}
        Let $\sigma \in \Sm$ with $\overline{\sigma} \sim z\overline{\sigma}$, let $Y$ be a finite $\Csigma$-set, and let $\psi \colon \Csigma \to \{\pm 1\}$ be a homomorphism. Then
        \begin{equation*}
            \begin{split}
                \bigg|\bigg\{\begin{array}{cc} [y] \in Y/\ker(\kappa_{\sigma}) \\ \psi(\mrm{Stab}_{\ker(\kappa_{\sigma})}(y)) = 1 \end{array}\bigg\}\bigg|
                & = \bigg|\bigg\{\begin{array}{cc} [y] \in Y/\Csigma \\ \psi(\mrm{Stab}_{\Csigma}(y)) = 1 \end{array}\bigg\}\bigg| 
                 + \bigg|\bigg\{\begin{array}{cc} [y] \in Y/\Csigma \\ \psi \equiv \kappa_{\sigma} \text{ on } \mrm{Stab}_{\Csigma}(y) \end{array}\bigg\}\bigg|.
            \end{split}
        \end{equation*}
    \end{lemma}

    \begin{proof}
        As $\overline{\sigma} \sim z\overline{\sigma}$, then $K \coloneqq \ker(\kappa_{\sigma}) \subseteq \Csigma$ is a normal subgroup of index $2$ by~\cref{lem:kappa}.
        Every $K$-orbit is contained in a unique $\Csigma$-orbit. 
        Let $[y] \in Y/\Csigma$. We consider two cases:
        \begin{enumerate}
            \item $\mrm{Stab}_{\Csigma}(y) \subseteq K$.\\
            Let $g \in \Csigma \setminus K$, so $\Csigma = K \sqcup Kg$.
            The $\Csigma$-orbit of $y$ is the union of the $K$-orbits of $y$ and of $gy$, which are different, as $g \notin K = K\mrm{Stab}_{\Csigma}(y)$. Their stabilizers are
            \begin{equation*}
                \mrm{Stab}_{K}(y) = \mrm{Stab}_{\Csigma}(y) \qquad \text{and} \qquad \mrm{Stab}_{K}(gy) = g\mrm{Stab}_{\Csigma}(y)g^{-1},
            \end{equation*}
            where the second equality uses that $g\mrm{Stab}_{\Csigma}(y)g^{-1} \subseteq gKg^{-1} = K$.
            Since $\{\pm 1\}$ is abelian, $\psi(\mrm{Stab}_{K}(gy)) = \psi(\mrm{Stab}_{K}(y))$. Hence the $\Csigma$-orbit of $y$ contains two $K$-orbits $[y']$ with $\psi(\mrm{Stab}_{K}(y')) = 1$ if $\psi(\mrm{Stab}_{\Csigma}(y)) = 1$, and none otherwise.

            \item $\mrm{Stab}_{\Csigma}(y) \not\subseteq K$.\\
            Let $h \in \mrm{Stab}_{\Csigma}(y)$ with $\kappa_{\sigma}(h) = -1$.
            As $\Csigma = K \sqcup Kh$, and $hy = y$, the $\Csigma$-orbit of $y$ is the single $K$-orbit of $y$. 
            Its stabilizer $\mrm{Stab}_{K}(y) = \mrm{Stab}_{\Csigma}(y) \cap K$ is the kernel of the restriction of $\kappa_{\sigma}$ to $\mrm{Stab}_{\Csigma}(y)$, which is surjective. Hence the homomorphisms $\mrm{Stab}_{\Csigma}(y) \to \{\pm 1\}$ which are trivial on $\mrm{Stab}_{K}(y)$ are exactly $1$ and $\kappa_{\sigma}$, and the $\Csigma$-orbit of $y$ contains one $K$-orbit $[y']$ with $\psi(\mrm{Stab}_{K}(y')) = 1$ if $\psi(\mrm{Stab}_{\Csigma}(y)) = 1$ or $\psi = \kappa_{\sigma}$ on $\mrm{Stab}_{\Csigma}(y)$, and none otherwise. At most one of these two conditions holds, as $\kappa_{\sigma}$ is non-trivial on $\mrm{Stab}_{\Csigma}(y)$.
        \end{enumerate}
        Now denote by $p$ the number of $[y] \in Y/\Csigma$ with $\mrm{Stab}_{\Csigma}(y) \subseteq K$ and $\psi(\mrm{Stab}_{\Csigma}(y)) = 1$, and by $q_1$ and $q_2$ the numbers of those $[y]$ with $\mrm{Stab}_{\Csigma}(y) \not\subseteq K$ and $\psi(\mrm{Stab}_{\Csigma}(y)) = 1$, respectively $\psi = \kappa_{\sigma}$ on $\mrm{Stab}_{\Csigma}(y)$. By the two cases, the left hand side is $2p + q_1 + q_2$. On the other hand, if $\mrm{Stab}_{\Csigma}(y) \subseteq K$, then $\kappa_{\sigma} = 1$ on $\mrm{Stab}_{\Csigma}(y)$, so that the condition $\psi = \kappa_{\sigma}$ on $\mrm{Stab}_{\Csigma}(y)$ is the same as $\psi(\mrm{Stab}_{\Csigma}(y)) = 1$. Hence the two sets on the right hand side have $p + q_1$ and $p + q_2$ elements, respectively, which proves the claim.
    \end{proof}

    \begin{proposition}\label{prop:dim-of-alternating-in-height-1-<0}
        Let $V\in (\ModEn[1])\dbl$ with $\dim(V) = d < 0$. Then
        \begin{equation*}
            \dim \alt_{\sgn^{(1)}} V = (-1)^{\degree} {-d \choose \degree} + \sum_{\substack{[\sigma] \in \Sm^{(2)} / \conj \\ \sgn(\sigma) = -1}} (-1)^{\numcyc} \bigsqcap_{k=1}^{\degree} {-d+\Nk-1 \choose \Nk}.
        \end{equation*}
        Equivalently, the formula of~\cref{prop:dim-of-alternating-in-height-1-ge0} holds for $d < 0$, where ${a \choose N} \coloneqq a(a-1)\cdots(a-N+1)/N!$ for $a \in \ZZ$.
    \end{proposition}

    \begin{proof}
        Again by~\cref{cor:dim-of-alternating-depends-only-on-dim}, with $n = -d$,
        \begin{equation*}
            \dim \alt_{\sgn^{(1)}} V = \dim \alt_{\sgn^{(1)}} \Sigma \En[1]^{\oplus n}.
        \end{equation*}
        We may therefore assume $V = \Sigma \En[1]^{\oplus n}$. Let $X \coloneqq (\bigvee_{n} S^1)^{\wedge \degree}$, so that $X$ is a genuine $\Sm^+$-space, where $\Sm^+$ acts through $\Sm$. For $G = \Sm^+, \Sm$ we denote by $\redEn[1][X_{hG}] \coloneqq \cofib(\En[1][\B G] \to \En[1][X_{hG}])$ the reduced homology, where the map is induced by the base point of $X$. By the proof of~\cref{lem:E1-of-smash-of-wedge-of-circles}
        \begin{equation*}
            \redEn[1][X_{h\Sm^+}] \simeq \Sym(\Sigma \En[1]^{\oplus n}) \oplus \alt_{\sgn^{(1)}}(\Sigma \En[1]^{\oplus n}).
        \end{equation*}
        By~\cref{thm:intertia-groupoid}, applied to $X$ and to its base point, and since finite groups have no rational homology,
        \begin{equation}\label{eq:dim-red-Sm+}\tag{$\clubsuit$}
            \dim(\redEn[1][X_{h\Sm^+}]) = \sum_{[\overline{\sigma}]\in (\Sm^+)^{(2)} / \conj} \chi\big(\tilde{\H}_*(X^{\langle \sigma \rangle}; \QQ)_{\cent_{\Sm^+}(\overline{\sigma})}\big),
        \end{equation}
        where $\chi$ is the Euler characteristic. As in the proof of~\cref{prop:dim-of-alternating-in-height-1-ge0}, $\cent_{\Sm^+}(\overline{\sigma})$ acts through its image $\ker(\kappa_{\sigma})$ in $\Csigma$, which is of index $2$ if $\overline{\sigma} \sim z\overline{\sigma}$, and all of $\Csigma$ if $\overline{\sigma} \not\sim z\overline{\sigma}$.

        To describe the summands, let $\sigma \in \Sm^{(2)}$. Note that non-equivariantly $X \simeq \bigvee_{[n]^{\degree}} S^{\degree}$. For every non-trivial orbit of $\sigma$ in $[n]^{\degree}$, $\langle \sigma \rangle$ permutes the corresponding spheres, and therefore the corresponding fixed points consist only of the base point. If $\sigma$ fixes $x \in [n]^{\degree}$, then the corresponding fixed points form a sphere $S^{\numcyc}_x = \bigwedge_c S^1$, with one smash factor for each cycle $c$ of $\sigma$. Therefore
        \begin{equation*}
            X^{\langle \sigma \rangle} = \bigvee_{x \in ([n]^{\degree})^{\langle \sigma \rangle}} S^{\numcyc}_x.
        \end{equation*}
        For $\tau \in \Csigma$, let $\tau_k \in \Sm[\Nk]$ be the permutation it induces on the $k$-cycles of $\sigma$, and let $\psi(\tau) \coloneqq \bigsqcap_k \sgn(\tau_k)$, which defines a homomorphism $\psi \colon \Csigma \to \{\pm 1\}$. An element $\tau \in \Csigma$ maps $S^{\numcyc}_x$ to $S^{\numcyc}_{\tau x}$. If $\tau x = x$, then $\tau$ acts on $S^{\numcyc}_x$ by permuting the smash factors as it permutes the cycles of $\sigma$, hence by a map of degree $\psi(\tau)$, and on homology by $\psi(\tau)$. It follows that for $C = \Csigma$, and for $C = \ker(\kappa_{\sigma})$, through which $\cent_{\Sm^+}(\overline{\sigma})$ acts, the coinvariants $\tilde{\H}_*(X^{\langle \sigma \rangle}; \QQ)_{C}$ consist of one copy of $\QQ\langle \numcyc \rangle$ for each $[x] \in ([n]^{\degree})^{\langle \sigma \rangle}/C$ with $\psi(\mrm{Stab}_C(x)) = 1$, so that
        \begin{equation*}
            \chi\big(\tilde{\H}_*(X^{\langle \sigma \rangle}; \QQ)_{C}\big) = (-1)^{\numcyc} \big|\big\{ [x] \in ([n]^{\degree})^{\langle \sigma \rangle} / C \;\big|\; \psi(\mrm{Stab}_C(x)) = 1 \big\}\big|.
        \end{equation*}

        To compare with $\Sm$, let $[\sigma] \in \Sm^{(2)}/\conj$. The classes in $(\Sm^+)^{(2)}/\conj$ mapping to $[\sigma]$ are $[\overline{\sigma}]$ and $[z\overline{\sigma}]$, and by the formula above with $C = \ker(\kappa_{\sigma})$, the term of each of them in \labelcref{eq:dim-red-Sm+} is $(-1)^{\numcyc}$ times the number of $[x] \in ([n]^{\degree})^{\langle \sigma \rangle}/\ker(\kappa_{\sigma})$ with $\psi(\mrm{Stab}_{\ker(\kappa_{\sigma})}(x)) = 1$. If $\overline{\sigma} \sim z\overline{\sigma}$, this is a single class, and by~\cref{lem:orbit-count} its term is $(-1)^{\numcyc}$ times the sum of the numbers of $[x] \in ([n]^{\degree})^{\langle \sigma \rangle}/\Csigma$ with $\psi(\mrm{Stab}_{\Csigma}(x)) = 1$ and with $\psi = \kappa_{\sigma}$ on $\mrm{Stab}_{\Csigma}(x)$. If $\overline{\sigma} \not\sim z\overline{\sigma}$, these are two different classes, and $\kappa_{\sigma} = 1$ by~\cref{lem:kappa}, so that $\ker(\kappa_{\sigma}) = \Csigma$ and the two conditions above coincide. Therefore
        \begin{equation*}
            \begin{split}
                \dim(\redEn[1][X_{h\Sm^+}]) = \sum_{[\sigma] \in \Sm^{(2)}/\conj} (-1)^{\numcyc} \bigg( & \bigg|\bigg\{\begin{array}{cc} [x] \in ([n]^{\degree})^{\langle \sigma \rangle} / \Csigma \\ \psi(\mrm{Stab}_{\Csigma}(x)) = 1 \end{array}\bigg\}\bigg| \\
                & + \bigg|\bigg\{\begin{array}{cc} [x] \in ([n]^{\degree})^{\langle \sigma \rangle} / \Csigma \\ \psi = \kappa_{\sigma} \text{ on } \mrm{Stab}_{\Csigma}(x) \end{array}\bigg\}\bigg|\bigg).
            \end{split}
        \end{equation*}
        Similarly, by~\cref{thm:intertia-groupoid} and the proof of~\cref{lem:E1-of-smash-of-wedge-of-circles},
        \begin{equation*}
            \dim \Sym (\Sigma \En[1]^{\oplus n}) = \dim(\redEn[1][X_{h\Sm}]) = \sum_{[\sigma] \in \Sm^{(2)}/\conj} (-1)^{\numcyc} \bigg|\bigg\{\begin{array}{cc} [x] \in ([n]^{\degree})^{\langle \sigma \rangle} / \Csigma \\ \psi(\mrm{Stab}_{\Csigma}(x)) = 1 \end{array}\bigg\}\bigg|,
        \end{equation*}
        and therefore
        \begin{equation*}
            \dim \alt_{\sgn^{(1)}} V = \sum_{[\sigma] \in \Sm^{(2)}/\conj} (-1)^{\numcyc} \bigg|\bigg\{\begin{array}{cc} [x] \in ([n]^{\degree})^{\langle \sigma \rangle} / \Csigma \\ \psi = \kappa_{\sigma} \text{ on } \mrm{Stab}_{\Csigma}(x) \end{array}\bigg\}\bigg|.
        \end{equation*}

        It remains to compute the summands. As in the proof of~\cref{prop:dim-of-alternating-in-height-1-ge0}, we identify $([n]^{\degree})^{\langle \sigma \rangle}$ with $\bigsqcap_k [n]^{\Nk}$ using the values $x_c$, and $\mrm{Stab}_{\Csigma}(x)$ is generated by the cycles of $\sigma$ and by the elements $w$ of~\cref{lem:kappa} exchanging two $k$-cycles $c, c'$ of $\sigma$ with $x_c = x_{c'}$. A cycle of $\sigma$ maps every cycle of $\sigma$ to itself, and such a $w$ exchanges $c$ and $c'$ and fixes the other cycles of $\sigma$, so that $\psi = 1$ on the cycles of $\sigma$ and $\psi(w) = -1$. As $\psi$ and $\kappa_{\sigma}$ are homomorphisms, it is enough to compare them on these generators, using the values of $\kappa_{\sigma}$ given in~\cref{lem:kappa}:
        \begin{enumerate}
            \item If $\sigma = e$, then $\kappa_{\sigma} = 1$, so $[x]$ is counted if and only if no such $w$ lies in $\mrm{Stab}_{\Sm}(x)$, that is, if and only if the coordinates of $x \in [n]^{\degree}$ are pairwise distinct. The orbit of such an $x$ is determined by its set of coordinates, a subset of $[n]$ of size $\degree$, so the summand is $(-1)^{\degree}{n \choose \degree}$.
            \item If $\sigma \neq e$ is even, then, as the cycle lengths of $\sigma$ are powers of $2$, $\sigma$ has a cycle $c$ of even length, for which $\psi(c) = 1 \neq -1 = \kappa_{\sigma}(c)$. Hence the summand vanishes.
            \item If $\sigma$ is odd, then by~\cref{lem:kappa}, $\kappa_{\sigma}(c) = 1$ for every cycle $c$ of $\sigma$, and $\kappa_{\sigma}(w) = -1$ for every $w$ as in~\cref{lem:kappa}(\labelcref{item:kappa-swap}), so $\psi = \kappa_{\sigma}$ on $\mrm{Stab}_{\Csigma}(x)$ for every $x$. Hence every $[x]$ is counted, and as $\Csigma$ acts on $\bigsqcap_k [n]^{\Nk}$ through $\bigsqcap_k \Sm[\Nk]$, the summand is $(-1)^{\numcyc} \bigsqcap_k |[n]^{\Nk}/\Sm[\Nk]| = (-1)^{\numcyc} \bigsqcap_k {n+\Nk-1 \choose \Nk}$.
        \end{enumerate}
        The equivalent form follows from $(-1)^{\degree}{n \choose \degree} = {d+\degree-1 \choose \degree}$ and $(-1)^{N}{n+N-1 \choose N} = {d \choose N}$, as $\sum_k \Nk = \numcyc$.
    \end{proof}
    
    \begin{notation}
        For $n \ge 0$ we define:
        \begin{enumerate}
            \item $\mcal{P}_n(\degree)$ to be the set of $n$-tuples $(\lambda^1, \dots, \lambda^n)$ of (possibly empty) partitions into powers of $2$ with $|\lambda^1| + \cdots + |\lambda^n| = \degree$, which have an odd number of even parts in total;
            \item $\mcal{Q}_n(\degree) \subseteq \mcal{P}_n(\degree)$ to be the subset of tuples of partitions into distinct powers of $2$.
        \end{enumerate}
    \end{notation}

    \begin{corollary}\label{cor:dim-height-1->=0}
        Let $V\in (\ModEn[1])\dbl$ with $\dim(V) = d \ge 0$ and let $\degree \ge 4$. Then
        \begin{equation*}
            \dim \alt_{\sgn^{(1)}} V = {d+\degree-1 \choose \degree} + |\mcal{Q}_d(\degree)|.
        \end{equation*}
        In particular, if $d = 1$ and $\degree = \sum_i b_i 2^i$ is the binary decomposition, then $\dim \alt_{\sgn^{(1)}} V = 2$ if $|\{i > 0 \mid b_i = 1\}|$ is odd, and $\dim \alt_{\sgn^{(1)}} V = 1$ otherwise.
    \end{corollary}

    \begin{proof}
        Let $[\sigma] \in \Sm^{(2)}/\conj$. As in the proof of~\cref{prop:dim-of-alternating-in-height-1-ge0}, the product $\bigsqcap_k {d \choose \Nk}$ is the number of families $(S_k)_k$ of subsets $S_k \subseteq [d]$ with $|S_k| = \Nk$. Such a family determines the $d$-tuple $(\lambda^1, \dots, \lambda^d)$, where the parts of $\lambda^l$ are the $k$ with $l \in S_k$. As $\sigma$ is $2$-power torsion, this is a tuple of partitions into distinct powers of $2$, of total size $\sum_k k\Nk = \degree$. Conversely, such a tuple determines the subsets $S_k = \{l \mid k \text{ is a part of } \lambda^l\}$ and the class $[\sigma]$ with $\Nk = |S_k|$. Under this correspondence, the even parts of the tuple correspond to the cycles of $\sigma$ of even length, so $\sigma$ is odd if and only if the tuple has an odd number of even parts. Hence the sum over the odd classes in~\cref{prop:dim-of-alternating-in-height-1-ge0} is $|\mcal{Q}_d(\degree)|$. For $d = 1$, the only partition of $\degree$ into distinct powers of $2$ is its binary decomposition, whose even parts are the $2^i$ with $i > 0$ and $b_i = 1$.
    \end{proof}

    \begin{corollary}\label{cor:dim-height-1-<0}
        Let $V\in (\ModEn[1])\dbl$ with $\dim(V) = d < 0$ and let $\degree \ge 4$. Then
        \begin{equation*}
            \dim \alt_{\sgn^{(1)}} V = (-1)^{\degree}\Big({-d \choose \degree} - |\mcal{P}_{-d}(\degree)|\Big).
        \end{equation*}
        In particular, if $d = -1$, then $\dim \alt_{\sgn^{(1)}} V = \frac{(-1)^{\degree+1}}{2}|\Sm^{(2)}/\conj|$.
    \end{corollary}

    \begin{proof}
        Let $n = -d$. As in the proof of~\cref{prop:dim-of-alternating-in-height-1-<0}, for $[\sigma] \in \Sm^{(2)}/\conj$ the number of $\Csigma$-orbits on $[n]^{\numcyc} = \bigsqcap_k [n]^{\Nk}$ is $\bigsqcap_k {n+\Nk-1 \choose \Nk}$, and the orbit of $x = (x_k)_k$ is determined by the numbers $n_{k,l}$ of coordinates of $x_k$ equal to $l$, for $l \in [n]$. These numbers determine the $n$-tuple $(\lambda^1, \dots, \lambda^n)$, where $\lambda^l$ has $n_{k,l}$ parts equal to $k$, which is a tuple of partitions into powers of $2$ of total size $\degree$. Conversely, such a tuple determines the numbers $n_{k,l}$ and the class $[\sigma]$ with $\Nk = \sum_l n_{k,l}$. As in the proof of~\cref{cor:dim-height-1->=0}, $\sigma$ is odd if and only if the tuple has an odd number of even parts, and in this case $(-1)^{\numcyc} = -(-1)^{\degree}$. Hence the claim follows from~\cref{prop:dim-of-alternating-in-height-1-<0}.

        Finally, if $d = -1$, then ${1 \choose \degree} = 0$, and $\mcal{P}_1(\degree)$ corresponds to the set of odd classes in $\Sm^{(2)}/\conj$. By the uniqueness of binary decompositions, $\bigsqcap_{j \ge 1}(1+t^{2^j}) = (1-t^2)^{-1}$, so that
        \begin{equation*}
            \sum_{\degree \ge 0} \sum_{[\sigma] \in \Sm^{(2)}/\conj} \sgn(\sigma)\, t^{\degree} = (1-t)^{-1}\bigsqcap_{j \ge 1}\big(1+t^{2^j}\big)^{-1} = 1+t.
        \end{equation*}
        Hence for $\degree \ge 2$ exactly half of the classes in $\Sm^{(2)}/\conj$ are odd.
    \end{proof}

    In the categorical case $\cC = \Mod_{\sVect}$, similar arguments show that

    \begin{proposition}\label{prop:categorical-height-1}
        Let $d\in \ZZ_{\ge 0}$. Then
        \begin{equation*}
            \dim^2(\alt_{\sgn^{(1)}} (\sVect^{\oplus d})) = \sum_{[\sigma]\in \mcal{O}(\degree)} \bigsqcap_{k=1}^{\degree} {d+\Nk-1 \choose \Nk} + \sum_{\substack{[\sigma]\in \Sm / \conj \\ \sgn(\sigma) = -1}} \bigsqcap_{k=1}^{\degree} {d \choose \Nk},
        \end{equation*}
        where $\mcal{O}(\degree)$ is the set of conjugacy classes in $\Sm$ such that their decomposition into disjoint cycles contains no even cycles. Equivalently, $\dim^2(\alt_{\sgn^{(1)}} (\sVect^{\oplus d}))$ is the number of $d$-tuples of partitions into distinct parts of total size $\degree$, where a tuple is counted twice if it has an odd number of even parts in total.
    \end{proposition}

    \begin{remark}
        For $d = 1$,~\cref{prop:categorical-height-1} recovers Schur's count of the irreducible projective representations of $\Sm$ corresponding to the non-trivial class in $\H^2(\Sm; \QQ/\ZZ)$ (\cite{Schur-1911-alternating-groups}, see also~\cite{Bessenrodt-1994-covering-of-symmetric-groups}): they are indexed by the partitions of $\degree$ into distinct parts, where a partition corresponds to two representations if it has an odd number of even parts, and to one otherwise.
    \end{remark}
    
    \begin{remark}
        Restricting the sums in~\cref{prop:categorical-height-1} to the classes of $2$-power torsion elements, the only ones contributing in the chromatic case by~\cref{thm:intertia-groupoid}, gives the formula of~\cref{prop:dim-of-alternating-in-height-1-ge0}, since $[e]$ is the only such class without even cycles. Similarly,~\cref{cor:dim-height-1->=0} is the count of~\cref{prop:categorical-height-1} with only powers of $2$ allowed as parts.
    \end{remark}
    
    \begin{remark}
        From these computations, one can deduce the values of $\sgn^{(1)}$-twisted power operations evaluated on integers (in $\En[2]$ in the chromatic case, or in $\cVectn[2]$ in the categorical case).
    \end{remark}

\newpage

%%%%%%%%%%%%%%%%%%%%%%%%%%%%%%%%%%%%%%%%%%%%%%%%%%%%%%%%%%%%%%%%%%%%%%%%%%%%%%%%
%%%%%%%%%%%%%%%%%%%%%%%%%%%%%%%%%%%%%%%%%%%%%%%%%%%%%%%%%%%%%%%%%%%%%%%%%%%%%%%%
\begin{appendices}
%%%%%%%%%%%%%%%%%%%%%%%%%%%%%%%%%%%%%%%%%%%%%%%%%%%%%%%%%%%%%%%%%%%%%%%%%%%%%%%%
%%%%%%%%%%%%%%%%%%%%%%%%%%%%%%%%%%%%%%%%%%%%%%%%%%%%%%%%%%%%%%%%%%%%%%%%%%%%%%%%

% Old appendix names (Atheorem, Alemma, ...) -> the normal environments
\makeatletter
\@for\@tempa:=theorem,definition,notation,convention,terminology,%
  construction,recollection,solution,conjecture,remark,example,exercise,%
  question,proposition,lemma,claim,corollary,observation,fact,note\do{%
  \edef\@tempb{\noexpand\NewEnvironmentCopy{A\@tempa}{\@tempa}}\@tempb}
\makeatother

%%%%%%%%%%%%%%%%%%%%%%%%%%%%%%%%%%%%%%%%%%%%%%%%%%%%%%%%%%%%%%%%%%%%%%%%%%%%%%%%
%%%%%%%%%%%%%%%%%%%%%%%%%%%%%%%%%%%%%%%%%%%%%%%%%%%%%%%%%%%%%%%%%%%%%%%%%%%%%%%%
\section{Conjugacy classes of wreath products via the monoidal dimension}
\label{app:conjugacy-of-wreath}
%%%%%%%%%%%%%%%%%%%%%%%%%%%%%%%%%%%%%%%%%%%%%%%%%%%%%%%%%%%%%%%%%%%%%%%%%%%%%%%%
%%%%%%%%%%%%%%%%%%%%%%%%%%%%%%%%%%%%%%%%%%%%%%%%%%%%%%%%%%%%%%%%%%%%%%%%%%%%%%%%

    Our goal in this appendix is to obtain a formula for centralizers and conjugacy classes in the wreath product $G \wr \Sm$ in terms of those of $G$ and $\Sm$. We do so by working in the universal 1-semiadditive category $\Span(\spcpi[1])$, where we compute the dimension of the symmetric power of $\BG$ in two different ways.
    
    % \begin{Atheorem}[{\cite[Corollary 5.3]{Harpaz-2020-ambi}}]
    %     Any symmetric monoidal $1$-semiadditive category $\cC$ admits a symmetric monoidal, $1$-finite colimit-preserving functor
    %     \begin{equation*}
    %         \ounit_{\cC}[-] \colon \Span(\spcpi[1]) \to \cC.
    %     \end{equation*}
    % \end{Atheorem}

    First, recall that in $\Span(\spcpi[1])$
    \begin{equation*}
        \dim(\Sym(\BG)) \simeq \L\Sym \BG \simeq \L (\BG^{\degree})_{h\Sm} \simeq \L\B(G\wr\Sm) \simeq \bigsqcup_{[(\underline{g}, \sigma)] \in G \wr \Sm /\conj} \B\cent_{(\underline{g},\sigma)}(G \wr \Sm).
    \end{equation*}
    
    Second, we use the induced character formula~\cref{thm:induced-character-formula}
    \begin{equation*}
        \dim (\Sym (\BG)) \simeq \int_{\L\B\Sm} \chi_{\Tm(\BG)}.
    \end{equation*}
    
    \begin{Alemma}
        Let $\sigma\in \Sm$. Then, using the notations of~\cref{not:sigma} 
        \begin{equation*}
            \Csigma\simeq  ( \Sm[{\Nk[1]}])\times ((\ZZ/2)^{\Nk[2]}\rtimes \Sm[{\Nk[2]}]) \times \cdots \times ((\ZZ/\degree)^{\Nk[\degree]}\rtimes \Sm[{\Nk[\degree]}]).
        \end{equation*}
        moreover each $\ZZ/k$ is generated by the corresponding cycle and $\Sm[\Nk]$ permutes all the cycles of length $k$.
    \end{Alemma}
    
    % In order to calculate the integrals we want to understand the action of the stabilizers on the dimensions.
    
    \begin{Alemma}[{\cref{lem:character-of-Tm}}]
       Let $\cC\in \CAlg(\PrL)$ and $V\in \cC\dbl$. Then for any $\sigma \in \L\B\Sm$
       \begin{equation*}
            \chi_{\Tm V}(\sigma)=\dim(V)^{\numcyc} \qin \ounit_{\cC}^{\B\Csigma}
       \end{equation*}
       where $\dim(V)^{\numcyc}$ is equipped with the $\Csigma$-action of~\cref{def:stabilizer-action}.
    \end{Alemma}
    
    In the case $\cC = \Span(\spcpi[1])$ and $V = \BG$:
    \begin{equation*}
        \begin{split}
            \L\B(G \wr \Sm) & = \dim(\Sym (\BG)) \\
            & \simeq \int_{\L\B\Sm} \chi_{\Tm\BG} \\
            & \simeq \bigsqcup_{[\sigma] \in \Sm/\conj} \int_{\B\Csigma} \L\BG^{\numcyc} \\
            & \simeq \bigsqcup_{[\sigma] \in \Sm/\conj} (\L\BG^{\numcyc})_{h \Csigma} \\
            & \simeq \bigsqcup_{[\sigma] \in \Sm/\conj}\bigsqcap_{k=1}^\degree (\L\BG^{\Nk})_{h ((\ZZ/k)^{\Nk} \rtimes \Sm[\Nk])}.
        \end{split}
    \end{equation*}
    $((\ZZ/k)^{\Nk} \rtimes \Sm[\Nk])$-orbits can be computed by first taking $(\ZZ/k)^{\Nk}$-orbits and then computing orbits under the induced $\Sm[\Nk]$-action. Note that
    \begin{equation*}
        (\L\BG^{\Nk})_{h (\ZZ/k)^{\Nk}} \simeq ((\L\BG)_{h\ZZ/k})^{\Nk} \qin \spc^{\B\Sm[\Nk]}.
    \end{equation*}
    \begin{Acorollary}\label{cor:LB(G-wr-Sm)}
        Let $G$ be a finite group. Then
        \begin{equation*}
            \L\B(G \wr \Sm) \simeq \bigsqcup_{[\sigma] \in \Sm/\conj} \bigsqcap_{k=1}^\degree (((\L\BG)_{h\ZZ/k})^{\Nk})_{h \Sm[\Nk]}.
        \end{equation*}
    \end{Acorollary}
    
    The $\ZZ/k$-action on $\L\BG$ is the restriction of the $\TT$-action on $\L\BG = \Map(\TT,\BG)$. Recall that a $\TT$-action on $\BG$ is the data of an element in the center $\Z(G)$ (\cref{recol:T-action-on-BG}).
    In particular a $\TT$-action on $\L\BG \simeq \bigsqcup_{[g] \in G / \conj} \B\cent_g(G)$ is an element in $\Z(\cent_g(G))$ for every conjugacy class in $G$. 
    The $\TT$-action on $\L\BG$ is described by $g \in \Z(\cent_g(G))$ for any conjugacy class $[g]$ of $G$.

    \begin{Adefinition}[\cref{def:root-of-central-element}]
            Let $G$ be a group, $z \in \Z(G)$ and $k \in \NN$. We define the group
            \begin{equation*}
                G\kz{z} \coloneqq G \langle x \mid x^k = z, \ gx = xg\  \forall g \in G \rangle.
            \end{equation*}
        \end{Adefinition}
        
    \begin{Alemma}\label{lem:cyclic-fixed-points-of-BG}
        Let $G$ be a finite group and $z \in \Z(G)$ that induces a $\TT$-action on $\BG$.
        Then
        \begin{equation*}
            (\BG)_{h\ZZ/k} \simeq \B G\kz{z}.
        \end{equation*}
    \end{Alemma}
    
    \begin{proof}
        The case where $G$ is an abelian group is a reinterpretation of a classical fact. Indeed in this case we want to show that the following 
        \begin{equation*}
            \begin{tikzcd}
                {\BG\kz{z}} & \pt \\
                {\B\ZZ/k} & {\B^2 G=\B^2 \Z(G)}
                \arrow[from=1-2, to=2-2]
                \arrow[from=2-1, to=2-2]
                \arrow[from=1-1, to=2-1]
                \arrow[from=1-1, to=1-2]
            \end{tikzcd}
        \end{equation*}
        is a pullback diagram. This is the claim that $G\kz{z}$ is the central extension of $\ZZ/k$ by $G$ classified by $[z]\in G/kG\simeq \H^2(\ZZ/k; G)$ which is classical.
        
        For the general case one needs to verify that the following
        \begin{equation*}
            \begin{tikzcd}
                {\BG\kz{z}} & {\BG/\Z(G)} \\
                {\B\ZZ/k} & {\B^2\Z(G)}
                \arrow[from=1-2, to=2-2]
                \arrow[from=2-1, to=2-2]
                \arrow[from=1-1, to=2-1]
                \arrow[from=1-1, to=1-2]
            \end{tikzcd}
        \end{equation*}
        is a pullback diagram. 
        Let $P \coloneqq \B G/\Z(G) \times_{\B^2 \Z(G)} \B\ZZ/k$ and consider the following diagram
        \begin{equation*}
            \begin{tikzcd}
                {\B\Z(G)} & {\B\Z(G)\kz{z}} & \pt \\
                \BG & P & {\B G/Z(G)} \\
                \pt & {\B\ZZ/k} & {\B^2\Z(G).}
                \arrow[from=1-1, to=1-2]
                \arrow[from=1-1, to=2-1]
                \arrow["\lrcorner"{anchor=center, pos=0.125}, draw=none, from=1-1, to=2-2]
                \arrow[from=1-2, to=1-3]
                \arrow[from=1-2, to=2-2]
                \arrow["\lrcorner"{anchor=center, pos=0.125}, draw=none, from=1-2, to=3-3]
                \arrow[from=1-3, to=2-3]
                \arrow[from=2-1, to=2-2]
                \arrow[from=2-1, to=3-1]
                \arrow["\lrcorner"{anchor=center, pos=0.125}, draw=none, from=2-1, to=3-2]
                \arrow[from=2-2, to=2-3]
                \arrow[from=2-2, to=3-2]
                \arrow["\lrcorner"{anchor=center, pos=0.125}, draw=none, from=2-2, to=3-3]
                \arrow[from=2-3, to=3-3]
                \arrow[from=3-1, to=3-2]
                \arrow[from=3-2, to=3-3]
            \end{tikzcd}
        \end{equation*}
        Then the bottom-right square is a cartesian by definition, the right rectangle is cartesian as $\Z(G)$ is abelian, and the bottom rectangle and the full square are easily seen to be cartesian. Therefore all squares in the diagram are pullback squares. In particular, $P = \B K$ where $K$ is an extension of $\ZZ/k$ by $G$ and an extension of $G/\Z(G)$ by $\Z(G)\kz{z}$. 
        Choosing a base point for $\B\Z(G)$ we get a commutative diagram of groups where the middle row and middle column are exact
        \begin{equation*}
            \begin{tikzcd}
                && 1 \\
                & {\Z(G)} & {\Z(G)\kz{z}} \\
                1 & G & K & \ZZ/k & 1. \\
                && {G/\Z(G)} \\
                && 1
                \arrow[from=1-3, to=2-3]
                \arrow[from=2-2, to=2-3]
                \arrow[from=2-2, to=3-2]
                \arrow[from=2-3, to=3-3]
                \arrow[from=3-1, to=3-2]
                \arrow[from=3-2, to=3-3]
                \arrow[two heads, from=3-2, to=4-3]
                \arrow[from=3-3, to=3-4]
                \arrow[from=3-3, to=4-3]
                \arrow[from=3-4, to=3-5]
                \arrow[from=4-3, to=5-3]
            \end{tikzcd}
        \end{equation*}
        Let $x \in K$ be the image of $\sqrt[k]{z} \in \Z(G)\kz{z}$. Then $x^k = z$ and as the image of $x$ is trivial in $G/\Z(G)$, $x$ commutes with all elements of $G$, i.e.\ $K = G\kz{z}$.
    \end{proof}
    
    Finally, using the notations of~\cref{not:not1,not:not2}, we deduce:
    \begin{Acorollary}
        Let $G$ be a finite group. Then 
            \begin{equation*}
                \L\B(G\wr \Sm) 
                \simeq \bigsqcup_{[\sigma] \in \Sm/\conj} \ \bigsqcap_{k=1}^\degree \ \bigsqcup_{[\underline{g}] \in \Sm[\Nk] \backslash G^{\Nk} / \conj} \ \bigsqcap_{x \in \underline{g}} \B (\cent_G(x)\kz{x} \wr \Sm[\mul_{\underline{g}}(x)]).
            \end{equation*}
    \end{Acorollary}
    
    \begin{proof}
        By~\cref{cor:LB(G-wr-Sm)} and~\cref{lem:cyclic-fixed-points-of-BG} 
        \begin{equation*}
            \begin{split}
                \L\B(G \wr \Sm) 
                & \simeq 
                    \bigsqcup_{[\sigma] \in \Sm/\conj} \ 
                    \bigsqcap_{k=1}^\degree ((\bigsqcup_{[x] \in G/\conj} \B\cent_G(x)\kz{x})^{\Nk})_{h \Sm[\Nk]} \\
                & \simeq 
                    \bigsqcup_{[\sigma] \in \Sm/\conj} \ 
                    \bigsqcap_{k=1}^\degree \ 
                    \bigsqcup_{[\underline{g}] \in \Sm[\Nk] \backslash G^{\Nk}/\conj} \ 
                    \bigsqcap_{x\in \underline{g}}(\B\cent_G(x)\kz{x}^{\mul_{\underline{g}}(x)})_{h \Sm[\mul_{\underline{g}(x)}]} \\
                & \simeq 
                    \bigsqcup_{[\sigma] \in \Sm/\conj} \
                    \bigsqcap_{k=1}^\degree \ 
                    \bigsqcup_{[\underline{g}] \in \Sm[\Nk] \backslash G^{\Nk} / \conj} \ 
                    \bigsqcap_{x \in \underline{g}} \B (\cent_G(x)\kz{x} \wr \Sm[\mul_{\underline{g}}(x)]).
            \end{split}
        \end{equation*}
    \end{proof}

    \bibliographystyle{alpha}
    \phantomsection\addcontentsline{toc}{section}{\refname}
    \bibliography{references}

\end{appendices}

\end{document}